\documentclass[12pt, reqno]{amsart}

\usepackage[usenames,dvipsnames]{pstricks} 
\usepackage{epsfig}
\usepackage{graphicx,color}
\usepackage{geometry}
\usepackage[all]{xy}
\usepackage{amssymb,amscd}
\usepackage{cite}
\usepackage{fullpage}
\usepackage{marvosym}
\xyoption{poly}
\usepackage{url}
\usepackage{comment}
\usepackage{float}
\usepackage{bm}
\usepackage{marvosym}

\usepackage{tikz}
\usepackage{tikz-cd}
\usetikzlibrary{decorations.pathmorphing}
\newtheorem{introtheorem}{Theorem}

\newtheorem{theorem}{Theorem}[section]
\newtheorem{lemma}[theorem]{Lemma}
\newtheorem{proposition}[theorem]{Proposition}
\newtheorem{corollary}[theorem]{Corollary}
\theoremstyle{definition}
\newtheorem{definition}[theorem]{Definition}

\newtheorem{remark}[theorem]{Remark}

\newtheorem{example}[theorem]{Example}

\newtheorem*{question*}{Question}

\newtheorem*{questions*}{Questions}

\newtheorem*{steps*}{Answer/steps}

\newtheorem*{progress*}{Progress}

\newtheorem*{classification*}{Classification}

\newtheorem*{construction*}{Classification}
\newtheorem*{example*}{Example}

\newtheorem*{remark*}{Remark}
\newtheorem*{remarks*}{Remarks}
\newtheorem*{definition*}{Definition}

\usepackage{calrsfs}
\usepackage{url}
\usepackage{longtable}
\usepackage[OT2, T1]{fontenc}
\usepackage{textcomp}
\usepackage{times}
\usepackage[scaled=0.92]{helvet}

\renewcommand{\tilde}{\widetilde}

\newcommand{\C}{\mathbb{C}}
\newcommand{\Q}{\mathbb{Q}}

\newcommand{\R}{\mathbb{R}}
\newcommand{\Z}{\mathbb{Z}}

\newcommand{\F}{\mathbb{F}}

\newcommand{\X}{\mathcal{X}}

\DeclareMathOperator{\GL}{GL}

\DeclareMathOperator{\Aut}{Aut}

\DeclareMathOperator{\Sp}{S}
\DeclareMathOperator{\Gal}{Gal}

\DeclareSymbolFont{cyrletters}{OT2}{wncyr}{m}{n}
\DeclareMathSymbol{\Sha}{\mathalpha}{cyrletters}{"58}

\makeatletter

\def\greekbolds#1{%
 \@for\next:=#1\do{%
    \def\X##1;{%
     \expandafter\def\csname V##1\endcsname{\boldsymbol{\csname##1\endcsname}}
     }
   \expandafter\X\next;
  }
}

\greekbolds{alpha,beta,iota,gamma,lambda,nu,eta,Gamma,varsigma,Lambda}

\def\make@bb#1{\expandafter\def
  \csname bb#1\endcsname{{\mathbb{#1}}}\ignorespaces}

\def\make@bbm#1{\expandafter\def
  \csname bb#1\endcsname{{\mathbbm{#1}}}\ignorespaces}

\def\make@bf#1{\expandafter\def\csname bf#1\endcsname{{\bf
      #1}}\ignorespaces} 

\def\make@gr#1{\expandafter\def
  \csname gr#1\endcsname{{\mathfrak{#1}}}\ignorespaces}

\def\make@scr#1{\expandafter\def
  \csname scr#1\endcsname{{\mathscr{#1}}}\ignorespaces}

\def\make@cal#1{\expandafter\def\csname cal#1\endcsname{{\mathcal
      #1}}\ignorespaces} 

\def\do@Letters#1{#1A #1B #1C #1D #1E #1F #1G #1H #1I #1J #1K #1L #1M
                 #1N #1O #1P #1Q #1R #1S #1T #1U #1V #1W #1X #1Y #1Z}
\def\do@letters#1{#1a #1b #1c #1d #1e #1f #1g #1h #1i #1j #1k #1l #1m
                 #1n #1o #1p #1q #1r #1s #1t #1u #1v #1w #1x #1y #1z}
\do@Letters\make@bb   \do@letters\make@bbm
\do@Letters\make@cal  
\do@Letters\make@scr 
\do@Letters\make@bf \do@letters\make@bf   
\do@Letters\make@gr   \do@letters\make@gr
\makeatother

\def\ol{\overline}
\def\wt{\widetilde}

\def\ul{\underline}

\def\wh{\widehat}

\newcommand{\<}{\langle}   
\renewcommand{\>}{\rangle} 

\newcommand{\isoto}{\stackrel{\sim}{\longrightarrow}}
\newcommand{\embed}{\hookrightarrow}

\def\Fpbar{\overline{\bbF}_p}
\def\Fp{{\bbF}_p}
\def\Fq{{\bbF}_q}

\def\Qp{{\bbQ}_p}

\def\Zp{{\bbZ}_p}
\def\Qbar{\overline{\bbQ}}
\def\ch{characteristic\ }
\def\Gm{{\bbG_m}} 
\newcommand{\A}{\mathbb A}    

\def\makeop#1{\expandafter\def\csname#1\endcsname
  {\mathop{\rm #1}\nolimits}\ignorespaces}
\makeop{Hom}   \makeop{End}   \makeop{Aut}   \makeop{Isom}  \makeop{Pic} 
\makeop{Gal}   \makeop{ord}   \makeop{Char}  \makeop{Div}   \makeop{Lie} 
\makeop{PGL}   \makeop{Corr}  \makeop{PSL}   \makeop{sgn}   \makeop{Spf}
\makeop{Spec}  \makeop{Tr}    \makeop{Nr}    \makeop{Fr}    \makeop{disc}
\makeop{Proj}  \makeop{supp}  \makeop{ker}   \makeop{im}    \makeop{dom}
\makeop{coker} \makeop{Stab}  \makeop{SO}    \makeop{SL}    \makeop{SL}
\makeop{Cl}    \makeop{cond}  \makeop{Br}    \makeop{inv}   \makeop{rank}
\makeop{id}    \makeop{Fil}   \makeop{Frac}  \makeop{GL}    \makeop{SU}
\makeop{Nrd}   \makeop{Sp}    \makeop{Tr}    \makeop{Trd}   \makeop{diag}
\makeop{Res}   \makeop{ind}   \makeop{depth} \makeop{Tr}    \makeop{st}
\makeop{Ad}    \makeop{Int}   \makeop{tr}    \makeop{Sym}   \makeop{can}
\makeop{length}\makeop{SO}    \makeop{torsion} \makeop{GSp} \makeop{Ker}
\makeop{Adm}   \makeop{Mat}   \makeop{trdeg}

\DeclareMathOperator{\Mass}{Mass}

\newcommand{\dieu}{Dieudonn\'{e} }

\DeclareMathSymbol{\twoheadrightarrow} {\mathrel}{AMSa}{"10}

\DeclareMathOperator{\pr}{pr}

\newcommand{\abs}[1]{\lvert #1 \rvert}

\def\sfF{\mathsf{F}}
\def\sfV{\mathsf{V}}

\def\Gr{\mathrm{Gr}}
\def\char{{\rm char\,}}

\begin{document}

\title{Supersingular Ekedahl-Oort strata and Oort's conjecture}

\author{Valentijn Karemaker}
\address{Korteweg-de Vries Institute for Mathematics, University of Amsterdam, The Netherlands}
\email{V.Z.Karemaker@uva.nl}

\author{Chia-Fu Yu}
\address{Institute of Mathematics, Academia  Sinica and National Center for Theoretic Sciences, Taipei, Taiwan}
\email{chiafu@math.sinica.edu.tw}


\setcounter{tocdepth}{2}

\begin{abstract}
  Let $\calA_g$ be the moduli space over $\overline{\mathbb{F}}_p$ of $g$-dimensional principally polarised abelian varieties, where $p$ is a prime.
We show that if $g$ is even and $p\ge 5$, then every geometric generic member in the maximal supersingular Ekedahl-Oort stratum in $\calA_g$ has automorphism group $\{ \pm 1\}$. This confirms Oort's conjecture in the case of $p\ge 5$ and even $g$. We also separately prove Oort's conjecture for $g=4$ and $p>2$. 
\end{abstract}

\maketitle

\section{Introduction}

Let $g\ge 1$ be a positive integer and $p$ a prime number. 
Let $\calA_g$ be the moduli space over $\Fpbar$ of $g$-dimensional principally polarised abelian varieties, and let $\calS_g$ be the supersingular locus of $\calA_g$. 
Let $k$ be an algebraically closed field of characteristic $p$. For each member $(X,\lambda)$ in $\calA_g(k)$, it is a fundamental question to understand what the endomorphism ring $\End(X)$ of $X$ and the automorphism group $\Aut(X,\lambda)$ of $(X,\lambda)$ may be. Moreover, it is also interesting to understand how these arithmetic invariants vary in the moduli space $\calA_g$ or in a subvariety, for example, in $\calS_g$. 

Chai and Oort~\cite{COirr} showed that for any prime $\ell\neq p$, the $\ell$-adic monodromy attached to any non-supersingular central leaf $\calC \subseteq \calA_g(k)$ is surjective and that $\calC$ is irreducible. It follows that the geometric generic member $(X_{\bar \eta},\lambda_{\bar \eta})$ of $\calC$ has endomorphism ring $\Z$ and hence automorphism group $\{\pm 1\}$. Using this, every geometric generic member $(X_{\bar \eta},\lambda_{\bar \eta})$ of  either a non-supersingular Newton stratum or a non-supersingular Ekedahl-Oort (EO) stratum (i.e., an Ekedahl-Oort stratum that is not entirely contained in $\calS_g$ ) also shares the same property. One may ask what one can say for the supersingular case. Oort's conjecture \cite[Question 4]{edixhoven-moonen-oort} states that when $g\ge 2$, every geometric generic member in $\calS_g$ has automorphism group $\{\pm1 \}$, even though its endomorphism ring is always of $\Z$-rank $4g^2$.
For $g=2$ and $p>2$, Oort's conjecture has been proved by Ibukiyama~\cite{ibukiyama}, and by the first author and Pries \cite{karemaker-pries} independently, with a counterexample in $p=2$.

In~\cite{karemaker-yobuko-yu} the present authors and Yobuko consider the mass function on the supersingular locus:
\[ \Mass:\calS_g \to \Q_{>0}, \quad x \mapsto \Mass(x)=\Mass(\calC(x)):=\sum_{(X,\lambda)\in \calC(x)} \frac{1}{\abs{\Aut(X,\lambda)}},\]
where $\calC(x)$ is the central leaf passing through the point $x$.
This function decomposes the supersingular locus into pieces of strata, called mass strata, which encode information on the jump of arithmetic invariants. 
Concrete descriptions of mass strata and explicit mass formulae on each stratum were obtained by the second author and J.D.~Yu~\cite{yuyu} for $g=2$, and by the authors of~\cite{karemaker-yobuko-yu} for $g=3$.
The latter authors also showed that on the stratum of maximal mass, each geometric point has automorphism group $\{\pm 1\}$ if $p>2$, and $\{\pm 1\}^3$ if $p=2$, confirming Oort's conjecture for $g=3$, with again a counterexample in $p=2$. 

Chai and Oort~\cite{COirr} also showed that every non-supersingular Newton  stratum is irreducible. Ekedahl and van der Geer~\cite{EvdG} showed that every non-supersingular EO stratum is also irreducible. 
For the supersingular case, it is known that the number of irreducible components is large as long as either $p$ or $g$ is large, as shown by Li and Oort \cite{lioort} for $\calS_g$, and by Harashita~\cite{harashita:SSEO} for supersingular EO strata. 
In \cite{ibukiyama-karemaker-yu} the present authors and Ibukiyama determine precisely when the supersingular locus $\calS_g$,  and  a supersingular central leaf are irreducible. For determining the irreducibility of the latter, the authors were led to explore automorphism groups that occur in supersingular EO strata in $\calA_4$.\\

The present work is a continuation of the authors' exploration into the variation of arithmetic invariants on supersingular EO strata. 
We prove the following results.

\begin{introtheorem}\label{thm:main}
If $g$ is even and $p\ge 5$, then every geometric generic member in the maximal supersingular EO stratum has automorphism group $\{\pm 1\}$.    
\end{introtheorem}

The assertion of Theorem~\ref{thm:main} fails if either $g$ is odd or $p=2$; see Remark~\ref{rem:goddp2}. 

\begin{introtheorem}\label{thm:oc}
    Oort's conjecture holds true for the case where $g$ is even and $p\ge 5$.
\end{introtheorem}

We explain how Theorem A implies Theorem B. First, we show that the set of irreducible components of $\calS_g$ is transitive under $\ell$-adic Hecke correspondences, where $\ell\neq p$ is a prime. It follows that every irreducible component $Y$ of the supersingular locus $\calS_g$ contains an irreducible component $W$ of the maximal supersingular EO stratum. Since the polarised abelian variety corresponding to the geometric generic point of $W$ has automorphism group $\{\pm 1\}$ when $g$ is even and $p\ge 5$, the same holds for the geometric generic point of $Y$. 

The main idea of the proof of Theorem A is as follows. We consider the locus $\calS_g^{\rm eo}\subseteq \calS_g$ which is the union of all supersingular EO strata. Each irreducible component of $\calS_g^{\rm eo}$ admits a finite cover which is a Lagrangian variety ${\mathrm L}$. We introduce a stratification on ${\rm L}$ by the relative endomorphism algebra of its isotropic subspaces. 
Using this, we construct a new stratification on 
$\calS_g^{\rm eo}$. 
It describes the jumps of the endomorphism rings of abelian varieties in $\calS_g^{\rm eo}$. There is a unique maximal stratum which is open and dense in $\calS_g^{\rm eo}$.
We further show that 
when $g$ is even and $p\ge 5$, every geometric point in the maximal stratum has automorphism group~$\{\pm 1\}$. It follows that 
the polarised abelian variety corresponding to each geometric generic point of~$\calS_g^{\rm eo}$ has automorphism group $\{\pm 1\}$.\\

In~\cite{dusan}, Du\v{s}an Dragutinovi{\'c} obtained a different proof of Oort's conjecture when $g=4$ and $p>2$, as well as a new proof when $g=3$ and $p>2$, using the moduli space of curves in both cases. We separately prove Oort's conjecture for $g=4$ and $p>2$ with a counterexample in $p=2$. 
Oort's conjecture has been proved by Viehmann~\cite{viehmann:oort} in full generality except for a few counterexamples as above, and the present authors announced a different proof for $p\ge 3$ a while later; see~\cite{karemaker-yu:oort}. \\

A main tool in this paper is the notion of the relative endomorphism algebra $\End(V,W)$ of pairs of vector spaces $(V,W)$ over a field extension $L/K$, where $V$ is a finite-dimensional vector space over $K$, $W\subseteq V\otimes_K L$ is an $L$-subspace, and 
\[\End(V,W):=\{\alpha\in \End_K(V): \alpha(W)\subseteq W\}. \]  
This notion also applies to, e.g., the endomorphism algebras of Drinfeld modules over $\C_\infty$, the function field analogue of the complex numbers, and those of complex tori. The computation of $\End(V,W)$ already appeared in Albert's works~\cite{albert:1934, albert:1935} for computing the matrix multiplication of Riemann matrices. In this paper, we exploit this invariant for studying the endomorphism rings and automorphism groups of certain polarised supersingular abelian varieties (see Section~5); through the construction of the stratification by this invariant, we further investigate how the endomorphism rings and automorphism groups may vary for polarised abelian varieties on supersingular EO strata. Besides its application to Oort's conjecture, we also show that this stratification refines the mass stratification studied in \cite{karemaker-yobuko-yu}, and we provide the mass formula for each stratum in Theorem~\ref{thm:mass}. For proving Oort's conjecture for $g=4$ (and $p>2$), we use results of Harashita~\cite{harassg4} on the moduli space of four-dimensional rigid polarised flag type quotients.   \\

The paper is organised as follows. Section 2 studies basic algebraic properties for pairs of vector spaces and their endomorphism algebras. In Section 3 we give a method for computing the relative endomorphism algebra $\End(V_0,W)$ of a pair of vector spaces $(V_0,W)$ and introduce a stratification on the Grassmanian $\Gr$ of $r$-dimensional subspaces $W$ in a vector space~$V_0$ over a field $k_0$ by the relative endomorphism algebra $\End(V_0,W)$. The main result (Theorem~\ref{thm:nonempty}) shows that on a dense subset of $\Gr$, the relative endomorphism algebra $\End(V_0,W)$ is equal to $k_0$, except when~$k_0$ is either an algebraically closed field, or a real closed field. Section 4 gives the parallel results of Section 3 for the Lagrangian  
varieties of $r$-dimensional isotropic subspaces of a $2r$-dimensional symplectic space. The results of this section are exploited in Section 6 to construct the aforementioned stratification on $\calS_g^{\rm eo}$.
In Section 5, we show how relative endomorphism algebras are related to the endomorphism algebras of polarised abelian varieties and how the notion of relative endomorphism algebras can be used to compute the endomorphism rings and automorphism groups of certain supersingular polarised abelian varieties. 
The proofs of Theorems A and B are given in Section 6. 
Finally, Section 7 proves Oort's conjecture in dimension $g=4$ for $p>2$, by explicit computations for Dieudonn{\'e} modules and their automorphisms.

\subsection*{Acknowledgements}
Part of this work was carried out 
when one or both of the authors visited Academia Sinica, Utrecht University, the Max-Planck-Institut f\"ur Mathematik in Bonn, and the Research Institute for Mathematical Sciences of Kyoto University; we would like to thank these institutes for their support, warm hospitality and excellent working conditions. 
We are indebted to Akio Tamagawa and Eva Viehmann for valuable discussions, which have strengthened Theorems~\ref{thm:nonempty} and~\ref{thm:L:nonempty} and led to various improvements of Section 7. We are grateful to the referees for their very careful reading and helpful comments that have improved the manuscript significantly. Particularly, a referee's comment on a mistake in Proposition 7.6 in an earlier version led us to disprove Oort's conjecture for the case $g=4$ and $p=2$. 
The first author was partially supported by the Dutch Research Council (NWO) through grants VI.Veni.192.038 and VI.Vidi.223.028.
The second author was supported by
the NSTC grants 112-2115-M-001-010, 113-2115-M-001-001 and Academia Sinica IVA grant.

\section{Pairs of vector spaces and endomorphism algebras}

In this section we present algebraic preliminaries for pairs of vector spaces and their endomorphism algebras.

\subsection{Linear independence of specialisations of monomials}\

Let $L/K$ be a field extension and let $p:=\char K\ge 0$. Denote by $\ol K$ an algebraic closure  of the field $K$.

\begin{definition} \ 
   \begin{enumerate}
       \item $L/K$ is called \emph{essentially finite} if either $p=0$ and $L/K$ is finite; or $p>0$, $L/K$ is algebraic and there exist a finite extension $K'/K$ in $\ol K$ and an integer $r\in \Z_{\ge 0}$ such that $L\subseteq (K')^{p^{-r}}$. The latter condition is equivalent to the existence of an integer $r\in \Z_{\ge 0}$ such that the field extension $L^{p^r}K/K$ is finite. 
       \item $L/K$ is called \emph{strictly infinite} if it is not essentially finite. 
   \end{enumerate} 
\end{definition}

Recall that a real closed field is an ordered field $F$, so in particular $\char F=0$, such that every positive element of $F$ has a square root in $F^\times$ and every polynomial of odd degree has a root in $F$. 
A field $F$ is real closed if and only if $F\neq \ol F$ and the field extension $\ol F/F$ is finite.  

\begin{definition}\label{def:exceptional}
  We call a field extension $L/K$ {\it exceptional} if either $K=L=\ol K$, or $K$ is real closed and $L=\ol K$. Equivalently, $L/K$ is exceptional if and only if $L=\ol L$ and $L/K$ is finite.    
\end{definition}

Let $K[X_1,\dots, X_n]$ denote the polynomial ring over $K$ of $n$ variables, and $K[X_1,\dots, X_n]_{\le d}\subseteq K[X_1,\dots, X_n]$, for $d\in \Z_{\ge 0}$, be the subspace consisting of polynomials of degree $\le d$.

\begin{theorem}\label{thm:evT}
   Let $L/K$ be a field extension.
   \begin{enumerate}
   
   \item The following are equivalent:
   \begin{enumerate}
       \item[(i)] For any non-negative integers $n$ and $d$, there exists an element $T=(t_1,\dots, t_n)\in L^n$ such that the evaluation map
       \[ {\rm ev}_T: K[X_1,\dots, X_n]_{\le d}\to L, \quad f(X_1,\dots, X_n) \mapsto f(T) \]
       is injective.
       \item[(ii)] For any $N\in \Z_{\ge 1}$, there exists an element $\alpha\in L$ such that $[K(\alpha):K]\ge N$.
       \item[(iii)] $L/K$ is strictly infinite. 
   \end{enumerate}
   \item Assume that there exists a perfect subfield $F\subseteq K$ such that $\trdeg(K/F)< \infty$ and $K/F$ is separably generated. Then condition (iii) is equivalent to 
   \begin{itemize}
       \item[(iii')] $L/K$ is infinite. 
   \end{itemize}

   \item Assume that $L=\ol L$. Then condition (iii) is equivalent to 
   \begin{itemize}
       \item[(iii'')] $L/K$ is not exceptional. 
   \end{itemize}
   
   \end{enumerate}
\end{theorem}
\begin{proof}
    \begin{enumerate}
        \item  
        Observe that condition (ii) is equivalent to the $n=1$ case of (i). This shows (i)$\implies$(ii). 

        \noindent {(ii)$\implies$(i):} Take $k_i:=(d+1)^{i-1}$ for $i=1,\dots,n$. Using the unique $(d+1)$-adic expression of integers,  the integers 
        $i_1k_1+\dots +i_n k_n$ are mutually distinct for $i_1,\dots i_n\in \Z_{\ge 0}$ with $i_1+\dots+i_n \le d$. Then the evaluation map
        \[ {\rm ev}_{(X^{k_1}, \dots, X^{k_n})}: K[X_1,\dots, X_n]_{\le d} \to K[X]_{\le d'}\]
        is injective, where $d'=(d+1)^n$. 
        By (ii), there exists $\alpha\in L$ 
        such that ${\rm ev}_\alpha:K[X]_{\le d'}\to L$ is injective. Statement (i) holds. 
        
        \noindent {(iii)$\implies$(ii):} 
        If $\trdeg(L/K)>0$, letting $\alpha\in L$ be a transcendental element, then $[K(\alpha):K] \ge N$ for any integer $N$. So we may assume that $L/K$ is algebraic and infinite. 
        Suppose first that $p=0$. As $L/K$ is infinite, we may take a finite subextension $K'/K$ of $L/K$ of degree $\ge N$. Since $K'/K$ is separable, $K'=K(\alpha)$ for some $\alpha\in K'$ and $[K(\alpha):K]\ge N$.

        Suppose now that $p>0$. Let $K_s$ be the separable closure of $K$ in $L$. 

        Case 1: $K_s/K$ is infinite. We show (ii) as in the $p=0$ case. 

        Case 2: $K_s/K$ is finite. One has that $L/K_s$ is purely inseparable. For any integer $N \geq 1$, take an integer~$r$ such that $p^{r+1}\ge N$. 
        Since $L/K$ is strictly infinite, one has $L \not \subseteq (K_s)^{p^{-r}}$. Choose an element $\alpha\in L \setminus (K_s)^{p^{-r}}$. Then the minimal polynomial of $\alpha$ is $T^{p^k}-a$ for some $a\in K_s$ and some $k\ge r+1$, otherwise $L\subseteq (K_s)^{p^{-r}}$. Therefore, $[K(\alpha):K]\ge [K_s(\alpha):K_s]=p^{k}\ge N$.   

 \noindent {(ii)$\implies$(iii):} Condition (ii) implies that $L/K$ is infinite. If $p=0$, then $L/K$ is strictly infinite by definition. Assume $p>0$. Suppose $L/K$ is essentially finite. Then there exist a finite subextension $K'/K$ and an integer $r$ such that $L\subseteq (K')^{p^{-r}}$. Then for any $\alpha\in L$, one has $\alpha^{p^r}\in K'$ and 
 \[ [K(\alpha):K]=[K(\alpha): K(\alpha^{p^r})][K(\alpha^{p^r}):K] \le p^r [K':K], \]
  a contradiction to condition (ii). 
  \item A strictly infinite field extension is infinite. We show the reverse. It suffices to show the case $p>0$.  
  Since $K/F$ is separably generated and $\trdeg(K/F)< \infty$, there exist algebraically independent elements $t_1,\dots, t_m$ in $K$ over $F$ such that $K/F(t_1,\dots, t_m)$ is algebraic and separable. By definition, we need to show that for any finite subextension $K'/K$ in $L$ and any integer $r$, one has $L\not \subseteq (K')^{p^{-r}}$. 

   Since $K'/K$ is finite, the extension $(K')^{p^{-r}} /K^{p^{-r}}$ is also finite. 
   As $K$ is algebraic and separable over $F(t_1,\dots, t_m)$, one has \[ K^{p^{-r}}=K\cdot F(t_1,\dots,t_m)^{p^{-r}}=K\cdot F(t_1^{p^{-r}}, \dots, t_m^{p^{-r}})= K(t_1^{p^{-r}}, \dots, t_m^{p^{-r}}) \]
   and then $[K^{p^{-r}}:K]$ is finite.
   Therefore, $(K')^{p^{-r}}/K$ is finite and $L$ cannot be contained in  $(K')^{p^{-r}}$, as $L/K$ is infinite.

  \item If $L=\ol L$, then $L/K$ is essentially finite if and only if $L=\ol K$ and $L/K$ is finite. By definition, this is equivalent to that $L/K$ is exceptional.  
 \end{enumerate}
\end{proof}

\subsection{Pairs of vector spaces and endomorphism algebras}\

Fix a field extension $L/K$ as before. 

\begin{definition} \ 

   \begin{enumerate}
       \item Consider a pair $({V_0},W)$, where ${V_0}$ is a finite-dimensional vector space over $K$ and $W$ is an $L$-subspace of $V_{0,L}={V_0}\otimes_K L$. We call $({V_0},W)$ a {\it pair of vector spaces over $(K,L)$}, and $(\dim_K {V_0}, \dim_LW)$ the dimension of $({V_0},W)$. For two pairs $({V_0},W)$ and $({V_0}',W')$ of vector spaces over $(K,L)$, define
\[ \Hom(({V_0},W), ({V_0}',W')):=\{f\in \Hom_K({V_0},{V_0}'): f(W)\subseteq W' \}. \]
By $f(W)\subseteq W'$ we understand that $f$ is the extended map ${V_0}_L \to {V_0}'_L$ by linearity. 
Each element of $\Hom(({V_0},W), ({V_0}',W'))$ is called a {\it homomorphism} from $({V_0},W)$ to $({V_0}',W')$. 
A homomorphism $f\in \Hom(({V_0},W),({V_0}',W'))$ is an \emph{isomorphism} if there is an  $f'\in \Hom(({V_0}',W'),({V_0},W))$ such that $ff'=\bbI_{{V_0}'}$ and $f'f=\bbI_{{V_0}}$. 
Let $({\rm Vec}_{L/K})$ denote the category of pairs of vector spaces over $(K,L)$.  
\item When $\dim {V_0}\ge 1$, let $\End({V_0},W)$ denote the endomorphism algebra of the pair $({V_0},W)$. We call $\End({V_0},W)$ the endomorphism algebra of $W$ \emph{defined over $K$}, or the \emph{relative endomorphism algebra} of $W$ when ${V_0}$ is understood. The \emph{automorphism group} of $({V_0},W)$ is denoted by $\Aut({V_0},W):=\End({V_0},W)^\times$.
   \end{enumerate}    
\end{definition}

In the category $({\rm Vec}_{L/K})$, one has the zero object $\ul 0=(0,0)$, the identity object $\bbI =(K,L)$, the direct sum $\oplus$ and the tensor product $\otimes$ satisfying 
\[ (V_0,W)\oplus (V_0',W'):=(V_0\oplus V_0', W\oplus W'), \quad (V_0,W)\otimes (V_0',W'):=(V_0\otimes_K V_0', W\otimes_L W'), \]
and the dual defined by $({V_0},W)^\vee:=({V_0}^\vee, W^\vee)$, where 
\[ {V_0}^\vee:=\Hom_K({V_0},K), \quad W^\vee:=\{f\in {V_{0,L}}^\vee=\Hom_L({V_{0,L}},L): f(W)=0\}. \]
$({\rm Vec}_{L/K})$ is an additive tensor category but not an abelian category. In particular, it is not equivalent to the category $(R$-mod$)$ of $R$-modules for some ring $R$. 

\begin{lemma}
    Let $E$ be a finite-dimensional $K$-algebra and $G=E^\times$ be the multiplicative group. If $K$ is infinite, then $E$ is equal to the $K$-subalgebra generated by $G$. In this case, $E$ is uniquely determined by $G$.
\end{lemma}
\begin{proof}
Let $\ul E$ be the affine space over $K$ defined by $E$ and $\ul G$ be the multiplication group scheme associated to $E$ over $K$. Let $E'$ be the $K$-subalgebra generated by $G$. Since $\ul E$ is rational and $\ul G\subseteq \ul E$ is Zariski-open, $\ul G$ is rational and $G$ is dense in $\ul G$. It follows that $G$ is dense in $\ul E$. As $\ul E'\subseteq \ul E$ is closed, one has $\ul E'\supseteq \ol G=\ul E$. This shows $E'=E$.
\end{proof}

\begin{remark}
    When $K$ is finite, $E$ may not be determined by $G$. For example, if $K=\F_2$ and $E=\F_2^n$, then $G=1\in E$ and $E'=\F_2 \cdot 1$.
\end{remark}

Let $({V_0},W)$ be a pair of vector spaces over $(K,L)$ of dimension $(n,r)$, where $n\ge 1$ and $0\le r \le n$. We denote by $\Gr({V_0},r)$ the Grassmannian of subspaces of dimension $r$ in ${V_0}$. It is a smooth projective scheme over $K$ representing the functor which associates to a $K$-scheme~$S$ the set of locally free $\calO_S$-modules $\calF\subseteq {V_0}\otimes_K\calO_S$ of rank $r$ such that $\calF$ locally for the Zariski topology is a direct summand of ${V_0}\otimes_K\calO_S$.

We may regard $W$ as a point in $\Gr({V_0},r)(L)$.
The subspace $W$ is said be to defined over $K$ if $\dim_K (W\cap {V_0})=\dim_L W$. Then $W\in \Gr({V_0},r)(K)$ if and only if $W$ is defined over~$K$.
If $n\ge 2$, $0<r<n$ and $W\in \Gr({V_0},r)(K)$, then $\End({V_0},W)$ is the subalgebra that preserves the $r$-dimensional $K$-subspace ${V_0}\cap W$ and hence is a parabolic subalgebra of type $(r,n-r)$ by definition; for a suitable $K$-basis of ${V_0}$, it is a subalgebra $E$ represented as
\[\begin{bmatrix}
  \Mat_r(K) & \Mat_{r,n-r} (K) \\
0 & \Mat_{n-r} (K)
\end{bmatrix}.
\]
The unit group $E^\times$ is a maximal parabolic subgroup of type $(r,n-r)$.

\begin{lemma}\label{lm:end_basic} Assume $n\ge 2$ and $0<r <n$.
\begin{enumerate}
        \item There is a natural bijection between $\Gr({V_0},r)(K)$ and the set of parabolic subalgebras of $\End({V_0})$ of type $(r,n-r)$. 
        \item One has $\End({V_0},W)\subsetneq \End({V_0})$.
        \item If $L'$ is a field extension of $L$, then $\End({V_0}, W_{L'})=\End({V_0},W)$, where $W_{L'}=W\otimes_L L'$.   
    \end{enumerate}
\end{lemma}
\begin{proof}
\begin{enumerate}
    \item By definition, the set $P(r,n-r)$ of parabolic subalgebras of type $(r,n-r)$ is transitive under the action of $\GL({V_0})$, so of $\GL({V_0})/\Stab(E_0)$, where $E_0\in P(r,n-r)$ is a base point. One directly checks that $\Stab(E_0)$ is a parabolic subgroup $P_0$ of type $(r,n-r)$. On the other hand, one has an isomorphism $\GL({V_0})/P_0\simeq \Gr({V_0},r)(K)$. 
    \item If $\End({V_0},W)=\End({V_0})$, then every element $\alpha \in \End({V_0},W)\otimes L=\End({V_0})\otimes L$ preserves $W$, a contradiction. 
    \item It is clear that $\End({V_0},W)\subseteq \End({V_0}, W_{L'})$. Conversely, if $\alpha(W_{L'})\subseteq W_{L'}$, then $\alpha(W)\subseteq V_{0,L}\cap W_{L'}=W$.
     \end{enumerate}
\end{proof}

For each $W\in \Gr({V_0},r)(L)$, there is a unique smallest $L$-subspace~$\wt W$ that contains $W$ and is defined over $K$. To see this: if $W_1$ and $W_2$ are two $L$-subspaces that are defined over $K$ and contain $W$, then so is $W_1\cap W_2$. Therefore, such an $L$-subspace exists and is unique. Set $W_0:=W\cap {V_0}$; this is the largest $K$-subspace contained in $W$.
If $L/K$ is Galois, then the $L$-subspace $\wt W$ is equal to the sum of the $L$-subspaces $\sigma(W)$ for all $\sigma \in \Aut(L/K)$, and $W_{0,L}$ is equal to the intersection of $\sigma(W)$ for all $\sigma\in \Gal(L/K)$, since any $\Gal(L/K)$-stable $L$-subspace is defined over $K$.

\begin{definition}\label{def:envelope} Let $W\in \Gr({V_0},r)(L)$.
\begin{enumerate}
    \item Let $\wt W$ be the smallest $L$-subspace containing $W$ that is defined over $K$, set its $K$-model by $\wt W_0=\wt W \cap {V_0}$, and let $W_0:=W\cap {V_0}$.   
    We call $\wt W$ the \emph{$K$-hull} of $W$ and its $K$-model $\wt W_0=\wt W \cap {V_0}$  the \emph{envelope} of $W$.
    \item We call $W$ \emph{$K$-null} if $W_0=0$ and \emph{$K$-dense} if $\wt W_0={V_0}$.
    \end{enumerate}
\end{definition}

Set 
\[ \End({V_0}, W_0, \wt W_0):=\{\alpha\in \End_{K}({V_0}): \alpha(W_0)\subseteq W_0,\alpha(\wt W_0)\subseteq \wt W_0 \}. \]
Clearly if $\alpha\in \End({V_0},W)$, then $\alpha\in \End({V_0}, W_0, \wt W_0)$.
Denote by 
\[ p_{\wt W_0/W_0}: \End({V_0},W_0, \wt W_0)\to \End(\wt W_0/W_0)\] 
the natural projection and set $\ol W:=W/W_{0,L}$. Then 
\begin{equation}\label{eq:red}
    \End({V_0},W)=\{\alpha\in \End({V_0}, W_0, \wt W_0): p_{\wt W_0/W_0}(\alpha)(\ol W)\subseteq \ol W\}. 
\end{equation}

Denote by 
\[ \Gr({V_0},r)^{\rm nd}\subseteq \Gr({V_0},r)(L) \] 
the subset consisting of $L$-subspaces $W\in \Gr({V_0},r)(L)$ such that $W_0=0$ and $\wt W=V_{0,L}$ (The supscript "nd" refers to $K$-null and $K$-dense). 
Then we have the decomposition
\begin{equation}\label{eq:dec}
    \Gr({V_0},r)(L)=\Gr({V_0},r)(K) \sqcup \left( \coprod_{V_1\subseteq V_2} \Gr(V_2/V_1, r-\dim V_1)^{\rm nd}\right),
\end{equation}
where $V_1\subseteq V_2$ runs through all chains of two $K$-subspaces with $\dim V_1<r < \dim V_2$. 
Here we make the identification
\[ \Gr(V_2/V_1, r-\dim V_1)(L)=\{W\in \Gr(V_0,r)(L): V_{1,L} \subseteq W \subseteq V_{2,L}\, \}.  \]
We can also view $\Gr({V_0},r)(K)$ as the set of the special chains $V_1\subseteq V_2$ of $K$-subspaces with $\dim V_1=\dim  V_2=r$.

\begin{example} 
\begin{enumerate}
    \item Let $(n,r)=(n,1)$ and $n>1$. Then $\Gr({V_0},1)(L)=\bbP^{n-1}(L)$ and 
\begin{equation}
    \Gr({V_0},1)^{\rm nd}=\bbP^{n-1}(L)\setminus \bigcup_H \bbP_H(L),
\end{equation}
where $H$ runs through all $K$-rational hyperplanes of ${V_0}$ and $\bbP_H\subseteq \bbP^{n-1}$ is the projective $(n-2)$-space associated to $H$. 

When $n=2$, $K=\R$ and $L=\C$, we have
\[ \Gr({V_0},1)^{\rm nd}=\bbP^1(L)\setminus \bbP^1(K)=\C\setminus \R, \]
which is the union of the upper and lower half-planes. 

When $K$ is a non-archimedean local field and $L=\C_K:=\wh {\ol K}$ is the completion of the algebraic closure of $K$, we have 
\[ \Gr({V_0},1)^{\rm nd}=\Omega_K^n:=\bbP^{n-1}(\C_K) \setminus \bigcup_{H} \bbP_H(\C_K). \]
Here $\Omega_K^n$ is the Drinfeld period space of rank $n$ associated to $K$ introduced in \cite{drinfeld1}. As is well known, this space forms a rigid analytic space, which is equipped with a natural action of $\GL({V_0})\simeq \GL_n(K)$. 

\item When $(n,r)=(2r,r)$, $r\ge 1$, $K=\R$ and $L=\C$, we have that $\Gr({V_0},r)^{\rm nd}$ is in bijection with the set of real Hodge structures of dimension $2r$ with Hodge type $\{(-1,0), (0,-1)\}$. 
\end{enumerate}
\end{example}

By \eqref{eq:red} and \eqref{eq:dec}, one is reduced to computing $\End({V_0},W)$  where $W$ is $K$-null and $K$-dense. 
Based on this connection, the space $\Gr({V_0},r)^{\rm nd}$ can be viewed as an algebraic generalisation of a period space. When the field $K$ is a local field, $\Gr({V_0},r)^{\rm nd}$ is expected to form an analytic space in an appropriate sense, on which $\GL({V_0})$ acts bi-analytically. Furthermore, $\Gr({V_0},r)$ can be viewed as a compactification of $\Gr({V_0},r)^{\rm nd}$ in a suitable sense, on which $\GL({V_0})$ acts analytically and $\GL({V_0})$-equivariantly. Moreover, one sees from Equation \eqref{eq:dec} that the boundary components are algebraic period spaces of the same kind with smaller rank, and the group $\GL({V_0})$ acts transitively on the set of boundary components of same kind. \\

Choose a $K$-basis $e_1,\dots, e_n$ for ${V_0}$, and write ${V_0}=\Mat_{n\times 1}(K)$. Choose a $L$-basis $v_1,\dots v_r$ for $W$ and let
\begin{equation}\label{eq:P}
    P=[v_1, \dots, v_r]\in \Mat_{n\times r}(L)
\end{equation}
be the representative matrix of $W$ with respect to $\{v_i\}$. We recover $W$ from $P$ by the $L$-span $\<P\>_L$ of the column vectors of $P$.

If $r=1$ and $W=\<P\>_L$ is $K$-dense, where $P=[x_1,\dots, x_n]^T\in \Mat_{n\times 1}(L)$, then each $x_i$ is nonzero, that is, every row in $P$ is $L$-linearly independent. However, when $r>1$, it is not true that 
any $r$ rows of $P$ are $L$-linearly independent. That is, some $r\times r$ minor of $P$ may be zero. Below is an example.

\begin{example}
Put 
\[ P=\begin{bmatrix}
    1 & 0 \\
    x & 0 \\
    0 & 1
\end{bmatrix}, \ x\in L\setminus K, \quad \text{and} \quad W=\< P \>_L. \]
Then $W$ is $K$-dense but the first $2\times 2$-block is not invertible. In this case ${V_0}=V_1 \oplus V_2=\<e_1,e_2\>_{K} \oplus \<e_3\>_{K}$. Write
\[ P=\begin{bmatrix}
    P_1 \\
    P_2
\end{bmatrix},\quad  P_1=\begin{bmatrix}
    1 & 0 \\
    x & 0 
\end{bmatrix},\quad  P_2=\begin{bmatrix}
    0 & 1  
\end{bmatrix}. \]
Then $W=W_1 \oplus W_2$ with $W_1=\<P_1\>_L \subseteq V_1\otimes L$ and  $W_2=\<P_2\>_L \subseteq V_2\otimes L$, and each $L$-subspace $W_i$ is $K$-dense in $V_i$. 
\end{example}

\begin{definition} Let $({V_0},W)$ a pair of vector spaces as before. 
    \begin{enumerate} 
    
    \item A \emph{decomposition of $({V_0},W)$} is a decomposition into pairs of vector spaces:
    \[ ({V_0},W)=(V_1,W_1)\oplus (V_2,W_2), \] where for $i=1,2$, ${V}_i$ is a $K$-vector subspace of ${V_0}$ and $W_i\subseteq V_{i,L}$ is an $L$-subspace satisfying ${V_0}=V_1\oplus V_2$ and $W=W_1\oplus W_2$. We call $({V_0},W)$ \emph{indecomposable} if ${V_0}\neq 0$, and for any decomposition $({V_0},W)=(V_1,W_1)\oplus (V_2,W_2)$, either $V_1=0$ or $V_2=0$.
    
    \item We call a decomposition   
    $W=W_1\oplus W_2$ of $W$ into $L$-subspaces \emph{defined over~$K$} if the corresponding envelopes $\wt W_{1,0}$ and $\wt W_{2,0}$ have zero intersection.
     We say that $W$ is \emph{$K$-indecomposable} if $W\neq 0$, and for any decomposition $W=W_1\oplus W_2$ which is defined over $K$, either $W_1=0$ or $W_2=0$.
\end{enumerate}
\end{definition}

The following lemma shows that $\wt W_{1,0}\cap \wt W_{2,0}=0$ if and only if $\wt W_{1}\cap \wt W_{2}=0$; also see Proposition~\ref{prop:dec}(1). 

\begin{lemma}\label{lm:ff_descent}
   Let $A$ be a commutative ring and $B$ be a commutative flat $A$-algebra. Let $M$ be an $A$-module and $M_1,M_2\subseteq M$ be submodules of $M$. Viewing $M_1\otimes_A B$ and $M_2\otimes_A B$ as $B$-submodules of $M\otimes_A B$, we have 
   \begin{equation}\label{eq:5}
       (M_1+M_2)\otimes_A B = (M_1\otimes_A B) + (M_2\otimes_A B), \quad (M_1\cap M_2)\otimes_A B =(M_1\otimes_A B) \cap (M_2\otimes_A B).
   \end{equation} 
   Moreover, if $B$ is faithfully flat over $A$ then 
   \begin{equation} \label{eq:6}
       M_1\cap M_2=0 \iff (M_1\otimes_A B) \cap (M_2\otimes_A B)=0.
   \end{equation}
\end{lemma}

\begin{proof}    
    Consider the short exact sequence
\begin{equation*}
    0 \to M_1\cap M_2 \to M_1\oplus M_2 \to M_1 +M_2 \to 0.
\end{equation*}    
We have the short exact sequences
\[
\begin{split}
    0 \to (M_1\cap M_2)\otimes_A B \to & (M_1\oplus M_2) \otimes_A B\to (M_1 +M_2) \otimes_A B\to 0,\\
    0 \to (M_1\otimes_A B) \cap (M_2\otimes_A B) \to & M_1\otimes_A B\oplus M_2\otimes_A B \to (M_1\otimes_A B) +(M_2\otimes_A B) \to 0.
\end{split}
\]
We now show that $(M_1 +M_2) \otimes_A B=(M_1\otimes_A B) +(M_2\otimes_A B)$. We have the inclusion $\supseteq$ as $M_1+M_2\supseteq  M_1, M_2$. On the other hand, $(M_1 +M_2) \otimes_A B$ is generated by the elements $(m_1+m_2)\otimes b=m_1\otimes b+m_2\otimes b$ for $m_1\in M_1, m_2\in M_2$ and $b\in B$, which are all contained in $(M_1\otimes_A B) +(M_2\otimes_A B)$. Thus, we get the other inclusion $\subseteq$. Then we can identify the above two short exact sequences and obtain 
\[ (M_1\cap M_2)\otimes_A B =(M_1\otimes_A B) \cap (M_2\otimes_A B).\]
If $B$ is faithfully flat over $A$, then $(M_1\cap M_2)\otimes_A B=0$ if and only if $M_1\cap M_2=0$ and then statement \eqref{eq:6} follows from \eqref{eq:5}. 
\end{proof}

\begin{proposition}\label{prop:dec}
  Let $({V_0},W)$ be a pair of vector spaces as above.
  \begin{enumerate}
      \item A decomposition $W=W_1\oplus W_2$ is defined over $K$ if and only if  the corresponding $K$-hulls $\wt W_1$ and $\wt W_2$ have zero intersection.
      
      \item If $({V_0},W)=(V_1,W_1)\oplus (V_2,W_2)$ is a decomposition of $({V_0},W)$, then the decomposition $W=W_1\oplus W_2$ is defined over $K$. Conversely, any decomposition $W=W_1\oplus W_2$ that is defined over $K$ arises from a decomposition $({V_0},W)=(V_1,W_1)\oplus (V_2,W_2)$. Moreover, we then have the canonical decomposition $(\wt W_0,W)=(\wt W_{1,0},W_1)\oplus (\wt W_{2,0},W_2)$.

      \item If we have a decomposition $({V_0},W)=\oplus_{i=1}^t (V_i, W_i)$ with $dim_K V_i\ge 1$ for all $i$, then 
      \begin{equation}\label{eq:endo.1}
         \End(V_0,W)=\begin{bmatrix}
          H_{11} & \cdots & H_{1t} \\
          H_{21} & \cdots & H_{2t} \\
          \vdots & \ddots & \vdots \\
          H_{t1} & \cdots & H_{tt}
      \end{bmatrix}, \quad H_{ij}:=\Hom((V_j,W_j),(V_i,W_i)). 
      \end{equation}
      
      If $f=(f_{ij})$ with $f_{ij}\in \Hom(V_j,V_i)$ and $v=(v_1,\dots, v_t)$ with $v_i\in V_{i,L}$, then 
      \begin{equation}\label{eq:fv}
         f(v)=(\sum_{k=1}^t f_{ik} v_k)_i. 
      \end{equation}

      \item There is a decomposition 
      \begin{equation}\label{eq:can_dec}
         ({V_0},W)=(V_1,W_1) \oplus (V_2,W_2) \oplus(V_3,0) 
      \end{equation}
      with $V_1=W_0, W_1=W_{0,L}$, and $V_1\oplus V_2=\wt W_0$. The endomorphism algebra $\End({V_0},W)$ is given by
      \begin{equation}\label{eq:endo.2}
          \End({V_0},W)=\begin{bmatrix}
          \End(V_1) & \Hom(V_2,V_1) & \Hom(V_3,V_1) \\
          0 & \End(V_2,W_2) & \Hom(V_3, V_2) \\
          0 &  0 & \End(V_3)
      \end{bmatrix}.
      \end{equation} 
      Moreover, the subspaces $W_0$ and $\wt W_0$ are uniquely  determined by the endomorphism algebra $\End({V_0},W)$.

      \item Suppose $W'\in \Gr({V_0},r)(L)$ is another member such that $\End({V_0},W')$ is conjugate to $\End({V_0}, W)$ under $\GL({V_0})$. Then there is an element $\gamma\in \GL({V_0})$ and a decomposition
      \begin{equation}\label{eq:dec2}
         ({V_0},\gamma W')=(V_1,W_1') \oplus (V_2,W'_2) \oplus(V_3,0) 
      \end{equation}
      with $V_1$, $V_2$, $V_3$ and $W_1$ as in part (4), such that 
      $\End(V_2,W_2)=\End(V_2, W_2')$.
    
  \end{enumerate}  
\end{proposition}

\begin{proof}
\begin{enumerate}
    \item It suffices to show that $\wt W_{1,0} \cap \wt W_{2,0}=0$ if and only if $\wt W_1\cap \wt W_2=0$. If $\wt W_1\cap \wt W_2=0$, then $\wt W_{1,0} \cap \wt W_{2,0}\subseteq \wt W_1\cap \wt W_2=0$. Conversely, if $\wt W_{1,0} \cap \wt W_{2,0}=0$, by Lemma~\ref{lm:ff_descent}, one has 
    \[ \wt W_1\cap \wt W_2=(\wt W_{1,0} \cap \wt W_{2,0}) \otimes_{K} L=0.\]
    
    \item The first statement follows because $\wt W_{1,0} \cap \wt W_{2,0}\subseteq V_1\cap V_2=0$.
    Since $W=W_1\oplus W_2$ is defined over $K$, we have $\wt W_{1,0} \cap \wt W_{2,0}=0$. Then one can choose two $K$-subspaces $V_1\supset \wt W_{1,0}$ and $V_2\supset \wt W_{2,0}$ such that ${V_0}=V_1\oplus V_2$. One gets a decomposition $({V_0},W)=(V_1,W_1)\oplus (V_2,W_2)$ and this shows the second statement. For the last statement, we need to show that $\wt W_0=\wt W_{1,0}\oplus \wt W_{2,0}$. Since $W=W_1\oplus W_2$ is defined over $K$, $\wt W_{1,0}\cap \wt W_{2,0}=0$ and $\wt W_1\cap \wt W_2=0$. Since $\wt W\supset \wt W_i$ for $i=1,2$, one has $\wt W \supset \wt W_1\oplus \wt W_2$. On the other hand, since $\wt W_1\oplus \wt W_2$ is an $L$-subspace defined over $K$ that contains $W$, one has $\wt W \subseteq \wt W_1\oplus \wt W_2$. Therefore, one has $\wt W=\wt W_1 \oplus \wt W_2$ and $\wt W_0=\wt W_{1,0}\oplus \wt W_{2,0}$, as desired.

    \item It is clear that any $f\in \End({V_0},W)\subseteq \End({V_0})$ has the form $(f_{ij})$ with each $f_{ij}\in \Hom(V_j,V_i)$ and the map is given by the formula \eqref{eq:fv}. So it suffices to check $f_{ij}(W_j)\subseteq W_i$ for all $i,j$. But this is trivial since one has $f_{ij}(W_j)\subseteq V_{i,L}\cap W=W_i$, following from the decomposition of $(V,W)$.   

    \item Take $V_1=W_0$ and $V_2$ is a complement of $V_1$ in $\wt W_0$, so that $\wt W_0=V_1\oplus V_2$. Take $W_1=V_{1,L}$ and $W_2=V_{2,L} \cap W$. Then $(\wt W_0,W)=(V_1,W_1)\oplus (V_2,W_2)$. Take $V_3$ to be a complement of $\wt W_0$ in ${V_0}$. Then we obtain the desired decomposition in Equation~\eqref{eq:can_dec}. Equation~\eqref{eq:endo.2} then follows from Equation~\eqref{eq:endo.1}. From Equation~\eqref{eq:endo.2}, we see that the block $\End(V_1)$ has size $\dim V_1 = \dim W_0$, and $\End(V_3)$ has size $\dim V_3$, which is $n-\dim \wt W_0$. Therefore, the numbers $\dim W_0$ and $\dim \wt W_0$ can be read off from $\End({V_0},W)$.

    Let $p_1: {V_0}\to V_1$ be the projection map with respect to the decomposition ${V_0}=V_1\oplus V_2\oplus V_3$, and let
    \[ p_1:\End({V_0})=\Hom({V_0},V_1)\oplus \Hom({V_0},V_2)\oplus \Hom({V_0},V_3) \to \Hom({V_0},V_1) \]
    be the induced projection map. By Equation~\eqref{eq:endo.2}, we see that the restricted map $p_1:\End({V_0},W)\to \Hom({V_0},V_1)$ is surjective. 
    Also, restriction to $V_3$ gives a surjective map $\iota_3:\Hom({V_0},{V_0})\to \Hom(V_3,{V_0})$, and similarly, we have that the map $\iota_3: \End({V_0},W)\to \Hom(V_3,{V_0})$ is surjective. 

    Suppose $W'$ is another member such that $\End({V_0},W')=\End({V_0},W)$. Then we have $\dim V_1=\dim W_0'$ and $\dim \wt W_0=\dim \wt W_0'$.
    
    If $\dim V_1=\dim W_0'=0$, then $V_1=W_0'$. Suppose $\dim V_1=\dim W_0'>0 $ and $V_1\neq W_0'$, then $V_1\cap W_0'\subsetneq V_1$. Since any $f\in \End({V_0},W) = \End({V_0},W')$ satisfies $f(W_0')\subseteq W_0')$, the image satisfies
    \[ p_1(\End({V_0},W))\subseteq \{f\in \Hom({V_0},V_1): f(W_0')\subseteq V_1\cap W_0'\}\subsetneq \Hom({V_0},V_1), \] a contradiction.

    If $\dim \wt W_0=\dim \wt W_0'=n$, then $\wt W_0 =\wt W_0'$. 
    Suppose then that $\dim \wt W_0 = \dim \wt W_0'$ is less than $n$ and that $\wt W_0\neq \wt W_0'$. There exists a nonzero subspace $V_4 \subseteq \wt W_0'$ such that $\wt W_0+\wt W_0'=\wt W_0 \oplus V_4$. We take $V_3\subseteq {V_0}$ so that $V_3 \oplus \wt W_0={V_0}$ and $V_4 \subseteq V_3$. Choose a complementary subspace $V_5\subseteq V_3$ such that $V_3=V_4\oplus V_5$. Then the image satisfies
    \[ \iota_3 (\End({V_0},W))\subseteq \{f\in \Hom(V_3,{V_0}): f(V_4)\subseteq \wt W_0'\} \subsetneq \Hom(V_3,V_0), \] a contradiction.

    \item Let $\gamma\in \GL({V_0})$ such that 
    \begin{equation}
        \label{eq:conj2}
        \End({V_0},W)=\gamma \End({V_0},W') \gamma^{-1}=\End({V_0},\gamma W').
    \end{equation}
    By part (4), we have $W_0=(\gamma W')_0=\gamma W_0'$ and $\wt W_0=\gamma \wt W_0'$. Therefore $V_1=\gamma W_0'$ and $V_1\oplus V_2=\gamma \wt W_0'$. Put $W_1':=V_{1,L}$ and $W_2':=\gamma W'\cap V_{2,L}$, then we obtain a decomposition in Equation~\eqref{eq:dec2}. Since $\End({V_0},W)=\End({V_0},\gamma W')$, by Equation~\eqref{eq:endo.2}, we get 
    $\End(V_2,W_2)=\End(V_2, W_2')$. 
\end{enumerate}
\end{proof}

\begin{remark}
    Note that Equation~\eqref{eq:red}  is the coordinate-free description of Equation~\eqref{eq:endo.2}.
\end{remark}

\section{General linear groups}\label{sec:GL}

From now on, we let $k_0$ be a field and $k$ an algebraically closed field containing $k_0$. Let $V_0$ be an $n$-dimensional $k_0$-vector space with $n\ge 2$, and $r$ be an integer with $0 < r< n$. Let $\Gr(V_0,r)$ be the Grassmannian of subspaces of $V_0$ of dimension $r$. Set $\Gr=\Gr(V_0,r)\otimes_{k_0} k$. As a $k$-algebraic variety, we may identify $\Gr$ with the set of $k$-points of $\Gr$.

Two $k$-subspaces $W$ and $W'$ in $\Gr$ are said to be {\it isomorphic} if there is an element $\gamma \in \GL(V_0)$ such that $W'=\gamma \cdot W$; equivalently, there is an isomorphism $(V_0,W)\simeq (V_0, W')$ of pairs of vector spaces. In this case, we have
\begin{equation}
    \label{eq:conj}
    \End(V_0,W')=\gamma \End(V_0,W) \gamma^{-1},\quad \Aut(V_0,W')=\gamma \Aut(V_0,W) \gamma^{-1}.
\end{equation}

If $P=[v_1,\dots, v_r]$ is a presenting matrix for $W$ with basis $v_1,\dots, v_r$, then $P'=\gamma P=[v_1',\dots, v_r']$ is a presenting matrix for $W'$ with basis $v_1'=\gamma v_1, \dots , v_r'=\gamma v_r$. The operation $P\mapsto \gamma P$ for $\gamma\in \GL(V_0)$ contains all permutations of rows. Thus, we get the following:

\begin{lemma}
    Any $r$-dimensional $k$-subspace $W$ is isomorphic to a $k$-subspace $W'$ whose representative matrix~$P'$ with respect to any basis has its last $r\times r$ block invertible. In other words, the last $r$ row vectors are $k$-linearly independent. 
\end{lemma}

After a suitable translation by $\GL(V_0)$, we may represent $W$ by 
$\begin{bmatrix}
    T \\
    \bbI_r
\end{bmatrix}$, where $T=(t_{ij})\in \Mat_{n-r, r}(k)$, for a suitable basis $\{v_i\}$; see Equation~\eqref{eq:P}.
For any element \[ \alpha=\begin{bmatrix}
    A & B \\
    C & D 
\end{bmatrix} \in 
\begin{bmatrix}
    \Mat_{n-r,n-r}(k_0) & \Mat_{n-r,r}(k_0)  \\
    \Mat_{r,n-r}(k_0) & \Mat_{r,r}(k_0)
\end{bmatrix}
=\End(V_0), \]
one has 
\[ \alpha(W)\subseteq W \iff \alpha P = \begin{bmatrix}
    AT+B \\
    CT+D 
\end{bmatrix} =\begin{bmatrix}
    T \\
    \bbI_r 
\end{bmatrix}\cdot (CT+D). \]
The right hand side gives the equation
\begin{equation}\label{eq:def_eq}
   TCT+TD-AT-B=0\ \text{in } \Mat_{n-r,r}(k). 
\end{equation}
Thus, we get 
\begin{equation}
   \End(V_0,W)= \left \{\begin{bmatrix}
    A & B \\
    C & D 
\end{bmatrix}\in \Mat_n(k_0): TCT+TD-AT-B=0 \right \}.  
\end{equation}

Conversely, let $E\subseteq \End(V_0)$ be a $k_0$-subalgebra of $\End(V_0)$. Suppose $E=k_0[\alpha_1,\dots \alpha_s]$ is generated by some elements $\alpha_\mu \in \End(V_0)$ and write 
\[ \alpha_\mu=\begin{bmatrix}
    A_\mu & B_\mu \\
    C_\mu  & D_\mu 
\end{bmatrix}, \quad \mu=1,\dots, s. \]
Let $U$ be the open affine subset
\begin{equation}
   U:=\left \{\big\< \begin{bmatrix}
    T \\
    \bbI_r 
\end{bmatrix}\big\> : T\in \Mat_{n-r,r}(k) \right \}\subseteq \Gr,
\end{equation}
and let
\begin{equation}
   U_E:=\left \{W\in U 
 : E\subseteq \End(V_0,W)  \right \}.
\end{equation}
Clearly $U\subseteq \Gr$ is an open subvariety over $k$ that is defined over $k_0$. Then $U_E\subseteq U$ is the closed subvariety over~$k_0$ defined by the matrix equations
\begin{equation}
    \label{eq:defeqE}
\end{equation}
\[ TC_\mu T+TD_\mu-A_\mu T-B_\mu=0, \quad \mu=1,\dots, s. \]
We define 
\[ \Gr_E:=\{W\in \Gr: E\subseteq \End(V_0,W)\}, \]
which is a closed subvariety of $\Gr$. Indeed, 
we cover $\Gr$ by open affine subsets $\gamma U$ with $\gamma\in \GL(V_0)$. For each open affine subset $\gamma U$, the intersection $(\gamma U)\cap \Gr_E=(\gamma U)_E\simeq U_{\gamma^{-1} E \gamma}$ which is a closed subvariety shown as above. Therefore, $\Gr_E\subseteq \Gr$ is a closed subvariety. From \eqref{eq:defeqE}, we see that $\Gr_E$ is defined over $k_0$.

Clearly, for any two $k_0$-subalgebras $E, E'\subseteq \End(V_0)$, we have   
\begin{equation}\label{eq:strat_rel}
   E\subseteq E' \implies \Gr_{E} \supseteq \Gr_{E'}.  
\end{equation}
Therefore, we obtain a stratification on $\Gr=\Gr(V_0,r)$ by 
\[ \Gr=\bigcup_{E} \Gr_E  \]
by closed subvarieties $\Gr_E$, where $E$ runs through all $k_0$-subalgebras of $\End(V_0)$. Note that the reverse implication in~\eqref{eq:strat_rel} may not hold. Indeed, take $N=\left (\begin{smallmatrix}
   0 & 1 \\
   0 & 0
\end{smallmatrix}\right ), E=\Fq[N]$ and $ E'=B(\Fq)$. Then we have $\Gr_{E}=\Gr_{E'}=\{[1:0]\}$, but $E\neq E'$.

Denote by 
\begin{equation}
  \calE:=\{ \End(V_0,W): W\in \Gr \}   
\end{equation}
the set of all $k_0$-algebras of $\End(V_0)$ that occurs as the relative endomorphism algebra of $W$ for some $W\in \Gr$.

For each $E\in \calE$, set 
\begin{equation}
    (\Gr_E)^0:=\{W\in \Gr: \End(V_0,W)=E \}.    
\end{equation}
Then we have 
\begin{equation}
    \Gr=\coprod_{E\in \calE} (\Gr_E)^0, \quad \Gr_E= \coprod_{E'\in \calE: \atop E\subseteq E'} (\Gr_{E'})^0.
\end{equation}
By definition, $(\Gr_E)^0$ is always non-empty.

We shall call any member $E\in \calE$ a relative $k_0$-subalgebra of $\End(V_0)$ (of type $(r,n~-~r)$, if necessary).
When the set $\calE$ is finite, which is the case exactly when $k_0$ is finite by Lemma~\ref{lm:end_basic}.(1), $(\Gr_E)^0$ is a locally closed subset and can be viewed as a reduced subscheme over $k_0$. However, when $k_0$ is an infinite field, the set $(\Gr_E)^0$ is not even constructible, as it is the (non-empty) complement of an infinite union of proper closed subsets in the closed subvariety $\Gr_E$.    
Nevertheless, it still makes sense to ask whether  $(\Gr_E)^0$ is Zariski-dense in $\Gr_E$ and whether 
the strata $\{(\Gr_E)^0\}_{E\in \calE}$ satisfy the stratification property: that is, whether the closure of each stratum is a union of some strata. 
We expect that both questions have affirmative answers.\\

The group $\GL(V_0)$ acts on $\Gr$ on the left and permutes the strata $\Gr_E$. Namely, for each element $\gamma\in \GL(V_0)$ the action gives an isomorphism 
\[ \gamma: \Gr_E \isoto \Gr_{\gamma^{-1} E \gamma}, \quad W\mapsto \gamma W. \]
It is helpful to introduce another intermediate stratification which takes the action of $\GL(V_0)$ into consideration. 
Denote by $[\calE]$ the set of conjugacy classes of all $k_0$-subalgebras $E$ in $\calE$ and by $[E]\subseteq \calE$ the conjugacy class of~$E$ in $\calE$.  For each $[E]$ in $[\calE]$, set 
\begin{equation}
    \Gr_{[E]}:=\bigcup_{E\in [E]} \Gr_E, \quad (\Gr_{[E]})^0:=\bigcup_{E\in [E]} (\Gr_E)^0.
\end{equation}
Then we have a stratification
\begin{equation}\label{eq:st_G_[E]}
    \Gr=\bigcup_{[E]\in [\calE]} \Gr_{[E]}, \quad  \Gr=\coprod_{[E]\in [\calE]} (\Gr_{[E]})^0
\end{equation}
by the strata $\Gr_{[E]}$ or the strata $(\Gr_{[E]})^0$. 
When $k_0$ is finite, each $\Gr_{[E]}$ (resp.~$(\Gr_{[E]})^0$) is a closed (resp.~locally closed) reduced subscheme of finite type over $k_0$, and we also expect that the strata $\{(\Gr_[E])^0\}_{[E]\in [\calE]}$ satisfy the stratification property.

\begin{example}
    Let $k_0=\Fq$, $(n,r)=(2,1)$, in which case $\Gr\simeq\bbP^1$ is defined over $\Fq$. Let $W_t=\<t_1 e_1+t_2 e_2\>\subseteq V_{0,k}$ denote the subspace corresponding to the parameter $t=[t_1:t_2]\in \bbP^1(k)$. Let $B_2(\Fq)\subseteq \Mat_2(\Fq)$ be the Borel subalgebra consisting of upper-triangular matrices, which stabilises the subspace $W_{[1:0]}$. One has $\Gr_{B_2(\Fq)}=\{[1:0]\}$ and $\Gr_{[B_2(\Fq)]}=\bbP^1(\Fq)$, since we know that all Borel subalgebras are $\GL_2(\Fq)$-conjugate. We fix an embedding $\F_{q^2}\hookrightarrow \Mat_2(\Fq)$ of the quadratic extension $\F_{q^2}/\Fq$. For $t\in \bbP^1(k)\setminus \bbP^1(\Fq)=k-\Fq$, we have 
    \begin{equation}
        \label{eq:EndofGr(2,1)}
        \End(V_0,W_t)=\begin{cases}
            \Fq + \Fq\cdot \begin{bmatrix}
                0 & 1 \\
                \beta & \alpha 
            \end{bmatrix} 
            & \text{if $t\in \F_{q^2}\setminus \Fq$}; \\
            \Fq &  \text{if $t\in k\setminus \F_{q^2}$},
        \end{cases}        
    \end{equation}
    where $X^2-\alpha X-\beta$ is the minimal polynomial of $t$ over $\Fq$; see \cite[Section 3.2]{yuyu}. Thus, we have $(\Gr_{\Fq})^0=(\Gr_{[\Fq]})^0=\bbP^1(k)\setminus \bbP^1(\F_{q^2})$,  $(\Gr_{E})^0=\{t,t^q\}$, if $E=\End(V_0,W_t)$ and $t\in \F_{q^2}\setminus \Fq$, and $(\Gr_{[E]})^0=\bbP^1(\F_{q^2})\setminus \bbP^1(\F_{q})$. In particular, we have $[\calE]=\{[B_2(\Fq)],[\F_{q^2}], [\Fq]\}$ and 
    \begin{equation}
        \label{eq:strat(2,1)}
         \Gr=\coprod_{[E]\in [\calE]} (\Gr_{[E]})^0=\bbP^1(\F_{q}) \coprod \left ( \bbP^1(\F_{q^2})\setminus \bbP^1(\F_{q}) \right ) \coprod (\bbP^1\setminus \bbP^1(\F_{q^2})). 
    \end{equation}
    This shows that the collection of strata $\{(\Gr_{[E]})^0\}_{[E]\in [\calE]}$ (resp.~$\{(\Gr_{E})^0\}_{E\in \calE}$) satisfies the stratification property in this case. 
\end{example}

\begin{remark}
    If $k_0$ is infinite, then $\Gr_{[E]}$ is a union of possibly infinitely many closed subvarieties $\Gr_E$ of $\Gr$, so in many situations $\Gr_{[E]}$ is not a subscheme of finite type or even a constructible subset. 
We have already seen (in Lemma~\ref{lm:end_basic}.(1)) that when $E_0$ is a maximal parabolic $k_0$-algebra of type $(r,n-r)$, we have $\Gr_{E_0}=\{\mathrm{pt}\}$ and $\Gr_{[E_0]}=\Gr(V_0,r)(k_0)$. 
Thus, an expected minimal stratum $\Gr_{[E_0]}=\Gr(V_0,r)(k_0)$ is Zariski dense in $\Gr$, as $\Gr$ is a rational variety. This shows that taking the Zariski closure as our inclusion (or closure) relation is not a good notion to define the partial order on these strata.
\end{remark}

For any two members $[E], [E']\in [\calE]$, we define a partial ordering
\[ [E]\subseteq [E'] \] if there exist elements $E_1\in [E]$ and $E_2\in [E']$ satisfying $E_1\subseteq E_2$. It follows that for any $E_1\in [E]$ (resp.~$E_2\in [E']$) there is $E_2\in [E']$
(resp.~$E_1\in [E]$) such that $E_1\subseteq E_2$.
Clearly, we have 
\[ [E]\subseteq [E'] \implies \Gr_{[E']}\subseteq \Gr_{[E]}, \]
\[ \Gr_{[E]}=\coprod_{[E]\subseteq [E']} (\Gr_{[E']})^0. \]

Using this partial ordering, we have the notion of (non-empty) maximal strata and minimal strata. 
One might expect that every minimal stratum is zero-dimensional. This is true when~$k_0$ is finite. However, when $k_0=k$, we see that $\Gr$ itself is a stratum, which is of dimension $r(n-r)$.
Clearly, the stratum $(\Gr_{[k_0]})^0$ is the unique maximal stratum if it is non-empty.\\

We deduce a criterion for the existence of $W\in \Gr$ whose endomorphism algebra equals~$k_0$, that is, a criterion for deciding whether the stratum $(\Gr_{[k_0]})^0$ is non-empty. Recall that $r = \dim(W)$. Let $k_0[\ul X]$ be the polynomial ring over $k_0$ with variables $X_{ij}$ for $1\le i\le n-r$ and $1\le j \le r$ and let $k_0[\ul X]_{\le d}$ be the $k_0$-vector subspace of polynomials of total degree $\le d$. Set \[ N_d:=\dim_{k_0} k_0[\ul X]_{\le d}. \] 
For any $T=(t_{ij})\in \Mat_{n-r\times r}(k)$, let ${\rm ev}_T:k_0[\ul X]\to k$ be the evaluation map at $T$. 
Set 
\begin{equation}
    \bbV(T):={\rm ev}_T(k_0[\ul X]_{\le 2}), \quad d(T):=\dim_{k_0} \bbV(T).
\end{equation}
For each $1\le i\le n-r$ and $1\le j\le r$, let $E_{ij}$ be the matrix with $(i,j)$-entry $1$ and other entries $0$.

\begin{proposition}\label{prop:End=k_0}
   Let $T\in \Mat_{n-r\times r}(k)$ and $W_T=\big\<\begin{bmatrix}
        T \\ \bbI_r
    \end{bmatrix}\big\>$ be the corresponding $k$-subspace in~$\Gr$. 
    If $d(T)=N_2$, then $\End(V_0,W_T)=k_0$. 
\end{proposition} 
\begin{proof} 
    We write $T=\sum_{i,j} E_{ij} t_{ij}$. Using Equation \eqref{eq:def_eq}, we get 
\begin{equation} \label{eq:def-eq2}
   \sum_{\substack{1\le i_1,i_2 \le n-r \\ 1\le j_1,j_2 \le r}} (E_{i_1 j_1} C E_{i_2j_2}) t_{i_1 j_1} t_{i_2 j_2} +\sum_{\substack{1\le i \le n-r \\ 1\le j\le r}}  (E_{ij}D-A E_{ij})t_{ij} -B =0. 
\end{equation}

Since $N_2=d(T)$, the vectors $1, t_{ij}, t_{i_1 j_1} t_{i_2j_2}$ for all $i,j,i_1,j_1,i_2,j_2$ are $k_0$-linearly independent. So we have $B=0$, and 
\[ E_{ij}D-A E_{ij}=0,\quad E_{ij} C E_{ij}=0, \quad E_{i_1 j_1} C E_{i_2j_2}+E_{i_2j_2} C E_{i_1 j_1}=0 \  
 \]
for all $i,j$ and for all $(i_1,j_1)\neq (i_2,j_2)$.
One has $E_{i_1 j_1} C E_{i_2j_2}+E_{i_2j_2} C E_{i_1 j_1}=c_{j_1 i_2} E_{i_1 j_2}+c_{j_2 i_1} E_{i_2 j_1}=0$. So $c_{j_1 i_2}=0$ for $1\le j_1 \le r$, $1\le i_2\le n-r$ and $C=0$. Also from 
\[ E_{ij}D=\sum_{\substack{1\le i_1,j_1  \le r}} E_{ij} d_{i_1 j_1} E_{i_1,j_1}=\sum_{1\le j_1\le r} d_{j, j_1} E_{i,j_1}=d_{jj}E_{ij}+\sum_{1\le j_1\le r \atop j_1\neq j} d_{j, j_1} E_{i,j_1} \]
\[ A E_{ij}=\sum_{1\le i_1,j_1\le n-r} a_{i_1 j_1} E_{i_1 j_1} E_{ij}=\sum_{1\le i_1\le n-r} a_{i_1 i}
E_{i_1 j}=a_{ii}E_{ij}+\sum_{1\le i_1\le n-r\atop 
i_1\neq i} a_{i_1 i} E_{i_1 j} , \]
one concludes that $a_{ij}=d_{ij}=0$ for $i\neq j$ and $d_{jj}=a_{ii}$ for $1\le i\le n-r$, $1\le j\le r$. 
It follows that $\begin{bmatrix}
    A & B \\
    C & D
\end{bmatrix}=a\bbI_n$ for some $a\in k_0$. This proves the proposition.
\end{proof}
    
Recall that $\Gr^{\rm nd}\subseteq \Gr$ is the subset consisting of $k$-subspaces $W$ which are $k_0$-null and $k_0$-dense.

\begin{theorem}\label{thm:nonempty}
    If the field extension $k/k_0$ is infinite, then $(\Gr_{[k_0]})^0\subseteq \Gr^{\rm nd}$ and both are non-empty. 
\end{theorem}
\begin{proof}
    If $W\in \Gr \setminus \Gr^{\rm nd}$, then $\End(V_0,W)$ is of the form in \eqref{eq:endo.2} with either $\End(V_1)\neq 0$ or $\End(V_3)\neq 0$, which cannot be $k_0$. Therefore, $(\Gr_{[k_0]})^0\subseteq \Gr^{\rm nd}$.
    
    By Theorem~\ref{thm:evT}, there exists an element $T=(t_{ij})\in \Mat_{n-r,r}(k)$ such that the map ${\rm ev}_T:k_0[\ul X]_{\leq 2}\to k$ is injective. It follows that $N_2=d(T)$.  By Proposition~\ref{prop:End=k_0}, we have $\End(V_0,W_T)=k_0$. This shows that $(\Gr_{[k_0]})^0$ is non-empty.      
\end{proof}

\begin{proposition}\label{prop:endo} 
    Let $\{t_1'=1, t_2',\dots, t_s'\}\subseteq \{1,t_{ij}, t_{i_1j_1} t_{i_2 j_2} \text{ such that } 1\le i,i_1,i_2 \le n~-~r \text{ and } 1\le j,j_1,j_2 \le r\}$ be a maximal $k_0$-linearly independent subset. Write 
\begin{equation}\label{eq:lin-comb}
    t_{ij}=\sum_{\mu=1}^s \alpha_{ij}^\mu t_\mu', \quad t_{i_1j_1} t_{i_2 j_2}=\sum_{\mu=1}^s \alpha_{i_1j_1 i_2 j_2}^\mu t_\mu', \quad \text{with} \ \alpha_{ij}^\mu,\  \alpha_{i_1j_1 i_2 j_2}^\mu\in k_0. 
\end{equation}
 Let $T = (t_{ij}) \in \Mat_{n-r\times r}(k)$ and let $W_T=\big\<\begin{bmatrix}
        T \\ \bbI_r
    \end{bmatrix}\big\>$ be the corresponding $k$-subspace. 
Then $\End(V_0,W_T)$ consists of all matrices 
$   \begin{bmatrix}
    A & B \\
    C & D
\end{bmatrix}\in \Mat_n(k_0)$ satisfying
\begin{equation}\label{eq:endo2}
   \sum_{\substack{1\le i_1,i_2 \le n-r \\ 1\le j_1,j_2 \le r}} (E_{i_1 j_1} C E_{i_2j_2}) \alpha_{i_1j_1 i_2 j_2}^\mu \\  +\sum_{\substack{1\le i \le n-r \\ 1\le j\le r}}  (E_{ij}D-A E_{ij})\alpha_{ij}^\mu -B \delta_{1,\mu}=0, 
\end{equation}
for $\mu=1,\dots, s$, where $\delta_{1,\mu}$ is the Kronecker delta.

\end{proposition}

\begin{proof}
By \eqref{eq:def-eq2} and \eqref{eq:lin-comb}, we get 
\[ \sum_{\mu=1}^s \left [ \sum_{\substack{1\le i_1,i_2 \le n-r \\ 1\le j_1,j_2 \le r}} (E_{i_1 j_1} C E_{i_2j_2}) \alpha_{i_1j_1 i_2 j_2}^\mu +\sum_{\substack{1\le i \le n-r \\ 1\le j\le r}}  (E_{ij}D-A E_{ij})\alpha_{ij}^\mu -B \delta_{1,\mu} \right ] t_\mu' =0. \]
As the $t_\mu'$s are $k_0$-linearly independent, we get \eqref{eq:endo2}. 
\end{proof}
 
\begin{remark}
    Recall from Definition~\ref{def:exceptional} that an \emph{exceptional} field extension $k/k_0$ satisfies either $k=k_0$ or $k_0$ is real closed (and $k = \overline{k}_0$).
    When $k=k_0$, Lemma~\ref{lm:end_basic}.(1) gives a complete description of the elements of $\Gr = \Gr(V_0,r)(k_0)$ in terms of parabolic subalgebras of $\mathrm{End}(V_0)$ of type $(r,n-r)$ for $0<r<n$. Thus, we have a complete understanding of $\mathrm{End}(V_0,W)$ for all~$W$ when $k=k_0$.
    We shall determine all possible relative endomorphism algebras of $W$ for the case $k_0=\R$ and $k=\C$ in Section~\ref{sec:end_R}.
\end{remark}

\section{Symplectic groups}\label{sec:symp}

Let $(V_0,\psi_0)$ be a non-degenerate symplectic $k_0$-space of dimension $n=2r$. Let ${\rm L}(V_0)={\rm L}(V_0,\psi_0)\subseteq \Gr(V_0,r)$ 
denote the Lagrangian variety associated to $(V_0,\psi_0)$, which parametrises maximal isotropic subspaces of $(V_0,\psi_0)$. It is a smooth irreducible projective scheme over $k_0$ of dimension $r(r+1)/2$. We set ${\mathrm L}:={\rm L}(V_0)\otimes_{k_0} k$, and identify ${\rm L}$ with the set of $k$-rational points of ${\rm L}$. 
Let 
\[\Sp(V_0):=\{\alpha\in \GL(V_0): \psi_0(\alpha x, \alpha y)=\psi_0(x,y) \ \forall\, x,y \in V_0\} \] 
denote the symplectic group associated to $(V_0,\psi_0)$, viewed as an abstract group. For any $k$-subspace $W\in \mathrm{L}$, one defines the endomorphism algebra $\End(V_0,W)$ of $W$ as before, and the automorphism group of $W$ with respect to $(V_0,\psi_0)$ is defined as
\begin{equation}\label{eq:SpW}
    \Sp(V_0,W):=\Sp(V_0)\cap \End(V_0,W)^\times. 
\end{equation}

Choose a symplectic basis $e_1,\dots, e_r$, $e'_1, \dots, e'_r$ of $(V_0,\psi_0)$ and view $V_0=\Mat_{2r \times 1}(k_0)$. Then the pairing $\psi_0$ is represented by the matrix
\[ \psi_0\sim \bbJ_{2r}=\begin{bmatrix}
    0 & \bbI_r \\
   -\bbI_r & 0 \\    
\end{bmatrix}.\]
Choose a $k$-basis $v_1,\dots, v_r$ of $W$. Write
\[ P=[v_1,\dots, v_r]=\begin{bmatrix}
    u_1 \\
    \vdots \\
    u_r \\
    u_1' \\
    \vdots \\
    u_r' \\
\end{bmatrix}=\begin{bmatrix}
    U \\ U' 
\end{bmatrix}, \quad \text{with } U=(u_{ij}),\ U'=(u'_{ij}) \in \Mat_r(k), \]
for the representative matrix of~$W$ with respect to $\{v_i\}$, where $u_i$ and $u_i'$ are row vectors.
For each $1\le i \le r$, define $\varepsilon_i\in \Sp(V_0)$ by 
\[ \varepsilon_i(e_i)=e'_i, \quad \varepsilon_i(e'_i)=-e_i, \quad \text{and} \quad \varepsilon_i(e_j)=e_j, \quad \varepsilon_i(e'_j)=e'_j, \text{for all $j\neq i$}. \]
Let $S_r$ be the symmetric group of $\{1,\dots, r\}$.
One defines an action of $S_r$ on $V_0$ by 
\[ \sigma (e_i)=e_{\sigma(i)}, \quad \sigma (e'_i)=e'_{\sigma(i)}, \]
and one has $S_r\subseteq\Sp(V_0)$. One verifies that $\sigma \cdot \varepsilon_i \cdot \sigma^{-1}=\varepsilon_{\sigma(i)}$. Let 
\[ H:=\<\varepsilon_1,\dots, \varepsilon_r\> \cdot S_r \] 
be the subgroup generated by all $\varepsilon_i$ and $S_r$. 

Two $k$-subspaces $W$ and $W'$ in $\mathrm{L}$ are said to be \emph{isomorphic} if there is an element $\gamma \in \Sp(V_0)$ such that $W'=\gamma \cdot W$; equivalently, there is an isomorphism $(V_0, \psi_0, W) \simeq (V_0, \psi_0, W')$ of pairs of vector spaces preserving $\psi_0$. In this case, we have
\begin{equation}
    \label{eq:conj_sp}
    \End(V_0,W')=\gamma \End(V_0,W) \gamma^{-1},\quad \Sp(V_0,W')=\gamma \Sp(V_0,W) \gamma^{-1}.
\end{equation}

If $P=[v_1,\dots, v_r]$ is a presenting matrix for $W$ with basis $v_1,\dots, v_r$, then $P'=\gamma P=[v_1',\dots, v_r']$ is a presenting matrix for $W'$ with basis $v_1'=\gamma v_1, \dots , v_r'=\gamma v_r$. 
The operation $P\mapsto \gamma P$ for $\gamma\in \Sp(V_0)$ contains all simultaneous permutations of rows in $U$ and $U'$ respectively, and exchanges of the rows $u_i$ and $u_i'$ up to a sign for all $1\le i \le r$.

\begin{lemma}
    For any $k$-subspace $W\in \mathrm{L}$, there is an element $\gamma\in H$ such that the representative matrix $P'$ of $W'=\gamma W$ with respect to any basis has its first (resp.~last) $r\times r$-block invertible. 
\end{lemma}
\begin{proof} 
    Let $P$ be a representative matrix of $W$. If the first (resp.~last) $r\times r$-block of $P$ is invertible, then so is the last (resp.~first) block of $P'=\varepsilon_1 \cdots \varepsilon_r P$.  
    So it suffices to show the first case. We may assume $U\neq 0$: otherwise $U'$ is invertible and after exchanging all rows $u_i$ and $u_i'$, the first block $U$ is invertible. Let $s=\rank U$; we may assume $s< r$.   After a suitable permutation of rows in $U$ and  column reductions (which amount to a change of basis $\{v_i\}$), we may assume that 
    \[ P=\begin{bmatrix}
        U \\ U'
    \end{bmatrix}, \quad U=\begin{bmatrix}
        \bbI_s & 0 \\
        U_{21} & 0
    \end{bmatrix}, \quad U'=\begin{bmatrix}
        U'_{11} & U_{12}' \\
        U'_{21} & U_{22}'
    \end{bmatrix}, \quad U'_{11}\in \Mat_{s\times s} (k). \]
We claim that $U_{22}'\neq 0$. For suppose that $U_{22}'=0$, that is, $u'_{ij}=0$ for $s+1\le i, j\le r$. 
For each $1\le j \le s$ and each $s+1 \le \ell \le r$, we have 
\begin{equation}
    v_j=e_j+\sum_{i=s+1}^r u_{ij} e_i + \sum_{i=1}^r u_{ij} e'_i, \quad v_\ell=\sum_{i=1}^r u_{i\ell}' e'_i. 
\end{equation}
One computes
\[ 0=\psi_0(v_j, v_\ell)=u_{j\ell}'+\sum_{i=s+1}^r u_{ij} u_{i\ell}'=u_{j \ell}' \quad (\text{as $u_{i\ell}'=0$ for $i\ge s+1$}). \]
Thus, $u_{j \ell}'=0$ for $1\le j\le s$ and $s+1 \le \ell \le r$ and one also has $U_{12}'=0$. This shows that  $\rank P=s$, a contradiction.

Since $U_{22}'\neq 0$, after permuting the last $r-s$ rows of $U$ and the corresponding last $r-s$-rows of $U'$ if necessary, we get a representative matrix $P=\begin{bmatrix}
    U \\ U' 
\end{bmatrix}$ with $\rank U>s$. Continuing this procedure if necessary, we obtain a representative matrix $P$ whose first $r\times r$-block $U$ is invertible. 
\end{proof}

After a suitable translation by $H$, we may represent $W$ by 
$\begin{bmatrix}
    \bbI_r \\
    T
\end{bmatrix}$, where $T=(t_{ij})\in \Mat_{r\times r}(k)$.
Since $W$ is maximal isotropic, the matrix $T$ is symmetric. For $\alpha=\begin{bmatrix}
    A & B \\
    C & D 
\end{bmatrix}\in \End(V_0)=\Mat_{2r}(k_0)$, with $A,B,C,D\in \Mat_r(k_0)$, one has 
\begin{equation}\label{eq:def_eq_symp}
\alpha(W)\subseteq W \iff C+DT-TA-TBT=0. 
\end{equation}
Note that in Section~\ref{sec:GL} we represented $W$ by $\begin{bmatrix}
    T \\
    \bbI_r
\end{bmatrix}$ instead, to find Equation~\eqref{eq:def_eq} instead of \eqref{eq:def_eq_symp}.
Thus, 
\begin{equation}
    \End(V_0,W)=\left \{\begin{bmatrix}
    A & B \\
    C & D 
\end{bmatrix}\in \Mat_{2r}(k_0): C+DT-TA-TBT=0 \, \right \}.
\end{equation}

Let $\Sym_r(k)\subseteq \Mat_r(k)$ be the set of all symmetric matrices. Let $\alpha\mapsto \alpha^\dagger$ denote the symplectic involution on $\End(V_0)$ with respect to $\psi_0$; one has 
\[
\Sp(V_0)=\{\alpha\in \GL(V_0): \alpha^\dagger \alpha=1\}.
\]
Note that $\End(V_0,W)$ is stable under $\dagger$. 
To see this: if $w\in W$, then $\psi_0(\alpha^\dagger(w),W)=\psi_0(w,\alpha W)=0$, so $\alpha^\dagger(w)\in W^\bot=W$. Then one has
\begin{equation}\label{eq:SpVoW}
\Sp(V_0,W)=\{\alpha\in \End(V_0,W)^\times: \alpha^\dagger \alpha=1\}, \end{equation}
the unitary group of $\End(V_0,W)$ with involution $\dagger$.

Let $W_0=W\cap V_0$ and $\wt W_0$ be $k_0$-subspaces defined as in Definition~\ref{def:envelope}.

\begin{lemma}
    One has $W_0=\wt W_0^\bot$.
\end{lemma}
\begin{proof}
   Since $\psi_0(W_0, W)=0$, the space $W_0^\bot \otimes k$ is a $k$-subspace that is defined over $k_0$ and contains $W$. Therefore, $W_0^\bot \supseteq \wt W_0$. Conversely, since $\wt W_{0}\otimes k\supseteq W$, 
   \[ \wt W_0^\bot \subseteq \wt W_0^\bot\otimes k =(\wt  W_{0}\otimes k)^\bot \subseteq W^\bot=W.\]
So $\wt W_0^\bot\subseteq W\cap V_0=W_0$.
\end{proof}

Denote by $\mathrm{L}(V_0)^{\rm d}$ the locus consisting of $k$-subspaces $W \in \mathrm{L}$ such that $\wt W_0 = V_0$ (hence $W_0=0$). 
For any $W\in {\rm L}$, the pairing $\psi_0$ induces a non-degenerate alternating pairing $\psi_0$ on $W_0^\bot/W_0$ for which $\ol W:=W/W_{0,k}$ is maximal isotropic. Thus,  we have the following natural decomposition: 
\begin{equation}
    \mathrm{L}=\mathrm{L}(V_0)(k_0) \coprod_{ W_{0}} \mathrm{L}(W_{0}^\bot /W_{0})^{\rm d}, 
\end{equation}
where $W_0$ runs through all non-maximal isotropic $k_0$-subspaces of $V_0$, and where $\mathrm{L}(V_0)(k_0)$ is the set of maximal isotropic $k_0$-subspaces. Moreover, the map 
$W\mapsto \End(V_0,W)$ (resp.~$W\mapsto \Sp(V_0,W)$) gives a bijection between the set $\mathrm{L}(V_0)(k_0)$ and the set of all maximal parabolic $k_0$-subalgebras of $\End(V_0)$ of type $(r,r)$ (resp.~Siegel parabolic $k_0$-subgroup of $\Sp(V_0)$). Note that if $\alpha\in \Sp(V_0)$ is an element satisfying $\alpha(W_{0})=W_{0}$, then $\alpha(W_{0}^\bot)=W_{0}^\bot$. 

Denote by 
\begin{equation}\label{eq:calE}
  \calE=\calE_{\Sp(V_0)}:=\{ \End(V_0,W): W\in \mathrm{L} \}
\end{equation}
the set of all $k_0$-subalgebras of $\End(V_0)$ that occur as the relative endomorphism algebra of $W$ for some $W\in \mathrm{L}$. 
For each $E\in \calE$, set 
\begin{equation}
    \label{eq:LE}
    \mathrm{L}_E:=\{W\in L: E\subseteq \End(V_0,W) \}, \quad (\mathrm{L}_E)^0:=\{W\in \mathrm{L}: E = \End(V_0,W) \}.
\end{equation}
Then $(\mathrm{L}_{E})^0$ is non-empty, and we have 
\begin{equation}
   \mathrm{L}=\coprod_{E\in \calE} (\mathrm{L}_E)^0, \quad \mathrm{L}_E= \coprod_{E \subseteq E'}(\mathrm{L}_{E'})^0.
\end{equation}
When $k_0$ is finite, each stratum $(\mathrm{L}_E)^0$ is a quasi-projective subvariety. 

Similarly, denote by $[\calE]$ the set of $\Sp(V_0)$-conjugacy classes of all $k_0$-subalgebras $E$ in $\calE$ and by $[E]\subseteq \calE$ the $\Sp(V_0)$-conjugacy class in $\calE$.  For each $[E]$ in $[\calE]$, set 
\begin{equation}
    \mathrm{L}_{[E]}:=\bigcup_{E\in [E]} \mathrm{L}_E, \quad (\mathrm{L}_{[E]})^0:=\bigcup_{E\in [E]} (\mathrm{L}_{E})^0.
\end{equation}
Then we have a stratification by the strata $(\mathrm{L}_{[E]})^0$ with ``closure'' relation 
\begin{equation}\label{eq:st_L_[E]}
    \mathrm{L}=\bigcup_{[E]\in [\calE]} \mathrm{L}_{[E]},\quad  \mathrm{L}_{[E]}=\coprod_{[E]\subseteq [E']} (\mathrm{L}_{[E']})^0.
\end{equation}
When $k_0$ is finite, each stratum $(\mathrm{L}_{[E]})^0$ is a quasi-projective subvariety.

\begin{definition}
    Let $W\in \mathrm{L}$ and $\End(V_0,W)$ be the endomorphism algebra of $W$ over $k_0$. Let $\wt W_0$ be the envelope of $W$, let $W_0 = W \cap V_0$, and let 
    $\ol \psi_0: W_0^\bot / W_0 \times W_0^\bot / W_0 \to k_0$ 
    be the induced non-degenerate alternating pairing. 
    The \emph{reduced index} of $\Sp(V_0,W)$ in $\Sp(V_0)$ is defined by
    \[ \mathrm{ind}(W):=[\Sp(W_0^\bot / W_0):\Sp(W_0^\bot / W_0, \ol W)], \]
    where $\ol W=W/W_{0,k}$.
\end{definition}

\begin{lemma}
    Let $W\in \mathrm{L}$, and $E=\End(V_0,W)$. Then the envelope $\wt W_0$ and $\End(W_0^\bot / W_0,\ol W)$ are uniquely determined by $E$. Moreover, the reduced index $\mathrm{ind}(W)$ is uniquely determined by $E$.
\end{lemma}

Let $k_0[\ul X]$ be the polynomial ring over $k_0$ with variables $X_{ij}$ for $1\le i\le j \le r$ and $k_0[\ul X]_{\le d}$ be the $k_0$-vector subspace of polynomials of total degree $\le d$. Set $N_d:=\dim_{k_0} k_0[\ul X]_{\le d}$. For any $T=(t_{ij})\in \Sym_{r}(k)$, let ${\rm ev}_T:k_0[\ul X]\to k$ be the evaluation map at $T$. 
Set 
\begin{equation}
    \bbV(T):={\rm ev}_T(k_0[\ul X]_{\le 2}), \quad d(T):=\dim_{k_0} \bbV(T).
\end{equation}
Note that $N_2=d(T)$ if and only if the vectors  
$1, t_{ij}, t_{i_1 j_1} t_{i_2 j_2} $, for $1\le i\le j \le r$, $1\le i_1\le j_1 \le r$ and $1\le i_2\le j_2 \le r$ satisfying the lexicographic order $(i_1,j_1)\le (i_2,j_2)$, are $k_0$-linearly independent.   

Proposition~\ref{prop:End=k_0_Sp}, Theorem~\ref{thm:L:nonempty} and Proposition~\ref{prop:endo2} below are the respective symplectic analogues of Proposition~\ref{prop:End=k_0}, Theorem~\ref{thm:nonempty} and Proposition~\ref{prop:endo}.

\begin{proposition}\label{prop:End=k_0_Sp}
   Let $T\in \Sym_{r}(k)$ and $W_T=\big\<\begin{bmatrix}
    \bbI_r \\ T
    \end{bmatrix}\big\>$ the corresponding $k$-subspace of $\mathrm{L}$. If $d(T)=N_2$, then $\End(V_0,W_T)=k_0$. 
\end{proposition} 
\begin{proof}
    For $1\le i\le j\le r$, set \[ S_{ij}:=\begin{cases}
    E_{ij}+E_{ji} & \text{if $i<j$}; \\
    E_{ii} & \text{if $i=j$},
    \end{cases} \]
    then one has $T=\sum_{i\le j} t_{ij} S_{ij}$.
    Using Equation \eqref{eq:def_eq_symp}, we get 
\begin{equation} \label{eq:def-eq3}
   C+ \sum_{\substack{1\le i \le j \le r}}  (D S_{ij}-S_{ij}A)t_{ij}-\sum_{\substack{1\le i_1\le j_1 \le r \\ 1\le i_2\le j_2 \le r\\(i_1, j_1)\le (i_2,j_2)}} (S_{i_1 j_1} B S_{i_2j_2}) t_{i_1 j_1} t_{i_2 j_2} =0. 
\end{equation}

Since $N_2=d(T)$, the vectors $1, t_{ij}, t_{i_1 j_1} t_{i_2j_2}$ for all $i,j,i_1,j_1,i_2,j_2$ are $k_0$-linearly independent. So 
\[ C=0, \quad DS_{ij}-S_{ij}A =0, \quad 
\begin{cases}
  S_{i_1 j_1} B S_{i_2j_2}=0, & \text{if $(i_1,j_1)=(i_2,j_2)$}; \\
  S_{i_1 j_1} B S_{i_2j_2}+S_{i_2 j_2} B S_{i_1 j_1}=0 & \text{otherwise.}
\end{cases}\]

For $i_1=j_1$ and $i_2=j_2$ and $i_1\neq i_2$, one gets
\[ S_{i_1 i_1} B S_{i_2 i_2}+S_{i_2 i_2} B S_{i_1 i_1}=b_{i_1 i_2} E_{i_1 i_2}+b_{i_2 i_1} E_{i_2 i_1}=0, \] so $b_{i_1 i_2}=0$ for all $1\le i_1\neq i_2\le r$.
For $i_1=j_1=i_2=j_2$, one gets $S_{i_1 i_1} B S_{i_1 i_1}=b_{i_1 i_i} E_{i_1 i_1}=0$, so $b_{i_1 i_1}=0$ for $1\le i_1\le r$. This shows $B=0$. 

For $i=j$, one gets

\[ DS_{ii}=\sum_{1\le i_1\le r} d_{i_1, i} E_{i_1,i}, \]
\[ S_{ii}A=\sum_{1\le j_1\le r} a_{i j_1}
E_{i j_1}. \]
Thus $a_{ij}=d_{ij}=0$ for $i<j$ or $i>j$, and $a_{ii}=d_{ii}$ for all $1\le i\le r$.

It follows that $\begin{bmatrix}
    A & B \\
    C & D
\end{bmatrix}=a\bbI_n$ for some $a\in k_0$. This proves the proposition.
\end{proof}

\begin{theorem}\label{thm:L:nonempty}
    If $k/k_0$ is infinite, then $(\mathrm{L}_{[k_0]})^0$ and hence $\mathrm{L}(V_0)^{\rm d}\supseteq (\mathrm{L}_{[k_0]})^0$ are non-empty. 
\end{theorem}
\begin{proof}
    As in the proof of Theorem~\ref{thm:nonempty}, this follows from Theorem~\ref{thm:evT} and Proposition~\ref{prop:End=k_0_Sp}. 
\end{proof}

\begin{proposition}\label{prop:endo2}
    Let $\{t_1'=1, t_2',\dots, t_s'\}\subseteq \{1,t_{ij}, t_{i_1j_1} t_{i_2 j_2}: 1\le i\le j \le r$, $1\le i_1\le j_1 \le r$ and $1\le i_2\le j_2 \le r \text{ such that } (i_1,j_1)\le (i_2,j_2) \}$ be a maximal $k_0$-linearly independent subset. Write 
\begin{equation}\label{eq:lin-comb2}
    t_{ij}=\sum_{\mu=1}^s \alpha_{ij}^\mu t_\mu', \quad t_{i_1j_1} t_{i_2 j_2}=\sum_{\mu=1}^s \alpha_{i_1j_1 i_2 j_2}^\mu t_\mu', \quad \text{with} \ \alpha_{ij}^\mu,\  \alpha_{i_1j_1 i_2 j_2}^\mu\in k_0. 
\end{equation}
Then $\End(V_0,W)$ consists of all matrices 
$   \begin{bmatrix}
    A & B \\
    C & D
\end{bmatrix}\in \Mat_n(k_0)$ satisfying
\begin{equation} \label{eq:endo3}
   C \delta_{1,\mu}+ \sum_{\substack{1\le i \le j \le r}}  (D S_{ij}-S_{ij}A) \alpha_{ij}^\mu -\sum_{\substack{1\le i_1\le j_1 \le r \\ 1\le i_2\le j_2 \le r\\(i_1, j_1)\le (i_2,j_2)}} (S_{i_1 j_1} B S_{i_2j_2}) \alpha_{i_1 j_1i_2 j_2}^\mu =0,
\end{equation}
for all $\mu=1,\dots, s$, where $\delta_{1,\mu}$ again denotes the Kronecker delta.
\end{proposition}
\begin{proof}
Equation~ \eqref{eq:endo3} follows from Equations~\eqref{eq:def-eq3} and~\eqref{eq:lin-comb2}.
\end{proof}

\section{Examples and applications}

\subsection{Endomorphism algebras of complex tori}\

Let $X=V/\Lambda$ be a complex torus, where~$V$ is a $\C$-vector space of dimension $n$ and $\Lambda$ is a lattice in $V$. We regard $V$ as a real vector space together with complex structure $J\in \End_{\R}(V)$, so $J^2=-\bbI_{2n}$. We have a decomposition $V_\C=V^{-1,0}\oplus V^{0,-1}$ into eigenspaces on which $J$ acts by the multiplication by $i$ on $V^{-1,0}$ (resp.~$-i$ on $V^{0,-1}$). Put $V_0:=\Lambda\otimes \Q$ so we have $V=V_0\otimes_\Q \R$. The endomorphism algebra of $X$ is 
\[ \End^0(X)=\{\alpha\in \End(V_0): \alpha \cdot J=J \cdot \alpha \text{ on $V$} \  \}. \]
\begin{lemma}\label{lm:end_cpxtori}
    We have $\End^0(X)=\End(V_0,V^{-1,0}).$
\end{lemma}
\begin{proof}
    Suppose $\alpha\in \End^0(X)$. Since $\alpha$ commutes with $J$, for any $v\in V^{-1,0}$, we have
\[ J \alpha(v)=\alpha (J v)=i \alpha(v), \text{\ and then\ }  \alpha(v)\in V^{-1,0}.\]    
This shows $\End^0(X)\subseteq \End(V_0,V^{-1,0})$. Conversely, suppose $\alpha\in \End(V_0,V^{-1,0})$. Extending to $\C$, we have $\alpha \cdot J-J\cdot \alpha=0$ on $V^{-1,0}$. Taking the complex conjugation $c$ we obtain 
$c(\alpha \cdot J-J\cdot \alpha)=0$ on $V^{0,-1}$.
Since $\alpha, J\in \End(V)$ and $\End(V_\C)=\End(V)\otimes \C$, the maps $\alpha$ and $J$ are invariant under $c$. Thus, we have $\alpha \cdot J-J\cdot \alpha=0$ on $V_\C$ and therefore $\alpha\in \End^0(X)$. 
\end{proof}

\subsection{The case $\boldsymbol{k_0=\R}$ and $\boldsymbol{k=\C}$.}\label{sec:end_R}\ 

Let $V_0$ be a real vector space of dimension $n$.
\begin{proposition}\label{prop:real}
    Let $W\in \Gr(V_0,r)$, $W_0 := W \cap V_0$, $r_0:=\dim_\R W_0$ and $r_2:=n-(2r-r_0)$. Then $\dim_\R \wt W_0=2r-r_0$ and 
\begin{equation}\label{eq:real}
    \End(V_0,W)\simeq \begin{bmatrix}
        \Mat_{r_0}(\R) & * & * \\
        0 & \Mat_{r-r_0}(\C) & * \\
        0 & 0 &  \Mat_{r_2}(\R)\\
    \end{bmatrix}.
\end{equation}   
Here $\Mat_{r-r_0}(\C)\subseteq \Mat_{2(r-r_0)}(\R)$ is viewed as an $\R$-subalgebra and $*=\Mat_{a\times b}(\R)$ for suitable integers $0\le a, b\le n$.    
\end{proposition}
\begin{proof}
   Observe that $\wt W=W+c(W)$, where $c$ denotes complex conjugation (on~$V_{0,\mathbb{C}}$) as in Lemma~\ref{lm:end_cpxtori}, and that $W_{0, \C}=W\cap c(W)$ is the largest $\C$-subspace contained in $W$ which is defined over $\R$.
   We also have $\dim_\R \wt W_0+\dim_\R W_0=2r$. Put $r_0=\dim_\R W_0$; then $\dim \wt W_0=2r-r_0$ and $\dim_\R \wt W_0/W_0=2(r-r_0)$. Set $V_1:=\wt W_0/W_0$ and $\ol W:=W/W_{0,\C}$. Then by Equation~\eqref{eq:endo.2} we obtain \eqref{eq:real} with middle block in $\End(V_1,\ol W)$. So it suffices to show $\End(V_1,\ol W)\simeq \Mat_{r-r_0}(\C)$. 

   Since $\ol W+c(\ol W)=V_{1,\C}$ and $\ol W \cap c(\ol W)=0$, we have $V_{1,\C}=\ol W\oplus c(\ol W)$. Define a complex structure $J$ on $V_1$ by letting $J$ act by multiplication by $i$ on $\ol W$ (resp.~$-i$ on $c(\ol W)$). By the argument of Lemma~\ref{lm:end_cpxtori}, we have 
   \[ \End(V_1,\ol W)=\End_\R (V_1, J)\simeq \Mat_{r-r_0}(\C). \]
   This proves the proposition. 
\end{proof}

\subsection{Endomorphism rings of supersingular abelian varieties} \label{sec:end_ssav}\

Let $k$ be an algebraically closed field of \ch $p>0$, where $p$ is a prime number. Let $X$ be a supersingular abelian variety over $k$ of dimension $n>1$. Then the endomorphism algebra $\End^0(X)$ of $X$ is isomorphic to $\Mat_n(B_{p,\infty})$, where $B_{p,\infty}$ denotes the unique definite quaternion $\Q$-algebra of discriminant $p$. Choose a maximal order $O\subseteq B_{p,\infty}$ that contains an element $\Pi$ satisfying $\Pi^2=-p$. Since the algebra $\Mat_n(B_{p,\infty})$ has class number one, every maximal order is conjugate to the maximal order $\Mat_n(O)$. 

Let $\varphi:\wt X\to X$ be the minimal isogeny of $X$ in the sense of \cite[Lemma~1.8]{lioort}, i.e. $\wt X$ is a superspecial abelian variety over $k$ such that any other isogeny from a superspecial abelian variety to~$X$ factors through $\varphi$. Then $\End(X)\subseteq \End(\wt X)$ and $\End(\wt X)$ is a maximal order of $\End^0(\wt X)\simeq \Mat_n(B_{p,\infty})$. For simplicity, we choose an isomorphism $\End(\wt X)\simeq \Mat_n(O)$ and make identifications so that $\End(X)\subseteq \Mat_n(O)=\End(\wt X)$. For any  prime $\ell\neq p$, we have $\End(X)\otimes \Z_\ell=\Mat_n(O_\ell)$, where $O_\ell:=O\otimes \Z_\ell$.

Let $M\subseteq \wt M$ be the (contravariant) \dieu modules of $X$ and $\wt X$, respectively. We denote by $\sfF$ and $\sfV$ the Frobenius and Verschiebung operators on $\wt M$. One has 
\begin{equation}\label{eq:end_DM}
   \End(\wt X)\otimes \Zp = \Mat_n(O_p)\simeq \End_{\rm DM}(\wt M)^{\rm opp},\quad  
\End(X)\otimes \Zp\simeq \End_{\rm DM}(M)^{\rm opp}, 
\end{equation}
where $O_p:=O\otimes \Zp$ and the superscript ${\rm opp}$ denotes the opposite ring (with multiplication $a\circ b:=b\cdot a$). The ring $\Mat_n(O)$ carries the canonical involution $A\mapsto A^*=(\ol{A})^T$ and this gives an isomorphism $\Mat_n(O)\simeq \Mat_n(O)^{\rm opp}$ (noting that $A_1 \cdot A_2 \mapsto (A_1\cdot A_2)^*=A_2^*\cdot A_1^*=A_1^*\circ A_2^*$). Composing with this canonical involution, one has
\[ \End(\wt X)\otimes \Zp = \Mat_n(O_p)\simeq \End_{\rm DM}(\wt M),\quad  
\End(X)\otimes \Zp\simeq \End_{\rm DM}(M). \] 
Thus, to find $\End(X)$ one is reduced to determining the endomorphism ring $\End_{\rm DM}(M)$ of $M$.

We consider the special case where $\sfV \wt M\subseteq M \subseteq \wt M$, or equivalently, $\ker \varphi\subseteq \wt X[F]$, where $F:\wt X\to \wt X^{(p)}$ is the relative Frobenius morphism on $\wt X$. Put $W:= M/\sfV \wt M$, and 
\[ V_0:=\wt M^\diamond/\sfV \wt M^\diamond, \quad \text{ where } \wt M^\diamond:=\{m\in \wt M: \sfF^2 m=-p m\}.  \]
Then $V_0$ is an $n$-dimensional vector space over $\F_{p^2}$ and $W\subseteq V_{0,k}$ is a $k$-vector subspace.
Recall \[ \End(V_0,W):=\{\alpha\in \End(V_0): \alpha(W)\subseteq W\}. \]
Let $m_{\sfV}: \End_{\rm DM}(\wt M)\to \End(V_0)$ be the natural projection map. We have
\begin{equation}\label{eq:EndM}
    \End_{\rm DM}(M)=\{\alpha\in \End_{\rm DM}(\wt M): m_{\sfV}(\alpha)\in \End(V_0,W)\}. 
\end{equation}
That is, $\End_{\rm DM}(M)$ is the pre-image of $\End(V_0,W)$ under the map $m_{\sfV}$. From this, one reduces the problem of finding $\End(X)$ to computing the relative endomorphism algebra $\End(V_0,W)$. We summarise this as follows:

\begin{proposition}
    Let $\varphi: \wt X\to X$ be the minimal isogeny of a supersingular abelian variety~$X$ with respective \dieu modules $M\subseteq \wt M$. With the inclusions $\End(X)\subseteq \End(\wt X)\subseteq \End_{\rm DM}(\wt M)$ and $\End(X) \subseteq \End_{\rm DM}(M)\subseteq \End_{\rm DM}(\wt M)$, we have 
    \begin{equation}\label{eq:endX}
     \End(X)=\End(\wt X)\cap \End_{\rm DM}(M)\quad \text{and} \quad [\End(\wt X):\End(X)]=[\End_{\rm DM}(\wt M): \End_{\rm DM}(M)].    
    \end{equation}

    Moreover, if $\ker \varphi\subseteq \wt X[F]$, then 
    \begin{equation}
    \End(X)=\{\alpha\in \End(\wt X): m_{\sfV}(\alpha)\in \End(V_0,W)\}. 
\end{equation}
In particular, $[\End(\wt X):\End(X)]=[\End(V_0):\End(V_0,W)]$.
\end{proposition}

\subsection{Automorphism groups and masses of polarised supersingular abelian varieties}\label{sec:auto}\ 

Let $(X,\lambda)$ be a principally polarised supersingular abelian variety over $k$, and let $\varphi:(\wt X,\wt \lambda)\to (X,\lambda)$ be the minimal isogeny, where $\wt \lambda=\varphi^* \lambda$. Note that $\wt \lambda$ has $p$-power degree.
Let $\ul M=(M\,\<\, , \>)\subseteq \ul {\wt M}=(\wt M, \<\, ,\>)$ be the corresponding \dieu modules with quasi-polarisation. Then 
\[ \Aut_{\rm DM}(\ul M)=\{\alpha\in \Aut_{\rm DM}(\ul {\wt M}): \alpha(M)=M \} \]
and
\begin{equation}\label{eq:AutXLminisog} \begin{split}
    \Aut(X,\lambda)&=\{\alpha\in \Aut(\wt X,\wt \lambda): \alpha(M)=M \} \\
                   &=\Aut(\wt X,\wt \lambda)\cap \Aut_{\rm DM} (\ul M). 
\end{split} \end{equation}
This shows that in general the computation of $\Aut(X,\lambda)$ can be reduced to computing the groups $\Aut(\wt X,\wt \lambda)$ and $\Aut_{\rm DM} (\ul M)$. Unfortunately, as far as we know, computing either one is complicated. A small trick is observing that $\Aut(X,\lambda)$ is a finite subgroup of $\Aut_{\rm DM} (\ul M)$. So we can bound the size of $\Aut(X,\lambda)$ by the sizes of maximal finite subgroups of $\Aut_{\rm DM} (\ul M)$. In particular, if the torsion subset $\Aut_{\rm DM} (\ul M)_{\rm tors}$ is equal to $\{\pm 1\}$, then $\Aut(X,\lambda)=\{\pm 1\}$.

Concerning the computation of $\Aut(\wt X,\wt \lambda)$, one can change the problem to its arithmetic counterpart. Choose a supersingular elliptic curve $E$ over $k$ such that $\End(E)=O$. The functor $\wt X\mapsto \Hom(E,\wt X)$ induces an equivalence of categories between the category of polarised superspecial abelian varieties over $k$ and the category of positive-definite Hermitian right $O$-lattices; see~\cite[Corollary 4.9]{ibukiyama-karemaker-yu}. Therefore, if $(L,h)$ is the Hermitian $O$-lattice corresponding to $(\wt X, \wt \lambda)$ under this equivalence, then we have
\[ \Aut(\wt X, \wt \lambda)\simeq \Aut(L,h). \]
Kirschmer~\cite[Chapter 9]{KirschmerHab} gives a complete (finite) list of the automorphism groups $\Aut(L,h)$
for all positive-definite Hermitian right $O$-lattices
$(L,h)$ with class number one or two.

Let $x=(X,\lambda)$ and denote by  
$\Lambda_{x}$ the set of isomorphism classes of polarised supersingular abelian varieties $(X',\lambda')$ over $k$ such that $(X',\lambda')[\ell^\infty]\simeq (X,\lambda)[\ell^\infty]$ for all primes $\ell$.  
This is a finite set and the mass of $\Lambda_x$ is defined by 
\[ \Mass(\Lambda_x):=\sum_{(X',\lambda')\in \Lambda_x} \frac{1}{\vert \Aut(X',\lambda') \vert}.\]
For the minimal isogeny $\wt x=(\wt X, \wt \lambda)$ we define the finite set $\Lambda_{\wt x}$ and the mass $\Mass(\Lambda_{\wt x})$ in the same way. Then we have, cf.~\cite[Proposition 2.12]{karemaker-yobuko-yu},
\begin{equation}
    \label{eq:mass_relation}
    \Mass(\Lambda_{x})=\Mass(\Lambda_{\wt x})\cdot [\Aut_{\rm DM}(\ul{\wt M}):\Aut_{\rm DM}(\ul M)].
\end{equation}

In what follows, we assume that 
\begin{equation}\label{eq:asswtlambda}
    \ker \wt \lambda=\alpha_p^{2c}\subseteq \wt X[F],\quad \text{for some $c\ge 0$.}
\end{equation}
We have the inclusion $M\subseteq \wt M$. Also $M$ is self-dual with respect to the quasi-polarisation $\<\, , \>$, since $\lambda$ is a principal polarisation on $X$. So we have $\wt M^t \subseteq M\subseteq \wt M$, where ${\wt M}^t$ is the dual lattice of $\wt M$ with respect to $\<\, , \>$. Let $\wt M^{\diamond}$ be the skeleton of $\wt M$ defined by $\sfF^2+p=0$. 
Set 
\begin{equation}
    \label{eq:wtW0}
    V_0^{\rm min}:=\wt M^{\diamond} / ({\wt M}^t)^\diamond,
\end{equation}
which is a $2c$-dimensional vector space over $\F_{p^2}$. Moreover, the pairing $p\<\, , \>: \wt M\times \wt M\to W(k)$ induces a non-degenerate alternating pairing
\begin{equation}\label{eq:psi0}
   \psi^{\rm min}_0: V_0^{\rm min} \times V_0^{\rm min} \to \F_{p^2}. 
\end{equation}
Inside $V_{0,k}^{\rm min}=\wt M / {\wt M}^{t}$ there is a $c$-dimensional subspace 
\[ \ol M:=M/{\wt M}^t\subseteq  V_{0,k}^{\rm min}\]
which is isotropic with respect to the symplectic pairing $\psi_0^{\rm min}$. We have $\Aut_{\rm DM}(\wt M, \<\, , \>)=\Aut_{\rm DM}(\wt M,p \<\, , \>)=\Aut_{\rm DM}(\wt M^{\diamond},p \<\, , \>)$ and a surjective map $\wt M^{\diamond} \to V_0^{\rm min}$.
\begin{lemma}\label{lm:surj}
    The surjection $\wt M^{\diamond} \to V_0^{\rm min}$ induces a surjection
\[ m: \Aut_{\rm DM}({\wt M},p\<\,,\>) \to \Sp(V_0^{\rm min}).
\]
\end{lemma}

\begin{proof}
Let $J$ be the algebraic group over $\Qp$ of the automorphism group of the quasi-polarised rational \dieu module $(\wt M\otimes W[1/p], \<\, , \>)$ \cite[Proposition 1.12]{RZ}. The group $J$ is an inner form of $\Sp_{2g}$ over $\Qp$ \cite[Remark 1.15]{RZ}, and hence it is semi-simple and simply connected. 
Observe that $\Aut_{\rm DM}(\wt M,\<\, , \>)$ is a parahoric subgroup of $J(\Qp)=\Aut(\wt M\otimes W[1/p], \<\, , \>)$, since $\<\, ,\>$ is a polarisation of parahoric type by assumption and since $J$ is simply connected; see the proof of Lemma 5.2 of \cite{ibukiyama-karemaker-yu} and \cite[Theorem 3.13]{platonov-rapinchuk}. 
By Bruhat-Tits theory~\cite{BT2}, there exists a connected smooth group scheme $\calJ$ over $\Zp$ with generic fibre $J$ such that $\Aut_{\rm DM}(\underline{\wt M})=\calJ(\Zp)$. Then the map $m$ is given by the composition
\[ m: \calJ(\Zp) \to \calJ(\Fp) \to (\calJ\otimes \Fp)^{\rm rdt}(\Fp)=\Sp(V_0^{\rm min}) \times U_{g-2c}(\Fp) \to \Sp(V_0^{\rm min}),\]
where $(\calJ\otimes \Fp)^{\rm rdt}$ is the maximal reductive quotient of $\calJ\otimes \Fp$ and $U_{g-2c}$ is the unitary group on the $\F_{p^2}$-vector space $(\wt M^\diamond)^t/ \sfV \wt M^\diamond$ of dimension $g-2c$ with the Hermitian form induced by $(x,y):=\<x,Fy\>$. The surjectivity of the first map follows from the smoothness of $\calJ$ over $\Zp$. To show the surjectivity of the second map, it suffices to show  
that the kernel $U$ of $\calJ\otimes \Fp \to (\calJ\otimes \Fp)^{\rm rdt}$ satisfies $H^1(\Fp, U)=0$.
Since~$U$ is unipotent, it is a successive extension of additive groups $\bbG_a$, and the desired vanishing follows from $H^1(\Fp, \bbG_a)=0$ and induction on the dimension of $U$.    
\end{proof}

Recall from Equation~\eqref{eq:SpW} that 
\[ \Sp(V_0^{\rm min},\ol M):=\{ \alpha\in \Sp(V_0^{\rm min}): \alpha(\ol M)=\ol M \}.\]
On the other hand, we have
\begin{equation}
    \label{eq:loc_index}
    \Aut_{\rm DM}(\underline{M})=\{\alpha\in \Aut_{\rm DM}(\underline{\wt M}): m(\alpha)\in \Sp(V_0^{\rm min},\ol M) \}.
\end{equation}
Since the map $m$ is surjective, one has 
\begin{equation}\label{eq:loc_index2}
    [\Aut_{\rm DM}(\underline{\wt M}):\Aut_{\rm DM}(\underline{M})] =[\Sp(V_0^{\rm min})
    :\Sp(V_0^{\rm min}, \ol M)]. 
\end{equation}
Let $\Lambda_{g, p^c}$ be the set of isomorphism classes of polarised superspecial abelian varieties $(X',\lambda')$ of dimension $g$ with $\ker \lambda'\simeq \alpha_p^{2c}$. 
By the assumption Equation~\eqref{eq:asswtlambda} one identifies $\Lambda_{\wt x}$ with $\Lambda_{g, p^c}$, and we have
\begin{equation}
    \label{eq:Lambda_gpr}
    \Mass(\Lambda_{\wt x})=\Mass(\Lambda_{g,p^c}).
\end{equation}
By \cite[Theorem 3.1]{ibukiyama-karemaker-yu}, we have the mass formula
\[ \Mass(\Lambda_{g,p^c})=\prod_{i=1}^g \frac{|\zeta(1-2i)|}{2} \cdot  L_{g,p^c},  \]
where     
  \begin{equation}
    \label{eq:Lgpc}
      L_{g,p^c} =\prod_{i=1}^{g-2c} (p^i + (-1)^i)\cdot \prod_{i=1}^c
  (p^{4i-2}-1) 
 \cdot \frac{\prod_{i=1}^g
  (p^{2i}-1)}{\prod_{i=1}^{2c}(p^{2i}-1)\prod_{i=1}^{g-2c} (p^{2i}-1)}.
  \end{equation}
By Equations~\eqref{eq:mass_relation}, \eqref{eq:loc_index2} and \eqref{eq:Lambda_gpr}
we obtain the following result.
\begin{proposition}
    \label{prop:mass_Lambdax}
    Let $\varphi: (\wt X,\wt \lambda)\to (X,\lambda)$ be the minimal isogeny of a principally polarised supersingular abelian variety~$x=(X,\lambda)$ with respective quasi-polarised \dieu modules $(M, \<\, ,\>)\subseteq (\wt M, \<\, ,\>)$. Assume that 
    $\ker \wt \lambda=\alpha_p^{2c}\subseteq \wt X[F]$ for some $c\ge 0$. Let $(V_0^{\rm min},\psi_0^{\rm min})$ be the $2c$-dimensional symplectic space over $\F_{p^2}$ defined in Equations~\eqref{eq:wtW0} and~\eqref{eq:psi0}, and let $\ol M:=M/\wt M^t\subseteq V_{0,k}^{\rm min}$ be the image of $M$. Then 
\begin{equation} \label{eq:mass_Lambda_x}
    \Mass(\Lambda_x)=\prod_{i=1}^g \frac{|\zeta(1-2i)|}{2} \cdot [\Sp(V_0^{\rm min}):\Sp(V_0^{\rm min}, \ol M)] \cdot L_{g,p^c}
\end{equation}
where $L_{g,p^c}$ is the local factor defined in Equation~\eqref{eq:Lgpc}.    
\end{proposition}

\section{Proofs of Theorems~\ref{thm:main} and~\ref{thm:oc}}

\subsection{Supersingular EO strata}\

Let $g\ge 2$ and $\calA_g$ be the moduli space over $\Fpbar$ of $g$-dimensional principally polarised abelian varieties. 
Let $\calS_g$ be the supersingular locus of $\calA_g$. As ever, let $k$ be an algebraically closed field of characteristic~$p$.

A $\mathrm{BT}_1$ of rank $2g$ over $k$ is a finite group scheme $G$ of order $p^{2g}$ such that the following two sequences are exact
\[ \begin{tikzcd}
    G \arrow[r, "F_{G/k}"] & G^{(p)} \arrow[r, "V_{G/k}"] & G, \quad G^{(p)} \arrow[r, "V_{G/k}"] & G \arrow[r, "F_{G/k}"] & G^{(p)}
\end{tikzcd},\]
where $F_{G/k}$ and $V_{G/k}$ are the Frobenius and Verschiebung morphisms, respectively.
A $\mathrm{BT}_1$ $G$ of rank $2g$ over $k$ together with an isomorphism $\lambda: G\to G^D$ satisfying $\lambda^D=-\lambda$, where $G^D$ denotes the Cartier dual of $G$, is called a \emph{polarised} $\mathrm{BT}_1$, denoted $(G,\lambda)$.

Following \cite{OortEO}, an \emph{elementary sequence} of length $g$ is a function $\varphi:\{0,1,\dots, g\}\to \{0,1,\dots, g\}$ such that 
\[ \varphi(0)=0, \quad \varphi(i)\le \varphi(i+1) \le \varphi(i)+1, \ \forall\, 0\le i\le g-1. \]
It is shown in \emph{loc.~cit.~}that there is a bijection between the set of isomorphism classes of polarised $\mathrm{BT}_1$ of rank $2g$ over $k$ and the set $\Phi$ of elementary sequences of length $g$. For each $(G,\lambda)$, we denote by $\varphi(G,\lambda)$ the associated elementary sequence. For each $\varphi\in \Phi$, its associated Ekedahl-Oort (EO) stratum is defined by
\[ S_\varphi:=\{(X,\lambda)\in \calA_g(k): \varphi(X[p],\lambda)=\varphi\}. \]
This yields the EO stratification:
\[ \calA_g=\coprod_{\varphi\in \Phi} S_\varphi,\]
where each stratum $S_\varphi$ is quasi-affine and of equi-dimension $|\varphi| := \sum_{i=1}^g \varphi(i)$, cf.~\cite[Theorem 1.2]{OortEO}.
Set 
\[ \Phi^{\rm ss}:=\{\varphi\in \Phi: S_\varphi\subseteq \calS_g\}.\]
The set $\Phi^{\rm ss}$ consists of $\varphi\in \Phi$ such that 
$\varphi(g-c)=0$ for $c\ge \lfloor g/2 \rfloor$, cf.~
\cite[Thorem 4.8, Step 2]{COirr}, \cite[Remark 2.4.3]{harashita:SSEO}.
For each $\varphi\in \Phi^{\rm ss}$, let $c(\varphi)$ be the smallest integer $c$ such that $\varphi(g-c)=0$ and call it the \emph{conductor} of $\varphi$.
Let
\[ \calS_g^{\rm eo}:=\bigcup_{\varphi\in \Phi^{\rm ss}} S_\varphi \]
be the union of the supersingular EO strata. The closure relation for EO strata determines a partial order on $\Phi$, under which $\varphi' \prec \varphi$ if and only if $\varphi'(i) \leq \varphi(i)$ for all $i = 0, \ldots, g$, and where $\varphi' \prec \varphi$ implies that $S_{\varphi'} \subseteq \overline{S}_{\varphi}$, where the bar denotes the Zariski closure. Using this order, in $\Phi^{\rm ss}$, there is a unique maximal element $\varphi_{\max}$, given by
\[ \varphi_{\max}(g-i)=\lfloor g/2 \rfloor-i, \quad i=0,\dots, \lfloor g/2 \rfloor.\]
Then we have $\calS_g^{\rm eo}=\ol S_{\varphi_{\max}}$.

We shall describe the subvariety $\calS_g^{\rm eo}$ by constructing models for its irreducible components. To do so, we recall basic properties on finite flat group schemes from \cite{li, lioort}. 
Let $S$ be a base scheme over $\Fp$, which is assumed to be locally noetherian. Recall that a finite flat commutative group scheme $G$ over $S$ is said to be an \emph{$\alpha$-group} if $F_{G/S}=0$ and $V_{G/S}=0$. Every $\alpha$-group is locally in $S$ for the Zariski topology isomorphic to $\alpha_p^r$ for some integer $r$, called the \emph{rank} of $G$, which is a locally constant function on $S$.

Let $\pi:G\to S$ be an $\alpha$-group over $S$ of rank $r$, and denote by $m:G\times G\to G$, $e:S\to G$, and $\iota:G\to G$ the multiplication map, the zero section and the inverse map, respectively. 
Let~$\calO_G$ be the structure sheaf of $G$. 
The \emph{$\alpha$-sheaf} of $G$, denoted $\calA(G)$, is defined to be the subsheaf of $\pi_* \calO_G$ that consists of local sections $s\in \pi_* \calO_G$ satisfying 
\begin{equation}
    \label{eq:primitve}
    m^*(s)=s\otimes 1 + 1\otimes s. 
\end{equation}
Then $\calA(G)$ is a locally free sheaf of $\calO_S$-modules on $S$ of rank $r$. It is equal to the hom sheaf
\begin{equation}
    \label{eq:homsheaf} 
    \ul {\Hom}(G,\alpha_{p,S}): (T\to S) \mapsto \Hom_{T}(G\times_S T, \alpha_{p,S} \times_S T). 
\end{equation}
\begin{lemma}[{cf.~\cite{li}, \cite[2.4]{lioort}}]
   The functor $G\mapsto \calA(G)$ induces an anti-equivalence of categories between the category of $\alpha$-groups over $S$ and that of locally free sheaves of $\calO_S$-modules. Moreover, it is compatible with Cartier duality.   
\end{lemma}

We fix a supersingular elliptic curve $E_0$ over $\F_{p^2}$ whose Frobenius endomorphism $\pi_{E_0}$ satisfies $\pi_{E_0}=-p$, and set $E:=E_0\otimes_{\F_{p^2}} k$. The existence of $E_0$ is guaranteed by Honda-Tate theory, cf.~\cite[1.2]{lioort}.

\begin{proposition}
[{\cite[Proposition 3.1.5]{harashita:SSEO}} ]    \label{COH}
    Let $(X,\lambda)\in \calA_g(k)$ be a geometric point. Then $(X,\lambda)\in \calS_g^{\rm eo}$ if and only if there exist an integer $0\le c\le \lfloor g/2 \rfloor$, a polarisation $\mu$ on the superspecial abelian variety $E^g$ such that $\ker \mu\simeq \alpha_p^{2c}$, and a polarised isogeny $\rho: (E^g,\mu)\to (X,\lambda)$. 
    
    Moreover, if $(X,\lambda)\in S_{\varphi_{\max}}$, then $c=\lfloor g/2 \rfloor$; and $(X,\lambda)$ is superspecial if and only if one can take $c=0$.   
\end{proposition}

Put $r:=\lfloor g/2 \rfloor$.
For any $0\leq c \leq r $,
let $P_c$ be the set of isomorphism classes of polarisations~$\mu$ on $E^g$ such that $\ker \mu\simeq \alpha_p^{2c}\subseteq E^g[F]$. We fix a complete set of representatives $\mu_1,\dots \mu_{h_c}$ for $P_c$ and identify $P_c$ with $\{\mu_1,\dots, \mu_{h_c}\}$. Note that every polarisation on $E^g$ is defined over $\F_{p^2}$. 
In particular, $\ker \mu$ is a finite flat subgroup scheme defined over $\F_{p^2}$. 

\begin{definition}\label{def:Xmu} 
For any polarisation $\mu\in P_r$, let 
\[ \calX_{\mu}:{(\F_{p^2}{\rm-Sch})} \to {\rm (Set)} \]
be the functor from the category of $\F_{p^2}$-schemes to the category of sets, which sends an $\F_{p^2}$-scheme $S$ to the set of isomorphism classes of 
polarised isogenies $\rho:(E_{0}^g,\mu)\times S\to (X,\lambda)$ over $S$ such that 
\begin{itemize}
    \item [(i)] $\ker \rho$ is an $\alpha$-group of rank $r$ over $S$;
    \item [(ii)] $(X,\lambda)$ is a (necessarily principally) polarised abelian scheme of relative dimension $g$ over $S$.
\end{itemize}
Two objects $\rho:(E_{0}^g,\mu)\times S\to (X,\lambda)$ and $\rho':(E_{0}^g,\mu)\times S\to (X',\lambda')$ are said to be \emph{isomorphic} if there exists an isomorphism $\alpha: (X,\lambda)\isoto (X',\lambda')$ of polarised abelian schemes such that $\rho'=\alpha\circ \rho$. 
\end{definition}

We now explain the connection between $\mathcal{X}_{\mu}$ and a suitable Lagrangian variety. The short exact sequence 
\[ \begin{tikzcd}
    0 \arrow[r] & \ker \mu  \arrow[r] & E_0^g \arrow[r, "\mu"] & E_0^{g,t} \arrow[r] & 0
\end{tikzcd}\]
induces the short exact sequence 
\[ \begin{tikzcd}
    0 \arrow[r] & (\ker \mu)^D  \arrow[r] & E_0^{g,tt}\simeq E_0^g \arrow[r, "\mu^t"] & E_0^{g,t} \arrow[r] & 0,
\end{tikzcd}\]
where $\mu^t$ is the dual morphism induced by $\mu$ by functoriality of dual abelian varieties.
It follows from $\mu^t=\mu$ that there is a natural isomorphism $\lambda:\ker \mu \isoto (\ker \mu)^D$ with $\lambda^D=-\lambda$, which induces the Weil pairing 
\[ e_\mu: \ker \mu \times \ker \mu \to \Gm. \]

Let ${\mathrm X} (\ker \mu, e_\mu)$ denote the projective scheme over $\F_{p^2}$ that represents the functor sending an $\F_{p^2}$-scheme $S$ to the set of $\alpha$-subgroups $H$ of rank $r$ over $S$ of $\ker \mu_{S}$ such that $e_\mu(H,H)=0$. Representibility follows from~\cite[Lemma 2.8]{li} and the fact that $e_\mu(H,H)=0$ is a closed condition. 
The $\alpha$-sheaf $\calA(\ker \mu)$, after identifying it with its global sections, is an $\F_{p^2}$-vector space of dimension $2r$ together with a non-degenerate alternating pairing $\psi_{\calA(\ker \mu)}$ induced by $e_\mu$. Let ${\rm L}(\calA(\ker \mu), \psi_{\calA(\ker \mu)})$ denote the Lagrangian variety associated to the symplectic space $(\calA(\ker \mu), \psi_{\calA(\ker \mu)})$.

\begin{proposition}\label{prop:rep_L}
    We have isomorphisms of moduli functors 
    \[ \calX_{\mu} \simeq {\rm X}(\ker \mu, e_\mu) \simeq {\rm L}(\calA(\ker \mu), \psi_{\calA(\ker \mu)}). \]
    In particular, the moduli functor $\calX_{\mu}$ is representable by a geometrically irreducible projective smooth scheme over $\F_{p^2}$, which is denoted again by $\calX_{\mu}$. 
\end{proposition}
\begin{proof}
    The first isomorphism sends each polarised isogeny $\rho:(E_{0}^g,\mu)\times S\to (X,\lambda)$ to $\ker \rho \subseteq \ker \mu_S$. The inverse map sends an $\alpha$-subgroup scheme $H\subseteq \ker \mu_S$ to the isogeny $\pi: E_0^g\times S\to X$, where $X:=(E_0^g\times S)/H$. The polarisation $\mu$ descends to a polarisation $\lambda$ on $X$ as $e_\mu(H,H)=0$; see~\cite{moretbailly}. This establishes the first isomorphism. 

    The second isomorphism sends $H\subseteq \ker \mu_S$ to $\calF:=\ker\left (\calA(\ker \mu)\otimes \calO_S \to \calA(H)\right)$. One has that $e_\mu(H,H)=0$ if and only if $\calF$ is a maximal isotropic $\calO_S$-submodule. This establishes the second isomorphism.  
\end{proof}

\begin{remark}
The moduli space~$\calX_\mu$ was first constructed by Moret-Bailly~\cite{moretbailly} for $g=2$ and was exploited by Katsura and Oort~\cite{katsuraoort:compos87} to study the moduli space of principally polarised abelian surfaces.
\end{remark}

Put $\calX_{\mu,k}:=\calX_{\mu}\otimes_{\F_{p^2}}k$ and identify $\calX_{\mu,k}$ with the set of $k$-points in $\calX_{\mu}$. 

We write $(M_1,\<\,,\>)$ and $(M_1^{\diamond}, \<\,,\>)$ for the quasi-polarised \dieu modules of $(E^g, \mu)$ and $(E_0^g,\mu)$, respectively. As $\ker \mu=\alpha_p^{2r} \subseteq E_0^g[F]$, the polarisation $\mu: E_0^g \to E_0^{g,t}$ yields a quotient 
\[ \ol M_1^\diamond:=M_1^\diamond/(M_1^{t})^{\diamond}=M^*(\ker \mu),\] which is a vector space over $\F_{p^2}$ of dimension $2r$, and $p\<\, ,\>$ induces a non-degenerate alternating pairing on 
$\ol M_1^\diamond$:
 \begin{equation}
     \label{eq:psiM0}
     \psi_{\ol M_1^\diamond}: \ol M_1^\diamond \times \ol M_1^\diamond \to \F_{p^2}.
 \end{equation}
This yields a symplectic space $(\ol M_1^\diamond,\psi_{\ol M_1^\diamond})$ of dimension $2r$. 
For each geometric point $\rho:(E^g,\mu)\to (X,\lambda)$ in $\calX_\mu$, let $(M,\<\, ,\>)\subseteq (M_1,\<\, ,\>)$ be the corresponding chain of \dieu modules together with quasi-polarisations. 
Since $\lambda$ is principal and
the isogeny $\rho$ has degree~$p^r$, one has $\dim_k M_1/M=\dim_k M/M_1^t=r$. 
Moreover, since $M$ is self-dual with respect to the pairing $\<\,,\>$, the quotient $\ol M:=M/M_1^t$ is a maximal isotropic subspace of $\ol M_1= \ol M_{1,k}^{\diamond}$. Hence, starting with $\rho$ we obtain an element $\ol M$ in ${\rm L}(\ol M_1^\diamond,\psi_{\ol M_1^\diamond})$. 

\begin{lemma}\label{lem:Xmu}
    We have an isomorphism of algebraic varieties over $k$ 
    \[
    \begin{split}
    F_\mu: \calX_{\mu,k} & \to {\rm L}(\ol M_1^\diamond,\psi_{\ol M_1^\diamond})_k,\\ 
    (\rho:(E^g,\mu)\to (X,\lambda)) &\mapsto \ol M\subseteq \ol M_{1,k}^\diamond = \ol M_1,
    \end{split}
    \]
    which is defined over $\F_{p^2}$. Therefore, $F_\mu$ defines an isomorphism $\calX_{\mu}\isoto {\rm L}(\ol M_1^\diamond,\psi_{\ol M_1^\diamond})$ of algebraic varieties over $\F_{p^2}$.
\end{lemma}

\begin{proof}
    We show that the map $F_\mu$ is induced from the isomorphism 
    $\calX_{\mu} \simeq {\rm L}(\calA(\ker \mu), \psi_{\calA(\ker \mu)})$, and hence $F_\mu$ is an isomorphism that is defined over $\F_{p^2}$. 
    By Proposition~\ref{prop:rep_L}, it suffices to show that $\ol M_1^\diamond=\calA(\ker \mu)$ and that $\ol M=\ker (\calA(\ker \mu)\otimes k \to \calA(\ker \rho))$.

    By Equation~\eqref{eq:homsheaf}, we have 
    \begin{equation}\label{eq:MG=AG}
       \calA(\ker \mu)=\Hom_{\F_{p^2}}(\ker \mu, \alpha_p)=\Hom_{\F_{p^2}}(M^*(\alpha_p), M^*(\ker(\mu))=M^*(\ker \mu) =\ol M_1^\diamond. 
    \end{equation}
    Here $M^*(G)$ denotes the contravariant \dieu module of $G$.
    This shows the first equality. 

    Using the sequence
   \[ \begin{tikzcd}
        0 \arrow[r] & \ker \rho \arrow[r] & E^g \arrow[r, "\rho"] &  X  \arrow[r, "\rho^t "] &  E^{g,t}\arrow[r] & 0, 
    \end{tikzcd} \]
    we obtain a short exact sequence
    \[ \begin{tikzcd}
        0 \arrow[r] & \ol M \arrow[r] &  \ol M_1=M^*(\ker \mu)\otimes k \arrow[r] & M^*(\ker \rho)  \arrow[r] &  0.
    \end{tikzcd} \]
    Thus, we have $\ol M=\ker (\calA(\ker\mu)\otimes k \to \calA(\ker \rho))$ from Equation~\eqref{eq:MG=AG}. This shows the second equality.    
\end{proof}

\begin{corollary}\label{cor:prXmu}
We have a finite surjective morphism
\[ \pr: \coprod_{\mu\in P_r} \calX_{\mu} \to \calS_g^{\rm eo}.\]
Moreover, all irreducible components of $\calS_g^{\rm eo}$ are given by $\pr(\calX_{\mu})$, for some $\mu\in P_r$.
\end{corollary}
\begin{proof}
    This follows from Proposition~\ref{COH}, since the projection map  
\[ \pr:\calX_{\mu} \to \calS_g, \quad \rho\mapsto (X,\lambda) \]
factors through $\calS_g^{\rm eo}$. 
\end{proof}

\begin{remark}
    In fact, we can avoid introducing supersingular EO strata and redefine $\calS_g^{\rm eo}$ as the union of the images $\pr(\calX_{\mu})$ for $\mu\in P_r$, together with the induced reduced closed subscheme structure.     
\end{remark}

For each $0\le c\le r$, let $\calS_{g,\le c}^{\rm eo}\subseteq \calS_{g}^{\rm eo}$ be the closed subset consisting of all geometric points $(X,\lambda)$ such that there exist an element $\mu\in P_c$ and a polarised isogeny $\rho:(E^g,\mu)\to (X,\lambda)$. 
Let $\calS^{\rm eo}_{g, c} \subseteq \calS^{\rm eo}_{g,\le c}$ be the open dense subset consisting of objects $(X,\lambda)$ such that there exist a $\mu\in P_c$ and a \emph{minimal} isogeny $\rho:(E^g,\mu)\to (X,\lambda)$, as introduced in Subsection~\ref{sec:end_ssav}. In other words, $c$ is the smallest integer such that there exist a $\mu\in P_c$ and a polarised isogeny $\rho:(E^g,\mu)\to (X,\lambda)$ of degree $p^c$.
We have a decomposition
\[ \calS_g^{\rm eo}=\coprod_{0\le c \le r} \calS_{g,c}^{\rm eo}\]
where $\calS_{g,0}^{\rm eo}$ is the superspecial locus by Proposition~\ref{COH}. 

\subsection{A stratification on $\boldsymbol{\calS_g^{\rm eo}}$}\

Recall that $r = \lfloor g/2 \rfloor$. We fix a symplectic space $(V_0,\psi_0)$ of dimension $2r$ over $\F_{p^2}$. Let $\mathrm{L}(V_0,\psi_0)$ be the Langrangian variety associated to $(V_0,\psi_0)$. 
Any isomorphism $\beta:(\ol M_1^\diamond,\psi_{\ol M_1^\diamond})\isoto (V_0,\psi_0)$ gives rise to an isomorphism
\[ \beta_*:\calX_\mu \isoto \mathrm{L}(\ol M_1^\diamond, \psi_{\ol M_1^\diamond})\isoto \mathrm{L}(V_0,\psi_0)\]
which sends each $k$-point $\rho:(E^g,\mu)\to (X,\lambda)$ in $\mathcal{X}_{\mu}$ to the maximal isotropic $k$-subspace $\beta(\ol M)\subseteq V_{0,k}$.
Write $\mathrm{L} = {\rm L}(V_0,\psi_0)\otimes k$, define $\mathcal{E} := \{ \End(V_0,W): W \in \mathrm{L} \}$ as in Equation~\eqref{eq:calE} and consider its $\mathrm{Sp}(V_0)$-conjugacy classes as in Section~\ref{sec:symp}.\\

As a naive first attempt, for each $[E]\in [\calE]$, define
\[ \calX_{\mu,[E]}:=\{\rho\in \calX_{\mu}(k): \beta_*(\rho)\in \mathrm{L}_{[E]} \}, \quad \calX_{\mu,[E]}^0:=\{\rho\in \calX_{\mu}(k): \beta_*(\rho)\in (\mathrm{L}_{[E]})^0 \}, \]
where $({\mathrm L}_{[E]})^0$ is defined in Equation~\eqref{eq:LE}.
The strata $\calX_{\mu,[E]}$ and $\calX_{\mu,[E]}^0$ do not depend on the choice of the isomorphism $\beta$. We obtain a stratification
\begin{equation}\label{eq:XmuE}
    \coprod_{\mu\in P_r} \calX_\mu=\coprod_{[E]\in [\calE]} \coprod_{\mu\in P_r} \calX_{\mu,[E]}^0.
\end{equation}

However, this stratification does not descend to the moduli space $\calS_g^{\rm eo}$; the reason is that the relative endomorphism ring is an invariant of points $\rho$ in $\calX_{\mu}$ but not an invariant of $(X,\lambda)$. We modify the stratification as follows.\\

For each integer $1\le c\le r$, we fix a symplectic space $(V_c,\psi_c)$ of dimension $2c$ over $\F_{p^2}$. 

\begin{definition}
 As in Definition~\ref{def:Xmu}, for each $\mu_c\in P_c$ we define the moduli space $\calX_{\mu_c}$ over~$\F_{p^2}$ of polarised isogenies $\rho:(E^g,\mu_c)\to (X,\lambda)$ of degree $p^c$.   
 Let $\calX_{\mu_c}^d$ be the open subset consisting of minimal isogenies~$\rho$ in $\calX_{\mu_c}$. 
\end{definition}

Let $(M_1,\<\,,\>)$ be the \dieu module of $(E^g,\mu_c)$, and define $(\ol M_1^\diamond,\psi_{\ol M_1^\diamond})$ as in Equation~\eqref{eq:psiM0}, which is a symplectic space over $\F_{p^2}$ of dimension $2c$. 
The following lemma should be compared with Lemma~\ref{lem:Xmu}.
\begin{lemma}
    We have an isomorphism $\calX_{\mu_c}^d\simeq \mathrm{L}(\ol M_1^\diamond, \psi_{\ol M_1^\diamond})^{\rm d}$.
\end{lemma}
\begin{proof}
    Any $k$-point $\rho:(E^g,\mu_c)\to (X,\lambda)$ in $\calX_{\mu_c}$ gives an element $\ol M=M/M_1^t$ in $\mathrm{L}(\ol M_1^\diamond)$, where $(M,\<\,,\>)$ is the \dieu module of $(X,\lambda)$. If $M'$ is a \dieu module with $M_1^t\subseteq M'\subseteq M_1$, then the $k$-subspace $\ol M'=M'/M_1^t$ is defined over $\F_{p^2}$ if and only if $M'$ is superspecial, cf.~\cite[Lemma 6.1]{yu:docmath2006}. Thus, if $\wt M$ is the smallest superspecial \dieu module containing $M$, then $\wt M/M_1^t$ is the $\F_{p^2}$-hull of $\ol M$. Therefore, $\rho$ is the minimal isogeny if and only if $\wt M=M_1$, i.e. $\ol M$ is $\F_{p^2}$-dense. 
\end{proof}

We also have the projection map 
\[ \pr:\calX_{\mu_c}\to \calS_{g,\le c}^{\rm eo}. \]

\begin{definition}
Put \[ \calS_{g,c,\mu_c}^{\rm eo}:=\pr(\calX_{\mu_c})\cap \calS_{g,c}^{\rm eo}. \]
\end{definition}

The following lemma should be compared with Corollary~\ref{cor:prXmu}.
\begin{corollary}
    We have a finite surjective map
\begin{equation}\label{eq:pr}
    \pr: \coprod_{\mu_c\in P_c} \calX_{\mu_c}^d \to \calS_{g,c}^{\rm eo}=\coprod_{\mu_c\in P_c} \calS_{g,c,\mu_c}^{\rm eo},
\end{equation}
and $\calS_{g,c,\mu_c}^{\rm eo}$, for $\mu_c \in P_c$, are the connected and irreducible components of $\calS_{g,c}^{\rm eo}$.
\end{corollary}
\begin{proof}
If $\rho:(E^g,\mu_c)\to (X,\lambda)$ is the minimal isogeny, letting $\ol M=M/M_1^t\subseteq \ol M_1=\ol M^\diamond_{1,k}$, then $\ol M_1$ is the $\F_{p^2}$-hull of $\ol M$. This shows that the preimage of $\calS_{g,c,\mu_c}^{\rm eo}$ under $\pr$ is $\calX_{\mu_c}^d$. 
Note that $\calS^{\rm eo}_{g,c,\mu_c}\cap \calS^{\rm eo}_{g,c,\mu_c'}=\emptyset$ if $\mu_c\neq \mu_c'$, since $\mu_c$ is determined by the points in $\calS^{\rm eo}_{g,c,\mu_c}$.
\end{proof}

\begin{definition}
Choose an isomorphism $\beta:(\ol M_1^\diamond,\psi_{\ol M_1^\diamond})\isoto (V_c,\psi_c)$, which induces an isomorphism $\beta_*~:~\calX_{\mu_c}\isoto \mathrm{L}(V_c, \psi_c)$. Let 
\[ [\calE_{c}^d]:=\{[\End(V_c,W)]: W\in \mathrm{L}(V_c,\psi_c)^d \}.\] 
For each $[E]\in [\calE_c^d]$, define
\[ \calX_{\mu_c,[E]}^{d,0}:=\{ \rho\in \calX_{\mu_c}^{d}(k):\beta_*(\rho)\in ((\mathrm{L}(V_c,\psi_c) \otimes k)^d_{[E]})^0 \}. \]
Then we have 
\[ \coprod_{\mu_c\in P_c} \calX_{\mu_c}^{d}=\coprod_{\mu_c\in P_c} \coprod_{[E]\in [\calE_c^d]} \calX_{\mu_c,[E]}^{{d},0}. \]
For $c=0$, we set $\calX_{\mu_0}=\{(E^g,\mu_0)\}$, $[\calE^d_0]=\{[0]\}$, and $\calX_{\mu_0,[0]}^{d,0}=\calX_{\mu_0}$. 
\end{definition}

For each integer $c$ with $0\le c\le r$, we have a finite surjective map induced by Equation~\eqref{eq:pr}:
\begin{equation}
    \label{eq:XmucE}
    \mathrm{pr}: \coprod_{\mu_c\in P_c} \calX_{\mu_c}^{d}=\coprod_{\mu_c\in P_c} \coprod_{[E]\in [\calE_c^d]} \calX_{\mu_c,[E]}^{{d},0}\to \calS_{g,c}^{\rm eo}=\coprod_{\mu_c\in P_c} \calS_{g,c,\mu_c}^{\rm eo}.
\end{equation}

\begin{lemma}
The $\Sp(\ol M_1^\diamond)$-conjugacy class $[\End(\ol M_1^\diamond, \ol M)]$ is a well-defined invariant for the isomorphism classes $[X,\lambda]$ in $\calS_{g,c,\mu_c}^{\rm eo}$.    
\end{lemma}
\begin{proof}
Let $(X,\lambda)$ and $(X',\lambda')$ be two objects in $\calS_{g,c,\mu_c}^{\rm eo}$ such that there is an isomorphism $\alpha:(X,\lambda)\to (X',\lambda')$. 
Let $(E^g,\mu_c)\to (X,\lambda)$ and $(E^g,\mu_c)\to (X',\lambda')$ be the respective minimial isogenies.
With the notation above, we have two members $\ol M=M/M^t_1$ and $\ol M'=M'/M^t_1$ in $\mathrm{L}(\ol M_1^\diamond)$, where $M$ and $M'$ are the respective \dieu modules of $X$ and $X'$. The isomorphism $\alpha$ lifts to an automorphism $\wt \alpha\in \Aut(E^g, \mu_c)$, which induces an element $\ol \alpha\in \Sp(\ol M_1^\diamond)$ such that $\ol \alpha(\ol M')=\ol M$. This shows that the endomorphism algebras $\End(\ol M_1^\diamond,\ol M)$ and $\End(\ol M_1^\diamond, \ol M')$ are $\Sp(\ol M_1^\diamond)$-conjugate. 
\end{proof}

\begin{definition}
Set 
\[ \calS_{g,c,\mu_c,[E]}^{{\rm eo},0}:=\pr(\calX_{\mu_c,[E]}^{{d},0}).\]
Then we obtain a stratification 
\begin{equation}\label{eq:st_Sgeo}
\calS_{g}^{\rm eo}=\coprod_{0\le c\le r} \coprod_{[E]\in [\calE_c^d]}\calS_{g,c,[E]}^{{\rm eo},0}, \quad  
\calS_{g,c,[E]}^{{\rm eo},0}:=\coprod_{\mu_c\in P_c}\calS_{g,c,\mu_c,[E]}^{{\rm eo},0}.
\end{equation}
\end{definition}

\begin{lemma}
    \label{lm:Vp}
Let $D_p$ be the division qutaternion $\Qp$-algebra, and $O_p$ the unique maximal order. Let $\Pi$ be a uniformiser of $O_p$. For integers $n\ge 1$ and $s\ge 1$, set 
\[ V_{p,s}:=1+\Pi^s \Mat_n(O_p)\subseteq \GL_n(O_p).\] Then $V_{p,s}$ 
is torsion-free if and only if either 
(i) $s\ge 3$, (ii) $p\ge 3$ and $s=2$, or (iii) $p\ge 5$ and $s=1$. 
\end{lemma}

\begin{proof}
We first remark that the conditions $p\ge 3$ when $s=2$, and $p\ge 5$ when $s=1$ are necessary: one has $-1\in V_{p,s}$ for $p=2$ and $s=2$, and see \cite[Remark 6.3]{karemaker-yobuko-yu} for $s=1$. 

For simplicity, write $(\Pi)$ for the two-sided ideal in $\Mat_n(O_p)$
generated by $\Pi$. Case (iii) has been proven in \cite[Lemma 6.2]{karemaker-yobuko-yu}; we now show the cases (i) and (ii).

  We must show that any $\alpha\in V_{p,s}$ of finite order  must equal $1$. 
  Since $V_{p,s}$ is a pro-$p$ group, we have $\alpha^{p^r}=1$ for some
$r\geq 1$. By induction, we may assume that $\alpha^p=1$. Suppose that
$\alpha\neq 1$ and write
$\alpha=1+\Pi^{s_0} \beta$ for some $\beta \in\Mat_n(O_p)$ with $\beta \not \in  (\Pi)$.  Since $p\mid \binom{p}{i}$ for all $1\leq
i\leq p-1$, we find
\begin{equation}
  \label{eq:1}
1=\sum_{i=0}^p \binom{p}{i}(\Pi^{s_0} \beta)^i\equiv \begin{cases}
   1+p\Pi^{s_0} \beta
  \pmod{\Pi^{2s_0}} & \text{if $s_0\ge 3$;} \\
   1+p \Pi^2 \beta \pmod{\Pi^6} & \text{if $s_0=2$ and $p\ge 3$.} 
\end{cases} 
\end{equation}
In both cases this implies that $\beta\in (\Pi)$, which leads to a contradiction. 
\end{proof}

Theorem~\ref{thm:main} now follows from the following result. 

\begin{theorem}\label{thm:A1}
   \begin{enumerate}
       \item The stratum $\calS_{g,r,[\F_{p^2}]}^{{\rm eo},0}$ is nonempty, and is open and dense in $\calS_{g}^{{\rm eo}}$.
       \item If $g$ is even and $p\ge 5$, then for any polarised supersingular abelian variety $(X,\lambda)$ in $\calS_{g,r,[\F_{p^2}]}^{{\rm eo},0}(k)$, we have $\Aut(X,\lambda)=\{\pm 1\}$. 
   \end{enumerate} 
\end{theorem}
\begin{proof} 
\begin{enumerate} 
    \item It follows from Theorem~\ref{thm:L:nonempty} that for each $\mu\in P_r$, $\calX_{\mu,[\F_{p^2}]}^0$ is nonempty, and moreover open and dense in $\calX_{\mu}$. As the projection map $\pr: \coprod_{\mu\in P_r} \calX_\mu \to \calS_g^{\rm eo}$ is finite and surjective, the statement follows. 

    \item With the notation of Subsection~\ref{sec:auto}, for the minimal isogeny of $(X,\lambda)$ we have $(\wt X,\wt \lambda) \simeq (E^g,\mu)$ for some $\mu\in P_r$. Since $g$ is even, $\ker \lambda=\wt X[F]$ and hence $\wt M^t =\sfV \wt M$. We have 
    \[ \Aut(X,\lambda)\subseteq \Aut_{\rm DM}(\wt M,p\<\,,\>) \xrightarrow{m_\sfV} \Sp(V_0^{\rm min}), \]
    where $m_\sfV$ is given by the reduction modulo $\Pi$: this also maps $\Aut_{\rm DM}(\wt M)=\GL_g(O_p)\to \GL(V_0^{\rm min})=\GL_g(\F_{p^2})$. Since $p\ge 5$, $\ker m_\sfV=1+\Pi \Mat_g(O_p)$ is torsion-free by Lemma~\ref{lm:Vp}. Therefore, the restriction map $m_\sfV:\Aut(X,\lambda)\to \Sp(V_0^{\rm min})$ is injective. 
    The image $m_\sfV(\Aut(X,\lambda))$ is contained in 
    \[ \Sp(V_0^{\rm min}, \ol M)=\{\alpha\in \F_{p^2}: \alpha^\dagger \alpha=1\}=\{\pm 1\}, \]
    cf.~Equation~\eqref{eq:SpVoW}. Thus, $\Aut(X,\lambda)=\{\pm 1\}$. 
\end{enumerate}   
\end{proof}

\begin{remark}\label{rem:goddp2} 
\begin{itemize}
    \item[(1)] 
Theorem 6.4.(2) fails if either $g\ge 3$ is odd, or $p=2$. For odd $g\ge 3$, we take $\mu=(\mu',\lambda_E)\in P_r$ on $E^g=E^{g-1}\times E$, where $\mu'$ is a polarisation on $E^{g-1}$ with $\ker \mu'\simeq \alpha_p^{2r}$, and $\lambda_E$ is the canonical principal polarisation on $E$. Then any isogeny $\rho:(E^g,\mu)\to (X,\lambda)$ in $\calX_\mu$ is a product isogeny $\rho=(\rho',{\rm id}_E): (E^{g-1},\mu')\times (E,\lambda_E) \to (X,\lambda)=(X',\lambda')\times (E,\lambda_E)$. Therefore, $\Aut(X,\lambda)$ cannot be $\{\pm 1\}$. 

    For $p=2$, the case $g=2$ in \cite{ibukiyama} already gives a counterexample.
 \item[(2)] It is natural to ask whether for each EO stratum  $S_\varphi$ with $c(\varphi)=r$, the intersection $S_\varphi\cap \calS_{g,r,[\F_{p^2}]}^{{\rm eo}, 0}$ is open and dense in $S_\varphi$. Using the argument of $\ell$-adic Hecke correspondences (see Subsection~\ref{sec:lhecke}), this is equivalent to asking whether the intersection $S_\varphi\cap \calS_{g,r,[\F_{p^2}]}^{{\rm eo}, 0}$ is non-empty. The latter    
    has recently been answered affirmatively by Shimada and Takamatsu~\cite{shimada-takamatsu:oort}. We thank Ryosuke Shimada for explaining the proof to us.
\end{itemize}
\end{remark}

Using the stratification constructed in Equation~\eqref{eq:st_Sgeo}, we give a concrete description of the mass function on the supersingular EO locus $\calS_g^{\rm eo}$ for arbitrary $g$. Note that for $x=(X,\lambda)\in \calS_g(k)$, the central leaf $\calC(x)$ passing through $x$ is $\Lambda_x$, as for any prime $\ell\neq p$, any two principally polarised $\ell$-divisible groups $(X',\lambda')[\ell^\infty]$ and $(X,\lambda)[\ell^\infty]$ are isomorphic.  

\begin{theorem}
    \label{thm:mass}
    The function $\Mass: \calS_{g}^{\rm eo}\to \Q$, sending $x$ to
    $\sum_{(X',\lambda')\in \calC(x)} |\Aut(X',\lambda')|^{-1} =: \Mass(x)$, is constant on $\calS_{g,c,[E]}^{\rm eo, 0}$ with value
    \[ \prod_{i=1}^g \frac{|\zeta(1-2i)|}{2} \cdot [\Sp(V_c):E^1] \cdot L_{g,p^c},\]
where $E^1:=\{\alpha\in E^\times: \alpha^\dagger \alpha=1\}$. Here we set $[\Sp(V_c):E^1]:=1$ when $c=0$.  

In particular, each non-empty fibre of the function $\Mass$ is a union of locally closed subsets. 
\end{theorem}
\begin{proof}
This follows immediately from Proposition~\ref{prop:mass_Lambdax}.
\end{proof}

\begin{remark}
    In~\cite{karemaker-yobuko-yu} the authors gave explicit mass formulae and a concrete description for the mass strata for $g=3$.  Theorem~\ref{thm:mass} provides a conceptual understanding of the interplay between the mass function and the geometry on the supersingular EO locus. Then through the stratification defined in Equation~\eqref{eq:st_Sgeo}, one sees a more direct relationship between the arithmetic and geometry on supersingular strata.
\end{remark}

\begin{remark}
    In~\cite{chen-viehmann} Chen and Viehmann introduce a new stratification on affine Deligne-Lusztig varieties (ADLVs) $X_\mu(b)$, called the $J$-stratification. (Note that in this context, $\mu$ denotes a minuscule dominant cocharacter $\mu \in X_*(T)_{\mathrm{dom}}$ of a fixed maximal torus~$T$ in a connected reductive algebraic group~$G$, rather than a polarisation on a superspecial abelian variety as considered above.) They construct, for each element $g$ in $X_\mu(b)$, a function $f_g$ from the 
    $\sigma$-centraliser $J_b(\breve \Q_p)$ of $b\in G(\breve \Qp)$ to the set $X_*(T)_{\rm dom}$ of dominant co-characters of $G$, where $\breve \Qp$ is the completion of the maximal unramified extension of $\Qp$. Then the \emph{$J$-strata} of $X_\mu(b)$ are defined to be the fibres of the map
    \[ X_\mu(b) \longrightarrow \{f_g: g\in X_\mu(b)\}.\]
    The authors of \cite{chen-viehmann} compare the $J$-stratification with various  stratifications studied previously: the Bruhat-Tits stratification introduced by Vollaard and Wedhorn~\cite{VW}, the semi-module stratification introduced by de Jong and Oort~\cite{dejong-oort}, and the locus with $a$-number one. 

    For the special case where $G=\GSp_{4}$, $\mu=(1,1,0,0)$ and $[b]\in B(G,\mu)$ is the supersingular $\sigma$-conjugacy class, $X_\mu(b)$ is the ADLV (or the special fibre of the Rapoport-Zink space) associated to the moduli space $\calS_2$ of principally polarised abelian surfaces. In this case there are two $J$-strata of $X_\mu(b)$: the superspecial locus and its complement; see~\cite[Proposition~4.3]{chen-viehmann}. 

    We compare the $J$-stratification with our stratification on $X_\mu(b)$ in this special case. Each irreducible component $X$ of $X_\mu(b)$ is isomorphic to $\calX_\mu=\bbP^1$ up to perfection. Using our stratification indexed by $[\calE]$, the moduli space $\calX_\mu$ decomposes into three pieces:
    \[ \bbP^1(\F_{p^2})\ \text{(the superspecial locus)}, \quad \bbP^1(\F_{p^4})\setminus \bbP^1(\F_{p^2}), \quad \bbP^1(k)\setminus \bbP^1(\F_{p^4}),\]
    and hence it refines the $J$-stratification.

    For general $g\ge 2$, let $G=\GSp_{2g}$, $\mu=(1^g,0^g)$ and $[b]$ be the supersingular $\sigma$-conjugacy class, and let
    \[ \Theta: X_\mu(b) \longrightarrow \calS_g^{\rm pf} \longrightarrow \calS_g \]
    be the map induced by the ADLV incarnation of the Rapoport--Zink uniformisation~(cf.~\cite[Theorem 6.21]{RZ} and~\cite[p.~4]{HZZ}), where $\calS_g^{\rm pf}$ is the perfection of $\calS_g$. Let $X_\mu(b)^{\rm eo}:=\Theta^{-1}(\calS_g^{\rm eo})$ be the preimage of $\calS_g^{\rm eo}$. Then each irreducible component $X$ of $X_\mu(b)^{\rm eo}$ is isomorphic to $\calX_\mu$ up to perfection. 
    We expect that our stratification refines the $J$-stratification on $X\simeq \calX_\mu^{\rm pf}$. 
\end{remark}

\subsection{$\boldsymbol{\ell}$-adic Hecke correspondences on the supersingular locus}\label{sec:lhecke}\ 

Our references for the exposition of $\ell$-adic Hecke correspondences are \cite{chai:mono,yu:crelle}.
We choose a projective system of
primitive prime-to-$p$-th roots of unity
$\zeta=(\zeta_m)_{(m,p)=1}\subseteq 
\Qbar\subseteq \C$. We also fix an algebraic closure $\overline{\Q}_p$ of $\Qp$ and an embedding $\Qbar \hookrightarrow
\overline{\mathbb{Q}}_p$. 
For any prime-to-$p$ integer $m\ge 1$ and any connected
$\Fpbar$-scheme $S$, the choice $\zeta$ determines an
isomorphism  
$\zeta_m:\Z/m\Z\isoto \mu_m(S)$, or equivalently, a $\pi_1(S,\bar
s)$-invariant $(1+m\wh \Z^{(p)})^\times$-orbit of isomorphisms $\bar
\zeta_m: \wh \Z^{(p)}\to \wh \Z^{(p)}(1)_{\bar s}$, where $\wh
\Z^{(p)}:=\prod_{\ell\neq p}\wh \Z_\ell$ and $\bar s$ is a geometric
point of $S$.

Let $(V,\psi)$ be a symplectic space over $\Q$ of dimension $2g$, and $V_\Z\subseteq V$ be a self-dual $\Z$-lattice.
Let $G$ be the automorphism group scheme over $\Z$ associated to the symplectic $\Z$-lattice
$(V_\Z,\psi)$; that is, for any commutative ring $R$, the group of
its $R$-valued points is defined by
\begin{equation}\label{eq:601}
  G(R):=\{g\in \GL(V_R)\,;\,
\psi(g(x),g(y))=\psi(x,y), \ \forall\, x, y\in V_R \,
\}, \quad V_R=V_\Z\otimes_\Z R.
\end{equation}

Let $n\ge 3$ be a prime-to-$p$ positive integer and $\ell$ be a prime
with $(\ell,pn)=1$. 
Let $m\ge 0$ be a non-negative integer. Let $U_{n\ell^m}$ be the
kernel of the reduction map  $G(\wh \Z^{(p)})\to G(\wh
\Z^{(p)}/n\ell^m\wh \Z^{(p)})$; this is an open 
compact subgroup of $G(\wh \Z^{(p)})$.

For a $g$-dimensional polarised abelian scheme $(X,\lambda)$ with $p$-power polarisation degree over a connected locally Noetherian $\Fpbar$-scheme $S$, 
a {\it level-$U_{n\ell^m}$ structure} on $(X,\lambda)$ is 
a $\pi_1(S,\bar s)$-invariant
  $U_{n\ell^m}$-orbit $[\eta]_{U_{n\ell^m}}$ in ${\rm Isom}(V_\Z \otimes \wh\Z^{(p)}, T^{(p)}(X_{\bar s}))/U_{n\ell^m}$ of isomorphisms 
  \begin{equation}
    \label{eq:611}
    \eta: V_\Z\otimes \wh\Z^{(p)}
  \isoto T^{(p)}(X_{\bar s}):=\prod_{p'\neq p} T_{p'}(X_{\bar s}),
  \end{equation}
where $T_{p'}(X_{\bar s})$ is the $p'$-adic Tate module of $X_{\bar s}$, such that 
  \begin{equation}
    \label{eq:612}
   e_\lambda(\eta(x),\eta(y))=\bar \zeta_{n\ell^m}(\psi(x,y))\
  \quad ({\rm mod}\ (1+n\ell^m\wh \Z^{(p)})^\times), \quad 
  \forall\, x, y\in  V_\Z\otimes \wh\Z^{(p)}, 
  \end{equation}
where $e_\lambda$ is the Weil pairing induced by the polarisation
$\lambda$.

Let 
${\calA}_{g,n\ell^m}$ be the moduli space over
$\Fpbar$   
that parametrises isomorphism classes of $g$-dimensional principally polarised abelian varieties 
$(X,\lambda,[\eta]_{U_{n\ell^m}})$ with level-$U_{n\ell^m}$ structure. 
For integers $0 \le m \le m'$, we have a natural finite morphism $\pi_{m,m'}:\calA_{g,n\ell^{m'}}\to
\calA_{g, n\ell^m}$,  
sending $(X,\lambda,\iota,[\eta]_{U_{n\ell^{m'}}})$ to $
(X,\lambda,\iota,[ \eta]_{U_{n\ell^{m}}})$. 
Let $\wt \calA_{g,n}:=(\calA_{g, n\ell^m})_{m\ge 0}$ be the projective system. The natural projection $\wt \calA_{g,n}\to \calA_{g,n}$ forms a $G(\Z_\ell)$-torsor. 
The (right) action of
$G(\Z_\ell)$ on $\wt \calA_{g,n} $ extends uniquely to a continuous 
action of $G(\Q_\ell)$. Descending
from $\wt \calA_{g,n}$ to $\calA_{g,n}$, elements of  $G(\Q_\ell)$
induce algebraic correspondences on $\calA_{g,n}$, known as the
$\ell$-adic Hecke correspondences on $\calA_{g,n}$. 

More precisely, to each
$u\in G(\Q_\ell)$ we associate an $\ell$-adic Hecke correspondence
$(\calH_u,\pr_1,\pr_2)$ as follows. 
Extending the isomorphisms $\eta$ from~\eqref{eq:611} to isomorphisms
\[ \eta_\Q: V\otimes \A^{(p)}_f\to V^{(p)}(X):=T^{(p)}(X)\otimes
\A^{(p)}_f,\]
where $\A^{(p)}_f$ is the prime-to-$p$ finite adele ring of $\Q$,
we see that a class $[\eta]_{U_n}$ gives rise to a class
$[\eta_\Q]_{U_n}$ in ${\rm Isom}(V\otimes \A^{(p)}_f, V^{(p)}(X))/U_n$, 
and that $[\eta]_{U_n}$ is determined by $[\eta_\Q]_{U_n}$.
We have $u^{-1} (U_{n\ell^m}) u \subseteq U_n$ for some $m\ge 0$. 
For any open compact subgroup $U\subseteq G(\Z_\ell)$ we set  
$\calA_{g,U}:=\wt\calA_{g,n}/U$. 
Then we set 
$U_{n\ell^m,u}:=U_{n\ell^m} \cap u (U_{n\ell^m}) u^{-1}$ and $\calH_{u}:=\calA_{g,U_{n\ell^m,u}}$.
Noting that $u^{-1} (U_{n\ell^m,u}) u=u^{-1} (U_{n\ell^m}) u \cap  U_{n\ell^m} = U_{n\ell^m,u^{-1}}$, the right translation 
\[ \rho_u: (X,\lambda,[\eta_\Q]_{U_{n\ell^m,u}})
\mapsto
(X,\lambda,[\eta_\Q u]_{u^{-1} (U_{n\ell^m,u})u})\]  
gives rise to an isomorphism
\[ \rho_u: \mathcal{H}_u = \calA_{g,U_{n \ell^m,u}}\simeq \calA_{g,U_{n \ell^m,u^{-1}}}=\calH_{u^{-1}}. \]
Let $\pr_1$ be the natural
projection $ \calH_{u}\to\calA_{g,n}$ and $\pr_2:=\pr \circ \rho_u: 
\calH_{u}\to\calA_{g,n}$ be the composition of the isomorphism 
$\rho_u$ with the
natural projection $\pr:\calH_{u^{-1}}\to \calA_{g,n}$. This defines
an $\ell$-adic Hecke correspondence $(\calH_u,\pr_1,\pr_2)$. 

For two
$\ell$-adic Hecke correspondences 
$\calH_{u_1}=(\calH_{u_1},p_{11},p_{12})$ and 
$\calH_{u_2}=(\calH_{u_2},p_{21},p_{22})$, one defines the composition
$\calH_{u_2}\circ \calH_{u_1}$ by
\[ (\calH_{u_2}\circ \calH_{u_1}, p_1, p_2),
  \]
where $\calH_{u_2}\circ
\calH_{u_1}:=\calH_{u_1}\times_{p_{12},\calA_{g,n}, p_{21}} \calH_{u_2}$, 
$p_1$ is the composition $\calH_{u_2}\circ
\calH_{u_1}\to \calH_{u_1}\stackrel{p_{11}} {\to} \calA_{g,n}$ and $p_2$
is the composition 
$\calH_{u_2}\circ \calH_{u_1}\to \calH_{u_2}\stackrel{p_{22}} {\to}
\calA_{g,n}$. A correspondence
$(\calH,\pr_1,\pr_2)$ generated by correspondences of the form
$\calH_u$ is also called an $\ell$-adic Hecke correspondence.  
 
For a scheme $X$ of finite type over a field $K$, we denote by $\Pi_0(X)$ the set of geometrically irreducible components of $X$.
A locally closed subset $Z$ of $\calA_{g,n}$ is
said to be {\it $\ell$-adic Hecke invariant} if ${\rm pr}_2
({\rm pr}_1^{-1}(Z))\subseteq Z$ for any $\ell$-adic Hecke correspondence
$(\calH,\rm{pr}_1,\rm{pr}_2)$. If $Z$ is an $\ell$-adic Hecke invariant,
locally closed subvariety of $\calA_{g,n}$, then the $\ell$-adic 
Hecke correspondences induce 
correspondences on the set $\Pi_0(Z)$ of geometrically irreducible
components. We say that $\Pi_0(Z)$ is {\it $\ell$-adic Hecke
  transitive} if 
the $\ell$-adic Hecke correspondences operate transitively on
$\Pi_0(Z)$; that is, for any two maximal (generic) points $\eta_1, \eta_2$ of
$Z$ there is an $\ell$-Hecke correspondence
$(\calH,\rm{pr}_1,\rm{pr}_2)$ so that $\eta_2\in \rm{pr}_2
(\rm{pr}_1^{-1}(\eta_1))$. 

For a geometric point $x\in \calA_{g,n}(k)$, denote by $\calH_\ell(x)$ the
$\ell$-adic Hecke orbit of $x$; this is the set of points generated by
$\ell$-adic correspondences starting from $x$. If $x=(X,\lambda,[\eta]_{U_n})$ and $x'=(X',\lambda',[\eta']_{U_n})$, then $x'\in \calH_\ell(x)$ if and only if there exists a polarised $\ell$-quasi-isogeny $\theta:(X',\lambda')\to (X,\lambda)$ (that is, $\ell^m \theta$ is an isogeny of $\ell$-power degree for some $m\ge 0$)
such that $\theta_*([\eta']_{U_n})=[\eta]_{U_n}$.

Let $\pi_n:\calA_{g,n}\to \calA_g$ be the natural map forgetting the level-$U_n$ structure. For any subscheme~$Z$ in $\calA_g$, let $Z_n:=Z\times_{\calA_g} \calA_{g,n}$. 
Thus, this defines subvarieties $\calS_{g,n}$ and $\calS^{\rm eo}_{g,n}$ of $\calA_{g,n}$.

\begin{proposition}\label{prop:hecke}
    The set $\Pi_0(\calS_{g,n})$ is $\ell$-adic Hecke transitive.
\end{proposition}
\begin{proof}
    Let $\Lambda_{g,n\ell^m}^*$ be the set of isomorphism classes of $g$-dimensional polarised superspecial abelian varieties $(X,\lambda,[\eta]_{U_{n\ell^m}})$ over $\Fpbar$ with level-$U_{n\ell^m}$ structure, and $\wt \Lambda_{g,n}^*=(\Lambda_{g,n\ell^m}^*)_{m\ge 0}$ be the tower of the projective system. 
    Let $\calS_{g,n\ell^m}(a=1)$ denote the subset of $\calS_{g,n\ell^m}$ consisting of objects $(X,\lambda)$ with $a(X)=1$. 
    It is an open dense subset of $\calS_{g,n\ell^m}$ by \cite[Subsection 4.9]{lioort} and therefore we have $\Pi_0(\calS_{g,n\ell^m}(a=1))=\Pi_0(\calS_{g,n\ell^m})$.
    By~\cite{lioort} and \cite[Theorem 2.2]{oda-oort}, the contraction map 
    $\calS_{g,n\ell^m}(a=1)\to \Lambda^*_{g,n\ell^m}$ that sends each object $(X,\lambda, [\eta]_{U_{n\ell^m}})$ to $(\wt X,\wt \lambda, [\wt \eta]_{U_{n\ell^m}})$, where $\varphi:(\wt X,\wt \lambda, [\wt \eta]_{U_{n\ell^m}})\to (X,\lambda, [\eta]_{U_{n\ell^m}})$ is the minimal isogeny, induces an isomorphism $\alpha:\Pi_0(\calS_{g,n\ell^m})\isoto \Lambda^*_{g,n\ell^m}$. 
    Since the construction of minimal isogenies and the formalism of $\ell$-adic Hecke correspondences commute, the map $\alpha$ induces a $G(\Q_\ell)$-equivariant isomorphism $\alpha: \Pi_0(\wt \calS_{g,n})\isoto \wt \Lambda^*_{g,n}$, where $\wt \calS_{g,n} = (\calS_{g,n})_{m \geq 0}$ is the projective system and where $G$ is as in Equation~\eqref{eq:601}. In particular, the isomorphism $\alpha:\Pi_0(\calS_{g,n})\isoto \Lambda^*_{g,n}$ is $\ell$-adic Hecke equivariant. Thus, it suffices to show that the set $\Lambda^*_{g,n}$ is $\ell$-adic Hecke transitive.  

    Choose a base point $x_0=(X_0,\lambda_0,[\eta_0]_{U_n})$ of $\Lambda_{g,n}^*$. Let $G_{x_0}$ be the group scheme over $\Z$ which represents the functor 
    \[ R\mapsto G_{x_0}(R):=\{a \in (\End(X_0)\otimes R)^\times: a^t \lambda a =\lambda \}, \]
    for any commutative ring $R$. The generic fibre $G_{x_0}\otimes \Q$ is a compact inner form of the group~$G$ defined in Equation~\eqref{eq:601} and is simply connected. The choice of $\eta_0$ gives an identification $G_{x_0}(\A_f^{(p)})=G(\A_f^{(p)})$, so one can regard  $U_n$ as an open compact subgroup of $G_{x_0}(\A_f^{(p)})$. As shown in~\cite[Theorem 2.3]{yu:smf}, there is an isomorphism 
    \[ \gamma: \Lambda^*_{g,n} \isoto DS:=G_{x_0}(\Q)\backslash G_{x_0}(\A_f)/G_{x_0}(\Z_p)U_n \] which sends the base point $x_0$ to the identity class $[{\rm id}]$. 
    The construction in \emph{loc.~cit.}~also shows that $\gamma$ is prime-to-$p$ Hecke equivariant.
    Thus, we are reduced to showing that the double coset space $DS$ is $\ell$-adic Hecke transitive; or in other words, that the double coset space $DS$ coincides with the $\ell$-adic Hecke orbit $\calH_\ell(x_0)$ of $x_0$ under the isomorphism $\gamma$. By the construction of $\ell$-adic Hecke correspondences, this amounts to showing that every class $[\wh g]\in DS$ can be represented by some element  $g_\ell\in G_{x_0}(\Q_\ell)$. 

    We have \[ DS=G_{x_0}(\Q)\backslash G_{x_0}(\Q_\ell) \times G_{x_0} (\A_f^{(\ell)})/G_{x_0}(\Z_\ell) \times G_{x_0}(\Z_p)U_n^\ell,\] where $U_n^\ell\subseteq G_{x_0}(\A_f^{(p\ell)})$ is the prime-to-$\ell$ part of $U_n$. Since $G_{x_0}(\R\times \Q_\ell)$ is non-compact and $G_{x_0}$ is simply connected, strong approximation holds for $G_{x_0}$ with respect to the places $\{\infty, \ell\}$, cf.~\cite{kneser:sa}, and $G_{x_0}(\Q)$ is dense in $G_{x_0}(\A_f^{(\ell)})$ for the adelic topology. Thus, every class in $DS$ is represented by an element in $G_{x_0}(\Q_\ell)$ and the proposition is proved. 
\end{proof}

\begin{proof}[Proof of Theorem~\ref{thm:oc}]
Since $g$ is even and $p\ge 5$, by Theorem~\ref{thm:A1}, any point $(X,\lambda,[\eta]_{U_n})\in \calS_{g,r,[\F_{p^2}],n}^{\rm eo,0}(\Fpbar)$ has automorphism group $\Aut(X,\lambda)=\{\pm 1\}$. 
Thus, it suffices to show that for any irreducible component $Y$ of $\calS_{g,n}$, the intersection $Y\cap \calS_{g,r,[\F_{p^2}],n}^{\rm eo,0}$ is non-empty, because the automorphism group of a generalisation decreases.

Let $W'$ be an irreducible component of $\calS_{g,r,[\F_{p^2}],n}^{\rm eo,0}$ and let $Y'\subseteq \calS_{g,n}$ be an irreducible component containing $W'$. By Proposition~\ref{prop:hecke}, there exist an $\ell$-adic Hecke correspondence $(\calH_u,\pr_1, \pr_2)$ and an irreducible component $\wt Y\subseteq \calH_u$ such that $\pr_1(\wt Y)=Y'$ and $\pr_2(\wt Y)=Y$. Since $\calS_{g,r,[\F_{p^2}],n}^{\rm eo,0}$ is $\ell$-adic Hecke invariant, we have $\pr_2(\pr_1^{-1}(\calS_{g,r,[\F_{p^2}],n}^{\rm eo,0})\cap \wt Y)\subseteq \calS_{g,r,[\F_{p^2}],n}^{\rm eo,0}$, and hence that $Y\cap \calS_{g,r,[\F_{p^2}],n}^{\rm eo,0}$ is non-empty.
\end{proof}

\section{Oort's Conjecture In Dimension $\bm{g=4}$}\label{sec:g4}

The goal of this section to prove Oort's conjecture in dimension~$g=4$ and in any characteristic~$p\ge 3$ with a counterexample in $p=2$:

\begin{theorem}\label{thm:ocg=4}
Every generic $4$-dimensional principally polarised supersingular abelian variety $(X,\lambda)$ over an algebraically closed field $k$ of characteristic $p$ has automorphism group $\{\pm 1\}$ if $p>2$, and $\{\pm 1\}^2$ if $p=2$.
\end{theorem}

A generic $4$-dimensional principally polarised supersingular abelian variety $(X,\lambda)$ over $k$ has $a$-number $a(X) = 1$. Let $E$ be a supersingular elliptic curve over $k$. By Li-Oort~\cite{lioort}, for each generic $(X,\lambda)$ as above, there exist a choice of polarisation $\eta$ on $E^4$ with $\ker(\eta) \simeq E^4[\mathsf{F}^3]$ of order $p^{12}$ and an associated polarised flag type quotient (PFTQ)
\begin{equation}\label{eq:PFTQ}
(Y_3, \lambda_3) \simeq (E^4, \eta) \longrightarrow (Y_2,\lambda_2) \longrightarrow (Y_1,\lambda_1) \longrightarrow (Y_0, \lambda_0) = (X,\lambda).
\end{equation}
The composition of the maps in~\eqref{eq:PFTQ} yields an isogeny $\alpha: (E^4, \eta) \to (X,\lambda)$ whose kernel~$H$ of order $p^6$ is an isotropic subspace of $\ker(\lambda_3) = \ker(\mathsf{F}^3)$ with respect to the Weil pairing induced from $\lambda_3$.

For $0 \leq i \leq 3$, let $M_i$ denote the contravariant Dieudonn{\'e} module associated with~$Y_i$. 
We have 
\begin{equation}\label{eq:V3M3M0M3}
\mathsf{V}^3M_3 \subseteq_{6} M_0 \subseteq_{1} M_1 \subseteq_{2} M_2 \subseteq_{3} M_3,  
\end{equation}
where we write $M'\subseteq_{i} M''$ for $W$-modules $M'$ and $M''$ if ${\rm length}_W M''/M'=i$. Here $W:=W(k)$ denotes the ring of Witt vectors over $k$. Moreover, we have $(\sfF,\sfV)M_{i+1} \subseteq M_{i}\subseteq M_{i+1}$ and $\dim_k M_i/(\sfF,\sfV)M_{i+1}=1$. 

On each Dieudonn{\'e} module $M_i$ we have the polarisation $\langle\, , \rangle_i$ induced from the polarisation~$\lambda_i$ in~\eqref{eq:PFTQ}; equivalently, $\langle\, , \rangle_i$ is the restriction of $\langle\, , \rangle_3$ on $M_i$ for $0 \leq i \leq 3$. We shall write $\langle\, , \rangle$ for $\langle\, , \rangle_i$ for simplicity. Since $M_{i+1}$ is uniquely determined by $M_i$ by rigidity, we have for $i=0,1,2$ that
\begin{equation}\label{eq:mpM0mpM1}  
\End_{\rm DM}(M_i)\subseteq \End_{\rm DM}(M_{i+1}), \quad \Aut_{\rm DM}(M_i, \<\, , \>)\subseteq \Aut_{\rm DM}(M_{i+1},\<\, , \>).\end{equation}
For $0 \leq i \leq 3$, let $M_i^t$ be the dual lattice of $M_i$ with respect to $\langle\, , \rangle$, which is isomorphic to the \dieu module of the dual abelian variety $Y^{\vee}_i$ of $Y_i$. By $\ker \lambda_3=Y_3[\sfF^3]$, we have $\sfV^3 M_3 = \sfF^3M_3 =M_3^t$.
Below, we will also write $\mathsf{V}^3M_3$ as $p\mathsf{V}M_3$.

\subsection{Explicit description of Dieudonn{\'e} modules}\

Based on results of Harashita \cite{harassg4}, we give the following explicit description of the Dieudonn{\'e} modules $M_i$ for $0 \leq i \leq 3$. First we choose a standard $W$-basis
\begin{equation}\label{eq:M3}
M_3 = \langle x_1, \ldots, x_4, y_1, \ldots, y_4 \rangle_W,
\end{equation}
where $\mathsf{F}x_i = y_i = -\mathsf{V}x_i$ and $\mathsf{F}y_i = -px_i = -\mathsf{V}y_i$, and where the basis elements satisfy

\begin{equation}\label{eq:pairing}
\langle x_1, x_2 \rangle = \frac{1}{p^2} = \langle x_3, x_4 \rangle, \qquad \langle y_1, y_2 \rangle = p \langle x_1, x_2 \rangle^{\sigma} = \frac{1}{p} = \langle y_3, y_4 \rangle,
\end{equation}
and the other pairings are zero. This follows from \cite[Lemma 6.1]{lioort} and the fact that $\sfV^3 M_3=M_3^t$.  Next, for $\mathsf{V}M_3 \subseteq M_2 \subseteq M_3$ we have
\begin{equation}\label{eq:M2}
M_2 = \langle \boldsymbol{z}, \mathsf{F}x_i = y_i \text{ for $1 \leq i \leq 4$} \rangle_W + pM_3 = \langle \boldsymbol{z} \rangle_W + \mathsf{V}M_3
\end{equation}
where $\boldsymbol{z} := \tilde{t}_1 x_1 + \tilde{t}_2 x_2 + \tilde{t}_3 x_3 + \tilde{t}_4 x_4$ and the reductions $t_i = \tilde{t}_i \bmod p$ satisfy the equation
\begin{equation}
    \label{eq:f1}
    t_1 t_2^{p^2} - t_2 t_1^{p^2} + t_3 t_4^{p^2} - t_4 t_3^{p^2} = 0.
\end{equation}
Since $\dim_k M_2/\sfV M_3=1$, we have $\dim_k (\sfF,\sfV)M_2/\sfV^2 M_3 \in \{1,2\}$, and furthermore we see that $\dim_k (\sfF,\sfV)M_2/\sfV^2 M_3 = 1$ if and only if $(t_1:t_2:t_3:t_4)\in \bbP^3(\F_{p^2})$. This shows that $a(M_2)\in \{3,4 \}$, and that $a(M_2)=4$ if and only if $(t_1:t_2:t_3:t_4)\in \bbP^3(\F_{p^2})$. Thus, the generic case where $M_2$ is not superspecial happens when the one-dimensional  subspace $\<\boldsymbol{z} \bmod p \>$ is not defined over $\mathbb{F}_{p^2}$.  Moreover, the condition $a(M_0)=1$ implies that
$M_3$ is the smallest superspecial \dieu module containing $M_0$~ \cite[Theorem~2.2]{oda-oort}.   
Thus, $M_3$ is the smallest superspecial \dieu module containing $M_2$ and $(t_1:t_2:t_3:t_4)$ is away from any $\F_{p^2}$-rational hyperplane of $\bbP^3$. \\
Without loss of generality we may set $\tilde{t}_1 = 1$. In the generic case where $(t_1:t_3) = (1:t_3)$ $\not\in \mathbb{P}^1(\mathbb{F}_{p^2})$, for $(\mathsf{F},\mathsf{V})M_2 \subseteq M_1 \subseteq M_2$ we further write
\begin{equation}\label{eq:M1}
M_1 = \langle \boldsymbol{u}, (\mathsf{F},\mathsf{V})M_2 \rangle_W = \langle \boldsymbol{u}, \mathsf{F}\boldsymbol{z}, \mathsf{V}\boldsymbol{z} \rangle_W + pM_3, 
\end{equation}
where $\boldsymbol{u} = \tilde{u}_1 z + \tilde{u}_2 \mathsf{F}x_2 + \tilde{u}_3 \mathsf{F}x_4$ and the reductions $u_i = \tilde{u}_i \bmod p$ satisfy the equation\footnote{Note that there is small difference between our defining equation with that in~\cite[Section~4.3]{harassg4}, due to that we use the relation $\sfF^2 x_i=-px_i$ while Harashita uses the relation $\sfF^2 x_i=px_i$ for all $i$. Using our choice, the $\F_{p^2}$-structure for $M_3$ is compatible with the $\F_{p^2}$-structure for the moduli space of PFTQs. Notice that Equation~\eqref{eq:f2} is the defining equation for $M_1$ with a fixed $M_2$ whose homogeneous coordinates are $(t_1:t_2:t_3:t_4)$, and which depends only on $(t_1:t_2:t_3:t_4)$. In particular, its coefficients are homogeneous polynomials in $t_i$. To get the defining equation, we compute the defining equation with $t_1=1$ and homogenise the coefficients.}  

\begin{equation}
    \label{eq:f2}
    u_1(t_1^p u_2^p + t_1^{p-1} t_3 u_3^p + t_1^p u_2u_1^{p-1} +  t_3^p u_3 u_1^{p-1}) = 0.
\end{equation}

The condition $u_1 = 0$ is equivalent to the condition $M_1 \subseteq \mathsf{V}M_3$ which would yield $a(M_1) = 3$, while a rigid PFTQ will satisfy $a(M_1) = 2$. In other words, the ``non-garbage component'' $u_1 \neq~0$ contains the rigid PFTQs over $(X,\lambda)$; from now on, we will work only on this component, simplifying notation by setting $\tilde{u}_1 = 1$. 
Then in~\eqref{eq:M1} also note that 
\[
(\mathsf{F}, \mathsf{V})M_2 = \langle \mathsf{F}\boldsymbol{z}, \mathsf{V}\boldsymbol{z}\rangle_W + pM_3 = \langle \mathsf{F}\boldsymbol{u}, \mathsf{V}\boldsymbol{u}\rangle_W + pM_3.
\]
We further read off from~\eqref{eq:M2} and~\eqref{eq:M1} that 
\[
M_1 = \langle \boldsymbol{u} \rangle_W + (\mathsf{F},\mathsf{V})M_2 + pM_3 = \langle \boldsymbol{u}, \mathsf{F}\boldsymbol{z}, \mathsf{V}\boldsymbol{z}, px_1, px_2, px_3, px_4 \rangle_W + p\mathsf{V}M_3.
\]
Hence, 
\[
(\mathsf{F},\mathsf{V})M_1 = \langle \mathsf{F}\boldsymbol{u}, \mathsf{V}\boldsymbol{u}, \mathsf{F}^2\boldsymbol{z}, p\boldsymbol{z}, \mathsf{V}^2\boldsymbol{z} \rangle_W + p\mathsf{V}M_3.
\]
Since $\langle px_1, px_2, px_3, px_4 \rangle \equiv \langle px_1, \mathsf{F}^2\boldsymbol{z}, p\boldsymbol{z}, \mathsf{V}^2\boldsymbol{z}\rangle \bmod p\mathsf{V}M_3$, we may rewrite this as 
\[
(\mathsf{F},\mathsf{V})M_1 + \langle px_1 \rangle_W = \langle \mathsf{F}\boldsymbol{u}, \mathsf{V}\boldsymbol{u}\rangle_W + pM_3.
\]
This in turn implies that $M_1/(\mathsf{F},\mathsf{V})M_1$ is generated over $W$ by $\boldsymbol{u}$ and $px_1$. Hence, we may choose the line 
\[
\boldsymbol{s} := s_1 \boldsymbol{u} + s_2 (px_1)
\]
to parametrise the one-dimensional subspace $M_0$, i.e.,
\begin{equation}\label{eq:M0}
M_0 = \langle \boldsymbol{s} \rangle_W + (\mathsf{F},\mathsf{V})M_1 = \langle \boldsymbol{s}, \mathsf{F}\boldsymbol{u}, \mathsf{V}\boldsymbol{u}, \mathsf{F}^2\boldsymbol{z}, p\boldsymbol{z}, \mathsf{V}^2\boldsymbol{z} \rangle_W + p\mathsf{V}M_3.
\end{equation}
Since a generic $M_0$ does not contain $pM_3$, we need to choose $s_1 \neq 0$; we simplify the notation by setting $s_1 = 1$ and writing $s = \boldsymbol{u} + s_2 (px_1)$.\\

The above descriptions yield the following diagram:
\begin{equation}\label{eq:diagram}
\begin{smallmatrix}
M_0 & \subseteq_1 & M_1 & \subseteq_2 & M_2 & \subseteq_3 & M_3 \\
\rotatebox{90}{$\subseteq$}_1 &  & \rotatebox{90}{$\subseteq$}_1 &  & \rotatebox{90}{$\subseteq$}_1 &  & \\
(\mathsf{F},\mathsf{V})M_1 & \subseteq_1 & (\mathsf{F},\mathsf{V})M_2 & \subseteq_2 & \mathsf{V} M_3 & & \\
\rotatebox{90}{$\subseteq$}_2 &  & \rotatebox{90}{$\subseteq$}_2 &  & & &\\
(\mathsf{F},\mathsf{V})^2 M_2 & \subseteq_1 & p M_3 & & & &\\
\rotatebox{90}{$\subseteq$}_3 & & & & & & \\
p \mathsf{V} M_3 & & & & & &\\
\end{smallmatrix}   
\end{equation}

We can now confirm the statements made in the previous subsection.
\begin{lemma}
We have $p\mathsf{V}M_3 \subseteq M_0$. Hence, choosing the $M_i$ as in \eqref{eq:M3}--\eqref{eq:M0}, the composition $\alpha: (E^4,\eta) \to (X,\lambda)$ of the maps in the PFTQ in \eqref{eq:PFTQ} has kernel $H := \ker(\alpha)$ satisfying $H \subseteq \ker(\mathsf{F}^3)$.
\end{lemma}

\begin{proof}
The first statement can be read off from~\eqref{eq:M0}. Indeed, by construction, we know that $\ker(\lambda_3) = \ker(\mathsf{F}^3)$ and $\mathsf{F}^3 M_3= \mathsf{V}^3 M_3$ as $M_3$ is superspecial. On the other hand, we know that $(M_0, \langle\, , \rangle_0)$ has length $p^3 \cdot p^2 \cdot p = p^6$ inside $(M_3, \langle\, , \rangle_3)$. 
Denoting the contravariant Dieudonn{\'e} functor by $\mathbb{D}$, furthermore $p\mathsf{V} M_3 \subseteq M_0$ implies that $\mathbb{D}(\ker(\mathsf{F}^3)) \simeq M_3/\mathsf{F}^3M_3=M_3/\mathsf{V}^3M_3 \twoheadrightarrow \mathbb{D}(H) \simeq M_3/M_0$, or equivalently, that $H \subseteq \ker(\mathsf{F}^3)$.
\end{proof}

\subsection{Endomorphisms and automorphisms of Dieudonn{\'e} modules}\

As in Lemma~\ref{lm:Vp}, let $D_p$ be the division quaternion algebra over $\mathbb{Q}_p$ and let $O_p$ denote its maximal order. We also write $D_p = \mathbb{Q}_{p^2}[\Pi]$ and $O_p = \mathbb{Z}_{p^2}[\Pi]$, where $\mathbb{Z}_{p^2} = W(\mathbb{F}_{p^2})$ and $\mathbb{Q}_{p^2} = \mathrm{Frac}(W(\mathbb{F}_{p^2}))$,
and where $\Pi^2 = -p$ and $\Pi a = a^{\sigma} \Pi$ for any $a \in \mathbb{Q}_{p^2}$. Here $a \mapsto a^\sigma$ denotes the non-trivial automorphism of $\mathbb{Q}_{p^2}/\mathbb{Q}_p$. If we let $*$ denote the canonical involution of~$D_p$, then $a^* = {a}^\sigma$ for
any $a \in \mathbb{Q}_{p^2}$, and $\Pi^* = - \Pi$.

Let $M_3^{\diamond} = \{ m \in M_3 : (\mathsf{F}+\mathsf{V})m = 0\}$ denote the skeleton of $M_3$. It is a \dieu module over $\F_{p^2}$ and we have $\End_{\rm DM}(M_3)=\End_{\rm DM}(M_3^\diamond)$.

We first consider the endomorphisms and automorphisms of $M_3$. 
Namely, it follows from Equation~\eqref{eq:M3} that $\mathrm{End}_{\rm DM}(M_3) \simeq \mathrm{Mat}_4(O_p)$. Hence $\mathrm{Aut}(M_3) \simeq \mathrm{GL}_4(O_p)$, and $\mathrm{Aut}(M_3, \langle, \rangle_3)$ consists of those automorphisms preserving $\langle, \rangle_3$. 

Since $O_p = \mathbb{Z}_{p^2}[\Pi]$, we may write any element $g \in \mathrm{End}_{\rm DM}(M_3)$ as $A_0 + B_0\Pi$. Indeed, $\mathrm{Mat}_4(\mathbb{Z}_{p^2}[\Pi]) = \mathrm{Mat}_4(\mathbb{Z}_{p^2}) + \mathrm{Mat}_4(\mathbb{Z}_{p^2}) \Pi$ when we identify $\Pi \mapsto \left( \begin{smallmatrix}
    0_4 & -p \mathbb{I}_4 \\ \mathbb{I}_4 & 0_4
\end{smallmatrix} \right)$, so 
\[
g = A_0 + B_0\Pi \mapsto \left( \begin{smallmatrix}
    A_0 & -pB_0^{\sigma} \\ B_0 & A_0^{\sigma}
\end{smallmatrix} \right)
\]
with identifications $A_0, B_0 \in \mathrm{Mat}_4(\mathbb{Z}_{p^2})$. 
We then have that $\mathsf{F}$ acts as $\Pi$ and $\mathsf{V}$ acts as $\Pi^*$.

We denote by $m_p$ the reduction-modulo-$pM_3$ map, and for any $n \geq 1$ by $m_{\mathsf{V}^n}$ the reduction-modulo-$\mathsf{V}^nM_3$ map. In particular, for any $M$ such that $\mathsf{V}^nM_3 \subseteq M \subseteq M_3$ we get $m_{\mathsf{V}^n}(M) = M/\mathsf{V}^nM_3$. Note that any $g \in \mathrm{End}_{\rm DM}(M_3)$ also satisfies $g(\mathsf{V}^nM_3) \subseteq \mathsf{V}^nM_3$ for any $n \geq 1$, so $g \in \mathrm{End}_{\rm DM}(\mathsf{V}^n M_3)$ as well. Finally, recall that both $\mathsf{F}$ and $-\mathsf{V}$ act as $\Pi$ on~$M_3^{\diamond}$ and note that $\mathrm{End}_{\rm DM}(M_3)/\mathrm{End}_{\rm DM}(M_3)\Pi^n \hookrightarrow \mathrm{End}_{\rm DM}(m_{\mathsf{V}^n}(M_3^\diamond))$ for all $n \geq 1$. In particular, it follows that 
\begin{align}\label{eq:mV3M3}
& \mathrm{End}_{\rm DM}(M_3)/\mathrm{End}_{\rm DM}(M_3)\Pi^3 \simeq \mathrm{Mat}_4(O_p/\Pi^3) \\ \nonumber
& \simeq \left\{ \left( \begin{smallmatrix}
    A & -pB^{\sigma} \\ B & A^{\sigma}
\end{smallmatrix} \right) : A \in \mathrm{Mat}_4(\mathbb{Z}_{p^2}/p^2\mathbb{Z}_{p^2}), B \in \mathrm{Mat}_4(\mathbb{Z}_{p^2}/p\mathbb{Z}_{p^2}) \right\}.
\end{align}
Note that the multiplicative structure is given by
\[
\left( \begin{smallmatrix}
    A & -pB^{\sigma} \\ B & A^{\sigma}
\end{smallmatrix} \right) \cdot \left( \begin{smallmatrix}
    C & -pD^{\sigma} \\ D & C^{\sigma} \end{smallmatrix} \right) = \left( \begin{smallmatrix}
    AC - p\tilde{B}^{\sigma}\tilde{D} & -p\left(B\overline{C} + \overline{A}^{\sigma}D \right)^\sigma \\ B\overline{C} + \overline{A}^{\sigma}D &   (AC-p\tilde{B}^{\sigma}\tilde{D})^{\sigma}
\end{smallmatrix} \right),
\]
where $\tilde{B}$ denotes a lift of $B$ from $\mathrm{Mat}_4(\mathbb{Z}_{p^2}/p\mathbb{Z}_{p^2})$ to $\mathrm{Mat}_4(\mathbb{Z}_{p^2}/p^2\mathbb{Z}_{p^2})$ and $\overline{A}$ denotes the reduction of $A$ from $\mathrm{Mat}_4(\mathbb{Z}_{p^2}/p^2\mathbb{Z}_{p^2})$ to $\mathrm{Mat}_4(\mathbb{Z}_{p^2}/p\mathbb{Z}_{p^2})$.
We have the exact sequence
\begin{equation}\label{eq:sesM3}
0 \longrightarrow \mathrm{End}_{\rm DM}(M_3)\Pi^2/\mathrm{End}_{\rm DM}(M_3)\Pi^3 \longrightarrow \mathrm{End}_{\rm DM}(M_3)/\mathrm{End}_{\rm DM}(M_3)\Pi^3 \xrightarrow{m_p} \mathrm{End}_{\rm DM}(m_p(M_3^\diamond)) 
\end{equation}
and we read off from~\eqref{eq:mV3M3} that for an element $\tilde{g} \in \mathrm{End}_{\rm DM}(M_3)/\mathrm{End}_{\rm DM}(M_3)\Pi^3$ we get
\begin{equation}\label{eq:mpg}
m_p(\tilde{g}) = \left( \begin{smallmatrix}
    \overline{A} & 0 \\ B & \overline{A}^{\sigma}
\end{smallmatrix} \right) \in \mathrm{Mat}_8(\mathbb{F}_{p^2}).
\end{equation}

Next, we consider the endomorphism ring of $M_1$. 
It follows from the diagram in~\eqref{eq:diagram} that $p\mathsf{V}M_3 = \mathsf{V}^3M_3 \subseteq M_1$, so we may consider $m_{\mathsf{V}^3}(M_1)$, as well as $m_p(M_1)$ (cf.~\eqref{eq:M1}). 

By~\eqref{eq:M1} we have $M_1 = \langle \boldsymbol{u}, \mathsf{F}\boldsymbol{z}, \mathsf{V}\boldsymbol{z}\rangle_W + pM_3$, where $\boldsymbol{z} = x_1 + \tilde{t}_2 x_2 + \tilde{t}_3x_3 + \tilde{t}_4 x_4$ and $\boldsymbol{u} = \boldsymbol{z} + \tilde{u}_2y_2 + \tilde{u}_3y_4$.
Hence we can choose a basis for $m_p(M_1)$ consisting of 
\begin{align*}
    \boldsymbol{u}  &=  x_1 + t_2 x_2 + t_3x_3 + t_4 x_4 + u_2y_2 + u_3y_4, \\
    \boldsymbol{v}_1 &=  y_1 + t''_2 y_2 + t''_4 y_4, \\
    \boldsymbol{v}_2 &=  t'_2 y_2 + y_3 + t'_4 y_4,
\end{align*}
with coordinates $(t_2,t_3,t_4,u_2,u_3)$ subject to the relations~\eqref{eq:f1} and~\eqref{eq:f2} for $t_1=1$ and $u_1=1$ where 
\begin{equation}\label{eq:ti'}
    t'_2 = \frac{(t^p_2 - t^{1/p}_2)}{(t^{p}_3 - t^{1/p}_3)}, \quad t'_4 = \frac{(t^{p}_4 - t^{1/p}_4)}{(t^{p}_3 - t^{1/p}_3)}, \quad  t''_2 = t_2^{p} - t_3^{p} t'_2, \quad t''_4 = t_4^{p}- t_3^{p} t'_4,
\end{equation}
which can be viewed as well-defined algebraic functions in $t_2,t_3,t_4$ with $t_3\not\in \F_{p^2}$. 
One computes
\[ t_2^{\prime p}=\frac{t_2^{p^2}-t_2}{t_3^{p^2}-t_3}=\frac{t_4 t_3^{p^2}-t_3t_4^{p^2}}{t_3^{p^2}-t_3},\quad t_4^{\prime \prime p}=\frac{t_4^{p^2}(t_3^{p^2}-t_3)-t_3^{p^2}(t_4^{p^2}-t_4)}{t_3^{p^2}-t_3}=\frac{t_4 t_3^{p^2}-t_3t_4^{p^2}}{t_3^{p^2}-t_3}\]
and hence $t_2'=t_4''$.

Consider the affine variety $Y$ over $\F_{p^2}$ whose coordinate ring is defined by 
\[ \F_{p^2}[Y]=\mathbb{F}_{p^2}[t_2,t_3,t_4,u_2,u_3]/(t_2^{p^2}-t_2+t_3 t_4^{p^2}-t_4 t_3^{p^2}, u_2^p+t_3  u_3^p+u_2+u_3  t_3^p), \]
and the affine variety $Z$ over $\F_{p^2}$ defined by
\[ \F_{p^2}[Z]=\mathbb{F}_{p^2}[t_2,t_3,t_4]/(t_2^{p^2}-t_2+t_3 t_4^{p^2}-t_4 t_3^{p^2}).\]
Then both $Y$  and $Z$ are geometrically irreducible, and we have 
\[ \F_{p^2}[Y]= \F_{p^2}[Z][u_2,u_3]/(u_2^p+t_3  u_3^p+u_2+u_3  t_3^p)=\bigoplus_{i=0}^{p-1}\F_{p^2}[Z][u_3] u_2^i.\]

\begin{lemma}\label{lm:lin_indp} 
Set 
\[ (\alpha_0,\alpha_1,\dots, \alpha_8):=(1,t_2,t_3,t_4,t_2',t_4',t_2'',u_3,u_2), \]
where $t_2',t_4',t_2''$ are defined in Equation~\eqref{eq:ti'}.
  Then there exists an element $x=(t_2,t_3,t_4,u_2,u_3)\in Y(k)$ such that the specialisations $\alpha_i(x)\alpha_j(x)$ of $\alpha_i \alpha_j$ at $x$, for $0\le i \leq j\le 8$ in $k$, are $\F_{p^2}$-linearly independent.   
\end{lemma}

\begin{proof}
    We regard the $\alpha_i$ as elements in
    \[ \F_{p^2}[Y][t_2^{\frac{1}{p}}, t_3^{\frac{1}{p}}, t_4^{\frac{1}{p}}][1/(t_3^p-t_3^{1/p})]=\bigoplus_{i=0}^{p-1} \left(\F_{p^2}[Z][t_2^{\frac{1}{p}}, t_3^{\frac{1}{p}}, t_4^{\frac{1}{p}}][1/(t_3^p-t_3^{1/p})]\right )[u_3] u_2^i. \] Then it is equivalent to show that the elements $\alpha_i \alpha_j$ for $0\le i \leq j\le 8$ are linearly $\F_{p^2}$-linearly independent in 
    the above $\F_{p^2}$-algebra. Indeed, if the elements $\alpha_i \alpha_j$ are $\F_{p^2}$-linearly independent, then for some point $x\in Y(k)\subseteq k^5$, their specialisations $\alpha_i(x) \alpha_j(x)$ are $\F_{p^2}$-linearly independent as $k$ is algebraically closed. Conversely, if there is an $\F_{p^2}$-linear relation among the elements $\alpha_i \alpha_j$, then no such specialisation exists.
    
    Suppose that we have an $\F_{p^2}$-linear relation \[ \sum_{0\le i\le j\le 8} a_{ij} \alpha_i \alpha_j=0, \quad  a_{ij}\in \F_{p^2}. \]
    Let us first assume that $\alpha_i \alpha_j$ for $0\le i \leq j\le 6$ are linearly independent. The terms involving $u_2$ are $u_2^2$ and $\gamma u_2$, where $\gamma\in \< \alpha_j\>_{\F_{p^2}, 0\le j \le 7}$. 
    If $a_{i8}\neq 0$ for some $i$, that is, if there is a linear relation involving $u_2^2$ and $\gamma u_2$, then 
    \[ u_2^2-\gamma u_2 \in \<\alpha_i \alpha_j\>_{\F_{p^2}, 0\le i\leq j\le 7}. \]
    This is impossible if $p>2$. If $p=2$, then we have
    \[ u_2^2=t_3u_3^2+u_2+t_3^2 u_3, \quad \gamma=1, \quad\text{and}\quad u_2^2-u_2=t_3u_3^2+t_3^2 u_3\in \<\alpha_i \alpha_j\>_{\F_{p^2}, 0\le i\leq j\le 7}. \]

This is also impossible, since the only term of $u_3$-degree $2$ in $\{\alpha_i \alpha_j\}_{0\le i\leq j\le 7}$ is $u_3^2$, which has total degree $2$ while $t_3 u_3^2$ has total degree $3$. 

We also have $a_{i7}=0$ for all $0\le i\le 7=0$ since $u_3$ is algebraically independent from $t_2,t_3,t_4$. 
    Under our assumption for $\alpha_i\alpha_j$ with $0\le i\leq j\le 6$, we get $a_{ij}=0$ for all $i,j$.
    
 We now show that the elements $\alpha_i \alpha_j$ for $0\le i\leq j\le 6$ are $\F_{p^2}$-linearly independent. Notice that this is equivalent to showing that the elements $\beta_i \beta_j$ for $0\le i\leq j\le 6$ are $\F_{p^2}$-linearly independent, where 
    \[ \beta_i:=(t_3^{p^2}-t_3)\alpha_i^p\in \F_{p^2}[Z]=\bigoplus_{j=0}^{p^2-1}\F_{p^2}[t_3,t_4]t_2^j. \]
    Suppose that we have an $\F_{p^2}$-linear relation $\sum_{0\le i\le j\le 6} b_{ij} \beta_i \beta_j=0$ with $b_{ij}\in \F_{p^2}$. We compute
    \[ \beta_0=t_3^{p^2}-t_3,   \quad   \beta_1= (t_3^{p^2}-t_3)t_2^p,  \quad \beta_2=(t_3^{p^2}-t_3)t_3^p,   \quad \beta_3=(t_3^{p^2}-t_3)t_4^p,  \]
    \[ \beta_4=t_4 t_3^{p^2}-t_3 t_4^{p^2},   \quad  \beta_5=t_4^{p^2}-t_4,   \quad\beta_6=(t_3^{p^2}-t_3)t_2 - t_3 (t_4 t_3^{p^2}-t_3 t_4^{p^2}). \]
    One has $\deg_{t_2} \beta_{1}=p$ and $\deg_{t_2} \beta_6=1$. The terms in $\{\beta_i \beta_j\}$ of positive $t_2$-degree are 
    $\beta_1^2$, $\beta_1\beta_6$, $\beta_6^2$, $\gamma_1 \beta_1$, $\gamma_2 \beta_6$, of respective $t_2$-degrees $2p$, $p+1$, $2$, $p$ and $1$, where $\gamma_1,\gamma_2\in \<\beta_j\>_{\F_{p^2}, j=0,2,3,4,5}$. If $p\ge 3$,  
    then the above degrees are all distinct and there is no linear relation among them. Therefore, $b_{ij}=0$ if some of $i,j$ is equal to $1$ or $6$. If $p=2$, then we have, in the reduced form,
    \[ \beta_1^2=(t_3^{4}-t_3)^2 t_2^4=(t_3^{4}-t_3)^2 (t_2+t_4 t_3^{4}-t_3 t_4^{4}),\quad \beta_6^2=(t_3^{4}-t_3)^2 t_2^2+t_3^2(t_4 t_3^{4}-t_3 t_4^{4})^2.\]
    If $b_{ij}\neq 0$ for some of $i, j$ equal to $1$ or $6$, then we have that either $\beta_1^2-\gamma_2 \beta_6$ (linear relation at $t_2$-degree 1) or $\beta_6^2-\gamma_1 \beta_1$ (linear relation at $t_2$-degree 2) lies in $\<\beta_i \beta_j\>_{\F_{p^2}, i,j=0,2,3,4,5}$. This could happen only when either $\gamma_1=(t_3^4-t_3)$ or $\gamma_2=(t_3^4-t_3)$ and we compute
    \[ \beta_1^2-(t_3^4-t_3) \beta_6=t_3^4 (t_3^4-t_3)(t_4 t_3^{4}-t_3 t_4^{4}),\quad
    \beta_6^2-(t_3^4-t_3) \beta_1=t_3^2(t_4 t_3^{4}-t_3 t_4^{4})^2.\]
    From the computation we see that neither $\beta_1^2-(t_3^4-t_3) \beta_6$ nor $\beta_6^2-(t_3^4-t_3) \beta_1$ lies in $\<\beta_i \beta_j\>_{\F_{p^2}, i,j=0,2,3,4,5}$ and hence $b_{ij}=0$ if some of $i,j$ is equal to $1$ or $6$. 

    It remains to show that there is no linear relation among $\beta_i \beta_j$ for $i\le j\in \{0,2,3,4,5\}$; the details for this step are omitted. 
\end{proof}

\begin{proposition}\label{prop:EndmpM1}
We have 
\begin{align*}
m_p(\mathrm{End}_{\rm DM}(M_1)) & = \{ \overline{g} \in m_p(\mathrm{End}_{\rm DM}(M_3)) : \overline{g}(m_p(M_1)) \subseteq m_p(M_1) \} \\
& = \{ a \cdot \mathbb{I}_8 : a \in \mathbb{F}_{p} \}.    
\end{align*}
\end{proposition}

\begin{proof}
By Lemma~\ref{lm:lin_indp}, one can choose $(t_2,t_3,t_4,u_2,u_3)\in Y(k)$ such that the products of two of $\{1,t_2, t_3, t_4,  t'_2, t'_4, t''_2, u_2, u_3\}$ are $\F_{p^2}$-linearly independent, that is, we have
\begin{equation}\label{eq:lin_indep}
    \dim_{\F_{p^2}} \F_{p^2}[t_2, t_3, t_4,  t'_2, t'_4, t''_2, u_2, u_3]_{\le 2}=\dim_{\F_{p^2}} \F_{p^2}[z_1,\dots, z_8]_{\le 2},
\end{equation}  
where the left hand side is an $\F_{p^2}$-subpace in $k$ and the right hand side is a polynomial ring.
Let 
\[P = \left( \begin{matrix}
    t_2 & 0 & 0 \\
    t_3 & 0 & 0 \\
    t_4 & 0 & 0 \\
    u_2 & t''_2 & t'_2 \\
    u_3 & t'_2 & t'_4
\end{matrix}\right)
\] 
so that $\left( \begin{smallmatrix}
   \mathbb{I}_3 \\ P 
\end{smallmatrix} \right)$
denotes $\boldsymbol{u}, \boldsymbol{v}_1, \boldsymbol{v}_2$ as column vectors in the basis $\{x_1, y_1, y_3, x_2, x_3, x_4, y_2, y_4\}$. Then to determine the endomorphisms of $m_p(M_1)$, we need to determine the matrices 
$\left( \begin{smallmatrix}
   A & B \\ C & D 
\end{smallmatrix} \right) \in \mathrm{Mat}_8(\mathbb{F}_{p^2})$ where $A, B, C, D$ are respectively $3 \times 3$, $3 \times 5$, $5 \times 3$, and $5 \times 5$ blocks,
satisfying 
\[
\left( \begin{smallmatrix}
   A & B \\ C & D 
\end{smallmatrix} \right) 
\left( \begin{smallmatrix}
   \mathbb{I}_3 \\ P 
\end{smallmatrix} \right)
= 
\left( \begin{smallmatrix}
   A + BP \\ C + DP
\end{smallmatrix} \right)
= 
\left( \begin{smallmatrix}
   \mathbb{I}_3 \\ P 
\end{smallmatrix} \right) Q  \text{\ \ for some $Q\in \Mat_3(k)$.}
\]
Then $Q=(A+BP)$ and we get a matrix equation 
\[
C + D P = P(A+BP).
\]
Using Equation~\eqref{eq:lin_indep}, we get $C=0$, $DP=PA$ and $PBP=0$. To simplify the notation, we rewrite
\[ P=\left( \begin{matrix}
    t_{11} & 0 & 0 \\
    t_{21} & 0 & 0 \\
    t_{31} & 0 & 0 \\
    t_{41} & t_{42} & t_{52} \\
    t_{51} & t_{52} & t_{53}
\end{matrix}\right)=\sum_{(i,j)\in I} t_{ij} E_{ij} +t_{52} E_{52}',
\] 
where $I=\{(i,1), 1\le i\le 5, (4,2), (5,3)\}$  and $E_{52}'=E_{52}+E_{43}$. We compute
\[ PBP=\sum_{(i_1,j_1), (i_2,j_2)\in I} t_{i_1j_1} t_{i_2j_2} E_{i_1j_1}B E_{i_2 j_2} + \sum_{(i_1,j_1)\in I} t_{i_1 j_1} t_{52} (E_{i_1j_1}B E_{52}'+ E_{52}'B E_{i_1j_1}) +t_{52}^2 E_{52}'B E_{52}'. \]
Write $B=\sum_{1\le i\le 3,1\le j\le 5} b_{ij} E_{ij}$ with $b_{ij}\in \F_{p^2}$. It follows from~Equation~\eqref{eq:lin_indep} that
\[ E_{i_1j_1}B E_{i_2 j_2}+E_{i_2 j_2}B E_{i_1j_1}=b_{j_1i_2} E_{i_1j_2}+ b_{j_2 i_1} E_{i_2 j_1}  =0, \quad \forall\,  (i_1,j_1)\neq (i_2,j_2) \in I. \]
Taking 
\begin{itemize}
    \item $(i_1,j_1)=(i',1)$ and $(i_2,j_2)=(i,1)$ for $1\le i\le 5$, where $i'$ is any integer between $1$ and $5$ with $i'\neq i$;
    \item $(i_1,j_1)=(4,2)$ and $(i_2,j_2)=(i,1)$ for $1\le i\le 5$; and 
    \item $(i_1,j_1)=(5,3)$ and $(i_2,j_2)=(i,1)$ for $1\le i\le 5$, respectively,
\end{itemize} 
we get 
\[ b_{1i} E_{i_1,1} +b_{1,i_1} E_{i,1}=0, \quad b_{2i} E_{41}+b_{14} E_{i 2}=0, \quad b_{3i} E_{51}+b_{15} E_{i,3}=0.\]
It follows that $b_{ij}=0$ for all $1\le i\le 3, 1\le j\le 5$, and hence $B=0$.

Write $D =\sum_{1 \le k, l \le 5} d_{k l} E_{k l}$ and $A =\sum_{1 \le r, s \le 3} a_{rs} E_{rs}$ with coefficients in $\F_{p^2}$. We have
\[ DP=\sum_{(i,j)\in I} t_{ij} D E_{ij}+t_{52} D E'_{52}=\sum_{(i,j)\in I} t_{ij}E_{ij} A+t_{52}  E'_{52}A=PA.  \]
It follows from Equation~\eqref{eq:lin_indep}  that 
\[ D E_{ij} = E_{ij} A, \quad \forall\,(i,j) \in I \quad \text{and}\quad  DE'_{52}=E'_{52}A. \]
Thus, we have  
\[ \sum_{1 \leq k \leq 5} d_{k i} E_{k j}=\sum_{1\le s \le 3} a_{j s} E_{i s}, \quad \forall\,(i, j) \in I.\]
If $(k, j) \neq(i, s)$, that is, if $k \neq i$ or $s \neq j$, then we have $d_{k i}=0$ and $a_{j s}=0$.
Since we can take $i$ to be $1,\dots,5$ and $j$ to be $1,2,3$, it follows that $A$ and $D$ are diagonal. This implies 
\[
d_{i i} \cdot E_{i j}=a_{j j} E_{i j}, \quad \forall\, (i, j) \in I.
\]
From the table 
\[ 
\begin{tabular}{|c|c|c|c|c|c|c|c|}
\hline
$(i, j)$ & (1,1) & (2,1) & (3,1) & (4,1) & (5,1) & (4,2)  & (5,3) \\
\hline
\text{cond} & $d_{11}=a_{11}$ & $d_{22}=a_{11}$ & $d_{33}=a_{11}$ & $d_{44}=a_{11}$ & $d_{55}=a_{11}$ & $d_{44}=a_{22}$  & $d_{55}=a_{33}$ \\
\hline
\end{tabular} \]

we obtain that $\left( \begin{smallmatrix}
   A & B \\ C & D 
\end{smallmatrix} \right) = a \cdot \mathbb{I}_8$ for some $a \in \mathbb{F}_{p^2}$.

Finally, we note that $m_p(\mathrm{End}_{\rm DM}(M_1)) \subseteq m_p(\mathrm{End}_{\rm DM}(M_3))$ and compare with~\eqref{eq:mpg} to conclude that $a = a^p$, which implies that $a \in \mathbb{F}_p$.
\end{proof}

Observe that we have \[ m_{\sfV}(\End_{\rm DM}(M_1))=\F_{p} \subsetneq m_{\sfV}(\End_{\rm DM}(M_2))=\F_{p^2} \] from Proposition~\ref{prop:EndmpM1} even through $m_{\sfV}(M_1)=m_\sfV(M_2)$. This indicates that $m_{\sfV}(\End_{\rm DM}(M_1))$ is not determined by the information of $M_2$ modulo $\sfV M_3$.

\begin{corollary}\label{cor:EndmV3M1}
   (1) We have $m_{p}(\mathrm{Aut}_{\rm DM}(M_1,\<\, , \>_1)=\{\pm 1\}$. 
   
   (2) We have 
   \[
   m_{\mathsf{V}^3}(\mathrm{End}_{\rm DM}(M_1)) = \left\{ \left( \begin{smallmatrix}
       A & 0 \\ 0 & A^{\sigma}
   \end{smallmatrix} \right) \text{ for } A \in \mathrm{Mat}_4(\mathbb{Z}_{p^2}/p^2\mathbb{Z}_{p^2}) \text{ s.t. } A \equiv a\cdot \mathbb{I}_4 \bmod p, a 
   \in \mathbb{F}_p \right\}.
   \]
\end{corollary}

\begin{proof}
These follow from Proposition~\ref{prop:EndmpM1} and from Equation~\eqref{eq:mpg}. 
\end{proof}

Finally, we determine the endomorphisms and automorphisms of $M_0$ modulo $\mathsf{V}^3$; that is, we determine $m_{\mathsf{V}^3}(\mathrm{End}_{\mathrm{DM}}(M_0))$ as a subset of the set $m_{\mathsf{V}^3}(\mathrm{End}_{\mathrm{DM}}(M_1))$ determined in Corollary~\ref{cor:EndmV3M1}. 
Recall that~$M_0\supseteq (\mathsf{F},\mathsf{V})M_1$ is the lift of a one-dimensional submodule of $M_1/(\mathsf{F},\mathsf{V})M_1 = \langle \boldsymbol{u}, px_1 \rangle_W$; in fact, we have from Equation~\eqref{eq:M0} that
\begin{equation}\label{eq:mV3M0}
m_{\mathsf{V}^3}(M_0) = \langle \boldsymbol{s}, \mathsf{F}\boldsymbol{u}, \mathsf{V}\boldsymbol{u}, \mathsf{F}^2\boldsymbol{z}, p\boldsymbol{z}, \mathsf{V}^2\boldsymbol{z} \rangle = \langle \boldsymbol{s} \rangle + m_{\mathsf{V^3}}((\mathsf{F},\mathsf{V})M_1),
\end{equation}
where $\boldsymbol{s} = \boldsymbol{u} + s_2 px_1$ as before.

\begin{proposition}\label{prop:EndmV3M0}
The mod-$\sfV^3$ endomorphism ring $m_{\mathsf{V}^3}(\mathrm{End}_{\mathrm{DM}}(M_0))$ consists of all matrices $\left( \begin{smallmatrix}
       A  & 0 \\ 0 & A^{\sigma}  
   \end{smallmatrix} \right)$ with $A\in \Mat_4(\Z_{p^2}/p^2\Z_{p^2})$ of the form 
  \begin{equation}\label{eq:mV3EndM0}
    A=a \bbI_4+ p A_1, a\in \Z_p/p^2 \Z_{p^2}, A_1=\begin{pmatrix}
        a_{11} \bbI_2 & A_2 \\
        \bbJ_2^{-1} A_2^T \bbJ_2 & a_{33} \bbI_2
    \end{pmatrix},  A_2=\begin{pmatrix}
        a_{13} & a_{14} \\
        a_{23} & a_{24}
    \end{pmatrix},     
\end{equation}
where $\bbJ_2=\left( \begin{smallmatrix}
       0  & 1 \\ -1 & 0  
   \end{smallmatrix} \right)$. 
\end{proposition}
\begin{proof}
 Since elements of $m_{\mathsf{V}^3}(\mathrm{End}_{
 \mathrm{DM}}(M_1))$ automatically preserve $m_{\mathsf{V^3}}((\mathsf{F},\mathsf{V})M_1)$, by~\eqref{eq:mV3M0} it suffices to check when $\boldsymbol{s} = \boldsymbol{u}+ s_2px_1$ is also preserved. 

 Using the basis $\{x_1, x_2, x_3, x_4, y_1,  y_2, y_3, y_4 \}$ 
 this comes down to showing when an element $ \left( \begin{smallmatrix}
       A & 0 \\ 0 & A^{\sigma}
   \end{smallmatrix} \right) \in m_{\mathsf{V}^3}(\mathrm{End}_{\mathrm{DM}}(M_1))$ satisfies
 \begin{equation}\label{eq:mV3EndM0.1}
    \left( \begin{smallmatrix}
       A & 0 \\ 0 & A^{\sigma}
   \end{smallmatrix} \right) (1+s_2p, t_2, t_3, t_4,0, u_2, 0,u_3)^T \in \langle \boldsymbol{s}, \mathsf{F}\boldsymbol{u}, \mathsf{V}\boldsymbol{u}, \mathsf{F}^2\boldsymbol{z}, p\boldsymbol{z}, \mathsf{V}^2\boldsymbol{z} \rangle_{W/p^2W} \!\!\! \! \mod \sfV^3 M_3. 
 \end{equation}
 Replacing $A$ by $A-a \bbI_4$ for a $\sigma$-invariant element $a\in \Z_{p^2}/p^2\Z_{p^2}$, we may assume by Corollary~\ref{cor:EndmV3M1} that $A=(pa_{ij})$ for some $a_{ij}\in \Z_{p^2}/p^2\Z_{p^2}$. The right hand side of \eqref{eq:mV3EndM0.1} is generated by $\boldsymbol{s}, \mathsf{F}\boldsymbol{u}, \mathsf{V}\boldsymbol{u}, \mathsf{F}^2\boldsymbol{z}+p\boldsymbol{s},  \mathsf{V}^2\boldsymbol{z}+p\boldsymbol{s}$ and we have
 \[ \mathsf{F}^2\boldsymbol{z}+p\boldsymbol{s} \equiv p\left ((t_2-t_2^{(2)})x_2+(t_3-t_3^{(2)})x_3+(t_4-t_4^{(2)})x_4 \right ) \bmod \sfV^3 M_3,\]  
 \[ \mathsf{V}^2\boldsymbol{z}+p\boldsymbol{s} \equiv p\left ((t_2-t_2^{(-2)})x_2+(t_3-t_3^{(-2)})x_3+(t_4-t_4^{(-2)})x_4 \right ) \bmod \sfV^3 M_3.\]
 Then $\left( \begin{smallmatrix}
       A & 0 \\ 0 & A^{\sigma}
   \end{smallmatrix} \right) \boldsymbol{s}=c_1 \boldsymbol{s}+c_2 \mathsf{F}\boldsymbol{u}+c_3 \mathsf{V}\boldsymbol{u}+c_4(\mathsf{F}^2\boldsymbol{z}+p\boldsymbol{s})+c_5 (\mathsf{V}^2\boldsymbol{z}+p\boldsymbol{s})$ for some $c_i\in W/p^2W$. Comparing the coefficients of $x_1$, $y_1$ and $y_3$, we have $c_1=p(a_{11}+a_{12} t_2+a_{13} t_3+ a_{14} t_4)$, 
   $c_2,c_3\in pW/p^2W$. Thus, 
   \[\left( \begin{smallmatrix}
       A & 0 \\ 0 & A^{\sigma}
   \end{smallmatrix} \right) \boldsymbol{s}-c_1 \boldsymbol{s} \equiv c_4(\mathsf{F}^2\boldsymbol{z}+p\boldsymbol{s})+c_5 (\mathsf{V}^2\boldsymbol{z}+p\boldsymbol{s}) \bmod \sfV^3 M_3.
   \]
   We have $\left( \begin{smallmatrix}
       A & 0 \\ 0 & A^{\sigma}
   \end{smallmatrix} \right) \boldsymbol{s}-c_1 \boldsymbol{s}\equiv p(b_2 x_2+b_3x_3+b_4 x_4) \bmod \sfV^3 M_3$, where
   \begin{equation}\label{eq:bi}
   \begin{split}
       b_2& = a_{21}+(a_{22}-a_{11}) t_2+a_{23} t_3+a_{24} t_4- a_{12} t_2^2-a_{13} t_2 t_3- a_{14} t_2 t_4,  \\
       b_3&=a_{31}+a_{32} t_2+(a_{33}-a_{11}) t_3+a_{34} t_4-a_{12} t_2t_3-a_{13} t_3^2- a_{14} t_3 t_4,\\
       b_4&=a_{41}+a_{42} t_2+a_{43} t_3+(a_{44}-a_{11}) t_4 -a_{12} t_2t_4-a_{13} t_3t_4- a_{14} t_4^2. 
   \end{split}     
   \end{equation}

Using the relation $t_2^{p^2}-t_2=-t_3t_4^{p^2}+t_4 t_3^{p^2}$, we put
\begin{equation}\label{eq:tij}
    \begin{split}
     t_{34}& := \begin{vmatrix}
      t_3-t_3^{p^2} & t_3-t_3^{p^{-2}} \\ 
      t_4-t_4^{p^2} & t_4-t_4^{p^{-2}} \\ 
      \end{vmatrix}=(-t_4+t_4^{p^{-2}})t_3^{p^2}+(t_4^{p^2} - t_4^{p^{-2}})t_3+ (t_4-t_4^{p^2})t_3^{p^{-2}}, \\
   t_{24}&:=- \begin{vmatrix}
     t_2-t_2^{p^2} & t_2-t_2^{p^{-2}} \\ 
     t_4-t_4^{p^2} & t_4-t_4^{p^{-2}} \\ 
\end{vmatrix}=(t_4-t_4^{p^2})(t_2-t_2^{p^{-2}})+(t_4-t_4^{p^{-2}})(-t_3t_4^{p^2}+t_4 t_3^{p^2}),  \\
 t_{23}&:= \begin{vmatrix}
     t_2-t_2^{p^2} & t_2-t_2^{p^{-2}} \\ 
     t_3-t_3^{p^2} & t_3-t_3^{p^{-2}} \\ 
\end{vmatrix}=(t_3^{p^2}-t_3)(t_2-t_2^{p^{-2}})+(t_3^{p^{-2}}-t_3)(-t_3t_4^{p^2}+t_4 t_3^{p^2}). 
 \end{split}
\end{equation}
Then we have $t_{34} b_2+t_{24} b_3+ t_{23} b_4=0$, or equivalently, 
\begin{equation}
    \label{eq:mV3EndM0.2} 
    t_{34}^{p^2} b_2^{p^2}+t_{24}^{p^2} b_3^{p^2}+ t_{23}^{p^2} b_4^{p^2}=0.
\end{equation}
As the elements $t_2,t_3,t_4$ can be chosen arbitrarily subject to the equation $t_2^{p^2}-t_2+t_3t_4^{p^2}-~t_4 t_3^{p^2}=0$, we may regard~\eqref{eq:mV3EndM0.2} as an equation in the ring $\F_{p^2}[t_2,t_3,t_4]/(t_2^{p^2}-t_2+t_3t_4^{p^2}-t_4 t_3^{p^2})$. Examining homogeneous components of the reduced form of Equation~\eqref{eq:mV3EndM0.2}, we obtain that the matrix $A_1=(a_{ij})$ satisfies the condition in Equation~\eqref{eq:mV3EndM0}; see the detailed computations in the next lemma. Thus, the assertion is proved.
\end{proof}
\begin{lemma}\label{lm:EndmV3M0}
    Let $a_{ij}\in \F_{p^2}$, for $1\le i,j\le 4$, that satisfy Equation~\eqref{eq:mV3EndM0.2}, where $b_i$ and $t_{ij}$ are defined in Equations~\eqref{eq:bi} and~\eqref{eq:tij}. Then the matrix $A_1=(a_{ij})$ satisfies the condition in Proposition~\ref{prop:EndmV3M0}. 
\end{lemma}
\begin{proof}
    Note that $(-1)^{p^2}=-1$ in $k$ 
    for 
    any prime $p$, so we have $(-x)^{p^2}=-x^{p^2}$ for any $x\in k$. Let $\ell_{i}$ be the linear term of $b_i$ for $i=2,3,4$, and $\alpha:=a_{12} t_2+a_{13}t_3+a_{14} t_4$ and $\delta:=-t_3 t_4^{p^2}+t_3^{p^2} t_4$. Then $\deg \delta=p^2+1$ and $t_2^{p^2}=t_2+\delta$. Then we have 
     \begin{equation}\label{eq:bip2}  
         \begin{split}
          b_2^{p^2}=a_{21}+\ell_2^{p^2}-\alpha^{p^2}(t_2+\delta), & \quad  \ell_2^{p^2}=(a_{22}-a_{11})(t_2+\delta)+a_{23} t_3^{p^2}+a_{24} t_4^{p^2}, \\
         b_3^{p^2}=a_{31}+\ell_3^{p^2}-\alpha^{p^2}t_3^{p^2}, &\quad  \ell_3^{p^2}=a_{32}(t_2+\delta)+(a_{33}-a_{11})  t_3^{p^2}+a_{34} t_4^{p^2},  \\
         b_4^{p^2}=a_{41}+\ell_4^{p^2}-\alpha^{p^2}t_4^{p^2}, &\quad  \ell_4^{p^2}=a_{42}(t_2+\delta)+ a_{43} t_3^{p^2}+(a_{44}-a_{11}) t_4^{p^2},
         \end{split} 
     \end{equation}
    and 
    \begin{equation} \label{eq:tijp2}        
    \begin{split}
      t_{34}^{p^2}&=(-t_4^{p^2}+t_4)t_3^{p^4}+(t_4^{p^4} - t_4)t_3^{p^2}+ (t_4^{p^2}-t_4^{p^4})t_3=-\delta^{p^2}-\delta+t_3^{p^4}t_4-t_3 t_4^{p^4}, \\
      t_{24}^{p^2}&=(t_4^{p^2}-t_4^{p^4})\delta+(t_4^{p^2}-t_4)\delta^{p^2},\\
      t_{23}^{p^2}&=(t_3^{p^4}-t_3^{p^2})\delta+(t_3-t_3^{p^2}) \delta^{p^2}. 
    \end{split}    
    \end{equation}    
    Observe that
    \begin{equation}\label{eq:cancelterm}
       \begin{split}
        t_{24}^{p^2} t_3^{p^2}+t_{23}^{p^2} t_4^{p^2}
        &=(t_4^{p^2}-t_4^{p^4})t_3^{p^2}\delta+(t_4^{p^2}-t_4)t_3^{p^2} \delta^{p^2}+(t_3^{p^4}-t_3^{p^2})t_4^{p^2}\delta+(t_3-t_3^{p^2})t_4^{p^2} \delta^{p^2} \\
        &=\delta^{p^2} \delta+(-\delta) \delta^{p^2}=0. 
       \end{split} 
    \end{equation}
    Then using Equations~\eqref{eq:mV3EndM0.2}, \eqref{eq:bip2} and~\eqref{eq:tijp2}, the three-term sum $t_{34}^{p^2} b_2^{p^2}+t_{24}^{p^2} b_3^{p^2}+ t_{23}^{p^2} b_4^{p^2}$ equals 
    \begin{equation} \label{eq:91.1}
    \begin{split}
        & t_{34}^{p^2}\left(a_{21}+\ell_2^{p^2}-\alpha^{p^2}(t_2+\delta)\right)
        +t_{24}^{p^2}(a_{31}+\ell_3^{p^2})+t_{23}^{p^2}(a_{41}+\ell_4^{p^2})-\alpha^{p^2}(t_{24}^{p^2} t_3^{p^2}+t_{23}^{p^2} t_4^{p^2} ) \\
        &=t_{34}^{p^2}\left ( a_{21}+\left( (a_{22}-a_{11})(t_2+\delta)+a_{23} t_3^{p^2}+a_{24} t_4^{p^2}\right )-\left(a_{12}(t_2+\delta)+a_{13} t_3^{p^2}+a_{14}t_4^{p^2}  \right ) (t_2+\delta)\right ) \\
        & \quad +t_{24}^{p^2}\left(a_{31}+\left(a_{32}(t_2+\delta)+(a_{33}-a_{11})  t_3^{p^2}+a_{34} t_4^{p^2}\right)\right)\\ 
        &\quad +t_{23}^{p^2}\left (a_{41}+\left(a_{42}(t_2+\delta)+ a_{43} t_3^{p^2}+(a_{44}-a_{11}) t_4^{p^2}\right)\right)=0.
    \end{split}    
    \end{equation}
    The coefficient of $t_2^2$ is $a_{12}t_{34}^{p^2}$, which gives $a_{12}=0$. The lowest degree term gives $a_{21}=0$. Then the coefficient of $t_2$ gives the relation
    \begin{equation}\label{eq:coefft_2}
       t_{34}^{p^2}\left((a_{22}-a_{11})-a_{13} t_3^{p^2}-a_{14} t_4^{p^2}\right )+t_{24}^{p^2}a_{32}+ t_{23}^{p^2} a_{42}=0. 
    \end{equation}
    Observe from~\eqref{eq:tijp2} that $t_{34}^{p^2}$ has only terms of degrees $p^4+p^2$, $p^4+1$, $p^2+1$  while $t_{24}^{p^2}$ and $t_{23}^{p^2}$ have only terms of degrees $p^4+p^2$, $p^4+p^2+1$, $2p^2+1$. 
     Thus, the degree $p^4+2p^2$ 
    terms of Equation~\eqref{eq:coefft_2} give
    \[ -\delta^{p^2}(-a_{13} t_3^{p^2}-a_{14} t_4^{p^2})+\delta^{p^2} t_4^{p^2} a_{32}-\delta^{p^2} t_3^{p^2} a_{42}=\delta^{p^2}\left((a_{13}-a_{42})t_3^{p^2}+(a_{14}+a_{32})t_4^{p^2}\right) =0 \]
    and hence $a_{42}=a_{13}$ and $a_{32}=-a_{14}$. From terms of degree $p^2+1$, we obtain $a_{22}-a_{11}=0$. 
    We check the terms of degrees $2p^2+1$ and $p^4+p^2+1$, and do not get any new linear relations. As $a_{22}-a_{11}=0$, the first term of Equation~\eqref{eq:coefft_2} has only degrees $p^4+2p^2, p^4+p^2+1, 2p^2+1$, and we have checked linear relations by all of them. Thus, with the above linear relations, we get the identity $a_{13}(t_{23}^{p^2}-t_3^{p^2} t_{34}^{p^2})-a_{14}(t_{34}^{p^2} t_4^{p^2}+t_{24}^{p^2})=0$ from Equation~\eqref{eq:coefft_2} and hence the relations 
    $t_{23}^{p^2}=t_3^{p^2} t_{34}^{p^2}$ and $t_{24}^{p^2}=-t_{34}^{p^2} t_4^{p^2}$.

    \indent With the relations \[ a_{12}=a_{21}=a_{22}-a_{11}=0,\ a_{42}=a_{13}, \ a_{32}=-a_{14}, \ t_{23}^{p^2}=t_3^{p^2} t_{34}^{p^2},\  t_{24}^{p^2}=-t_{34}^{p^2} t_4^{p^2} \] and $t_{34}^{p^2}\not\equiv 0$ and setting $t_2 = 0$, Equation~\eqref{eq:91.1} becomes 
    \begin{equation}\label{eq:91.2} 
      \begin{split}
        & \left( a_{23} t_3^{p^2}+a_{24} t_4^{p^2} -(a_{13} t_3^{p^2}+a_{14}t_4^{p^2})\delta)\right) -t_4^{p^2}\left( a_{31}-a_{14} \delta+(a_{33}-a_{11})  t_3^{p^2}+a_{34} t_4^{p^2}\right)\\
        & \quad +t_3^{p^2} \left(a_{41}+a_{13}\delta+ a_{43} t_3^{p^2}+(a_{44}-a_{11}) t_4^{p^2} \right) \\
        &=(a_{23}+a_{41})t_3^{p^2}+(a_{24}-a_{31}) t_4^{p^2}+(a_{44}-a_{33})t_3^{p^2}t_4^{p^2}+a_{34}t_4{2p^2}+a_{43}t_3^{2p^2}=0.
      \end{split}       
    \end{equation}
    From this, we get $a_{41}=-a_{23}, a_{31}=a_{24}, a_{44}=a_{33}, a_{34}=a_{43}=0$, and conclude
    \begin{equation}
   A_1:=(a_{ij})=\begin{pmatrix}
       a_{11} & 0 & a_{13} & a_{14} \\
       0 & a_{11} & a_{23} & a_{24} \\
       a_{24} & -a_{14} & a_{33} & 0 \\
       -a_{23} & a_{13} & 0 & a_{33}\\
   \end{pmatrix}= \begin{pmatrix}
       a_{11}\bbI_2 & A_2 \\
       \bbJ_2^{-1} A_2^T \bbJ_2 & a_{33}\bbI_2
   \end{pmatrix}. 
   \end{equation}
   Thus, the assertion is proved. 
\end{proof}

\begin{corollary}\label{cor:AutmV3M0}
The group $m_{\mathsf{V}^3}(\mathrm{Aut}_{\mathrm{DM}}(M_0, \langle, \rangle_0))$ is isomorphic to 
\[ \{\bbI_4+p A_1\in \Mat_4(\Z_{p^2}/p^2\Z_{p^2}): A_1 \text{\ is as in Equation~\eqref{eq:mV3EndM0}} \} \]
if $p=2$, and $\{\pm 1\}$ if $p\ge 3$. In particular, it is an elementary $2$-group when $p=2$.
\end{corollary}

\begin{proof}
We first determine the condition for matrices $h = \left( \begin{smallmatrix}
       A  & 0 \\ 0 & A^{\sigma} 
   \end{smallmatrix} \right)\in \Mat_8(\Z_{p^2})$ which preserve the pairing $\<\, , \>$ defined in \eqref{eq:pairing}, and then determine the condition for the matrices $h$ with $A$ of the form in \eqref{eq:mV3EndM0}. The Gram matrix of this pairing with respect to the basis $\{x_i, y_i\}$ is 
   \[ \begin{pmatrix}
       p^{-2} \bbJ_4  & 0 \\ 0 & p^{-1} \bbJ_4 
   \end{pmatrix}, \quad \bbJ_4=\begin{pmatrix}
       \bbJ_2  & 0 \\ 0 & \bbJ_2 
   \end{pmatrix},
   \]
   where $\bbJ_2=\left ( \begin{smallmatrix}
       0 & 1 \\ -1 & 0
   \end{smallmatrix} \right )$. So the adjoint of $h$ with respect to the pairing is given by $h^*=\left( \begin{smallmatrix}
       A^*  & 0 \\ 0 & (A^{\sigma})^* 
   \end{smallmatrix} \right)$ with $A^*=\bbJ_4^{-1} A^T \bbJ_4$, and we get the condition $A^* A=\bbI_4$. Now let $A=a\bbI_4+p A_1$ be a matrix as in \eqref{eq:mV3EndM0}; the condition shows that $A=\pm \bbI_4+pA_1$ and $A_1^*+A_1 \equiv 0 \bmod p$. One sees that $A_1^*=A_1$ and hence $2A_1\equiv 0 \bmod p$. This is automatically satisfied if $p=2$, and $A_1=0$ otherwise. When $p=2$, we may also express $A$ in the form $\bbI_4+pA_1$ with the same condition for $A_1$. 
\end{proof}

\begin{example}[Due to a referee] \label{ex:pm1pm1}
Let $p=2$ and $A=\diag(1,1,-1,-1)$. 
Then a direct computation shows
\[ \left (\begin{smallmatrix}
    A & 0 \\
    0 & A^{\sigma} \\
\end{smallmatrix}\right )  \boldsymbol{s} \equiv \boldsymbol{s}+ \frac{(t_2-t_2^{p^{-2}})}{t_{34}} (\sfF^2 \boldsymbol{z}+p\boldsymbol{s})+ \frac{(t_2^{p^2}-t_2)}{t_{34}} (\sfV^2 \boldsymbol{z}+ p\boldsymbol{s}) \bmod \sfV^3 M_3. \]
Thus, $\left (\begin{smallmatrix}
    A & 0 \\
    0 & A^{\sigma} \\
\end{smallmatrix}\right )\in \mathrm{Aut}_{\mathrm{DM}}(M_0, \langle\,, \rangle_0)_{\rm tors}\setminus \{\pm 1\}$. It turns out that this non-trivial element is contained in $\Aut(X,\lambda)$ and it generates $\Aut(X,\lambda)$ with $-1$.
\end{example}

We conclude with the proof of the main result. 

\begin{proof}[Proof of Theorem~\ref{thm:ocg=4}]

Let $(M_0,\langle\, , \rangle_0)$ denote the Dieudonn{\'e} module of $(X,\lambda)$.
Let $s=2$ if $p\ge 3$ and $s=3$ if $p=2$. Let $V_{p,s} := 1 + \Pi^s \mathrm{Mat}_4(O_{p})= 1 +  \mathrm{Mat}_4(O_{p}) \Pi^s\subseteq \mathrm{GL}_4(O_{p})$ be as in Lemma~\ref{lm:Vp}. Then it follows from Lemma~\ref{lm:Vp} that the reduction-modulo-$\mathsf{V}^s$ map 
\[
m_{\mathsf{V}^s}: \mathrm{Aut}(X,\lambda) \to m_{\mathsf{V}^s}(\mathrm{Aut}_{\mathrm{DM}}(M_0, \langle\, , \rangle_0))
\]
is injective, since $\ker(m_{\mathsf{V}^s}) \subseteq (V_{p,s})_{\mathrm{tors}}$ and the latter is trivial.  
In particular, the map 
 \[
m_{p}: \mathrm{Aut}(X,\lambda) \longrightarrow m_{p}(\mathrm{Aut}_{\mathrm{DM}}(M_0, \langle\,, \rangle_0))\subseteq m_{p}(\mathrm{Aut}_{\mathrm{DM}}(M_1, \langle\,, \rangle_1))=\{\pm 1\}
\]   
is injective when $p\geq 3$, where the first inclusion follows from Equation~\eqref{eq:mpM0mpM1} and the last equality follows from Proposition~\ref{prop:EndmpM1}. This proves the assertion for $p\ge 3$.

Now let $p=2$.
 By~\cite[Proposition~2.9]{ibukiyama-karemaker-yu} and \cite[Theorem~4.9]{lioort}, there is a unique polarised superspecial abelian fourfold $(Y_3, \lambda_3)$ up to isomorphism such that $\ker \lambda_3=Y_3[\sfF^3]$. Thus, for simplicity we may write $(Y_3, \lambda_3)=(Y_3', \lambda_3')\times (Y_3', \lambda_3')$, where $(Y_3', \lambda_3')$ is a polarised superspecial abelian surface $(Y_3', \lambda_3')$ with $\ker \lambda_3'=Y_3'[\sfF^3]$, which is also unique up to isomorphism by a similar (``class number one'') argument. It follows from \cite[Theorem 1]{Debarre} (also see~\cite[Corollary~1.5]{KTY}) that 
\[ \Aut(Y_3,\lambda_3)=\Aut(Y_3', \lambda_3')^2\cdot \<w\>, \]
where $w$ switches the components. 
The polarised \dieu module  $(M_3,\<\, ,\>)=(M_3',\<\, ,\>) \oplus (M_3'', \<\, ,\>)$ decomposes into \dieu submodules, where $M_3'$ and $M_3''$ are the \dieu modules of $Y_3'$ and the second copy of $Y_3'$, respectively, with the induced polarisations. One chooses a $W$-basis $x_1,x_2, y_1, y_2$ for $M_3'$ and $x_3,x_4, y_3, y_4$ for $M_3''$ with the same relations as in Subsection~7.1. 

We want to compute $\mathrm{Aut}(X,\lambda)$ via Equation~\eqref{eq:AutXLminisog}, using that $(Y_3,\lambda_3)$ is the minimal isogeny of $(X,\lambda)$.  
Then with the $W$-basis as above, by Corollary~\ref{cor:AutmV3M0}, the group $m_{\mathsf{V}^3}(\mathrm{Aut}_{\mathrm{DM}}(M_0, \langle\, , \rangle_0))$ consists of $8\times 8$ matrices
\[ \bbI_8+ 2 \begin{pmatrix}
    a_{11} \bbI_2 & 0  & A_2 & 0 \\
     0 & a_{11}^\sigma \bbI_2 & 0 & A_2^{\sigma} \\
     \bbJ_2^{-1}A_2^T \bbJ_2 & 0          &  a_{33} \bbI_2 & 0\\
     0 & (\bbJ_2^{-1}A_2^T \bbJ_2)^{\sigma} & 0 & a_{33}^{\sigma} \bbI_2 
\end{pmatrix}, \quad A_2=\begin{pmatrix}
    a_{13} & a_{14} \\
    a_{23} & a_{24}
\end{pmatrix}, \quad a_{ij}\in \Z_{2^2}/4 \Z_{2^2}.
\]
Since $m_{\sfV^3}(w)=\left (\begin{smallmatrix}
    0 & \bbI_4 \\ \bbI_4 & 0 
\end{smallmatrix}\right )\not \equiv \bbI_8 \bmod 2$, the group $\Aut(X,\lambda)$ cannot contain $w$, so it is contained in $\Aut(Y_3',\lambda_3')^2$. Therefore, 
$m_{\sfV^3}(\Aut(X,\lambda))$ is contained in the commutative group 
\[ \left \{\bbI_8+2 \diag(a_{11},a_{11}, a_{11}^\sigma, a_{11}^\sigma, a_{33},a_{33}, a_{33}^\sigma, a_{33}^\sigma):a_{11},a_{33}\in \Z_{2^2}/4 \Z_{2^2}\right \}\]
and we have 
\[ \Aut(X,\lambda)\subseteq \left ( 1+2 \Z_{2^2} +\Pi^3 \Mat_2(O_2) \right  )^2. \]
Recall that $O_p$ is the completion  of the a maximal order $O=\End(E)$ at $p$. Here we embed $\Z_{2^2} \embed \Mat_2(\Z_{2^2})$ diagonally and we write $\Mat_2(O_2)=\Mat_2(\Z_{2^2}) \oplus \Mat_2(\Z_{2^2}) \Pi$. As every torsion element $h=1+2a+C$ in $1+2 \Z_{2^2} +\Pi^3 \Mat_2(O_2)$ has order $1$ or $2$ (cf.~Corollary~\ref{cor:AutmV3M0}), we have 
\[ 1=(1+2a+C)^2 = 1+4a(a+1)+ 2C+2(aC+Ca) +C^2 \equiv 1+4a(a+1) \bmod \Pi^5. \] 
This implies that $a=0,1$ in $\Z_{2^2}/4 \Z_{2^2}$, so $h\in \pm 1+ \Pi^3 \Mat_2(O_2)$ and hence $h=\pm 1$. This shows 
$\Aut(X,\lambda)\subseteq (1+2 \Z_{2^2} +\Pi^3 \Mat_2(O_2))^2_{\rm tors}=\{\pm 1\}^2$. On the other hand, by Corollary~\ref{cor:AutmV3M0}, the subgroup $\{\pm 1\}\times \{\pm 1\}$ of $\Aut(Y_3',\lambda_3')^2$ preserves the lattice $M_0$ and is contained in $\Aut(X,\lambda)$. This proves that $\Aut(X,\lambda)=\{\pm 1\}^2$, as desired.  
\end{proof}

\providecommand{\bysame}{\leavevmode\hbox to3em{\hrulefill}\thinspace}

\end{document}